\documentclass[a4paper,11pt]{amsart}
\usepackage[english]{babel}
\usepackage[latin1]{inputenc}
\usepackage[T1]{fontenc}
\usepackage{amsmath}
\usepackage{amsfonts}
\usepackage{amssymb}
\usepackage{amsthm}
\usepackage{mathrsfs}
\usepackage{enumitem}
\usepackage{graphicx}
\usepackage{footnote}
\usepackage{hyperref}
\hypersetup{colorlinks=true,    linkcolor=blue,         citecolor=red,       filecolor=BrickRed,       urlcolor=darkgreen    }

\renewcommand{\geq}{\geqslant}
\renewcommand{\leq}{\leqslant}
\newtheorem{thm}{Theorem}
\newtheorem{defn}[thm]{Definition}
\newtheorem{rem}[thm]{Remark}
\newtheorem{cor}[thm]{Corollary}
\newtheorem{prop}[thm]{Proposition}
\newtheorem{lem}[thm]{Lemma}
\let\o=\omega
\let\s=\star
\let\ss=\bigstar
\let\dd=\bullet
\let\d=\circ

\usepackage{color}
\definecolor{darkgreen}{rgb}{0,0.4,0}
\definecolor{MyDarkBlue}{rgb}{0,0.08,0.50}
\definecolor{BrickRed}{rgb}{0.65,0.08,0}

\title[On the functions counting walks with small steps in the quarter plane]{On the functions counting walks with small steps in~the quarter plane}

\author{Irina Kurkova\and Kilian Raschel}

 \thanks{I.\ Kurkova: Laboratoire de Probabilit\'es et
        Mod\`eles Al\'eatoires, Universit\'e Pierre et Marie Curie,
        4 Place Jussieu, 75252 Paris Cedex 05, France. Email:~\url{irina.kourkova@upmc.fr}}

\thanks{K.\ Raschel: CNRS and Université de Tours,
        Faculté des Sciences et Techniques,
Parc de Grandmont,
37200 Tours, France. Email: \url{kilian.raschel@lmpt.univ-tours.fr}}

 \date{\today}

\begin{document}

\begin{abstract}
Models of spatially homogeneous walks in the quarter plane ${\bf Z}_+^{2}$ with steps taken from a subset $\mathcal{S}$ of the set of jumps to the eight nearest neighbors are considered. The generating function $(x,y,z)\mapsto Q(x,y;z)$ of the numbers $q(i,j;n)$ of such walks starting at the origin and ending at $(i,j) \in {\bf Z}_+^{2}$ after $n$ steps is studied. For all non-singular models of walks, the functions $x \mapsto Q(x,0;z)$ and $y\mapsto Q(0,y;z)$ are continued as multi-valued functions on ${\bf C}$ having infinitely many meromorphic branches, of which the set of poles is identified. The nature of these functions is derived from this result: namely, for all the $51$ walks which admit a certain infinite group of birational transformations of ${\bf C}^2$, the interval $]0,1/|\mathcal{S}|[$ of variation of $z$ splits into two dense subsets such that the functions $x \mapsto Q(x,0;z)$ and $y\mapsto Q(0,y;z)$ are shown to be holonomic for any $z$ from the one of them and non-holonomic for any $z$ from the other. This entails the non-holonomy of $(x,y,z)\mapsto Q(x,y;z)$, and therefore proves a conjecture of Bousquet-M\'elou and Mishna in \cite{BMM}.

\medskip

\medskip

\noindent{{\sc Keywords.} Walks in the quarter plane; counting generating function; holonomy; group of the walk; Riemann surface; elliptic functions; uniformization; universal covering}

\medskip

\medskip

{\footnotesize\noindent{{\sc AMS 2000 Subject Classification:} primary 05A15; secondary 30F10, 30D05}}
\end{abstract}

\maketitle

\section{Introduction and main results}
\label{Introduction}
\setcounter{equation}{0}

In the field of enumerative combinatorics,
 counting walks on the lattice ${\bf Z}^2$ is among the most classical topics.
  While counting problems have been largely resolved for unrestricted walks on ${\bf Z}^2$ and for walks staying in a half plane \cite{BMP},
  walks confined to the quarter plane ${\bf Z}_{+}^{2}$
  still pose considerable challenges.
  In recent years, much progress has been
  made for walks in the quarter plane with small steps,
  which means that the set $\mathcal{S}$ of possible steps is included in $\{-1,0,1\}^2\setminus \{(0,0)\}$; for examples, see Figures \ref{Ex} and \ref{ExExEx}. In \cite{BMM}, Bousquet-M\'{e}lou and Mishna constructed a thorough classification of these $2^8$ walks. After eliminating trivial cases and exploiting equivalences, they showed that $79$ inherently different walks remain to be studied. Let $q(i,j;n)$ denote the number of paths in ${\bf Z}_{+}^{2}$ having length
$n$, starting from $(0,0)$ and ending at $(i,j)$. Define the
counting function (CF) as
\begin{figure}[t]
\begin{center}
\begin{picture}(70.00,65.00)
\includegraphics{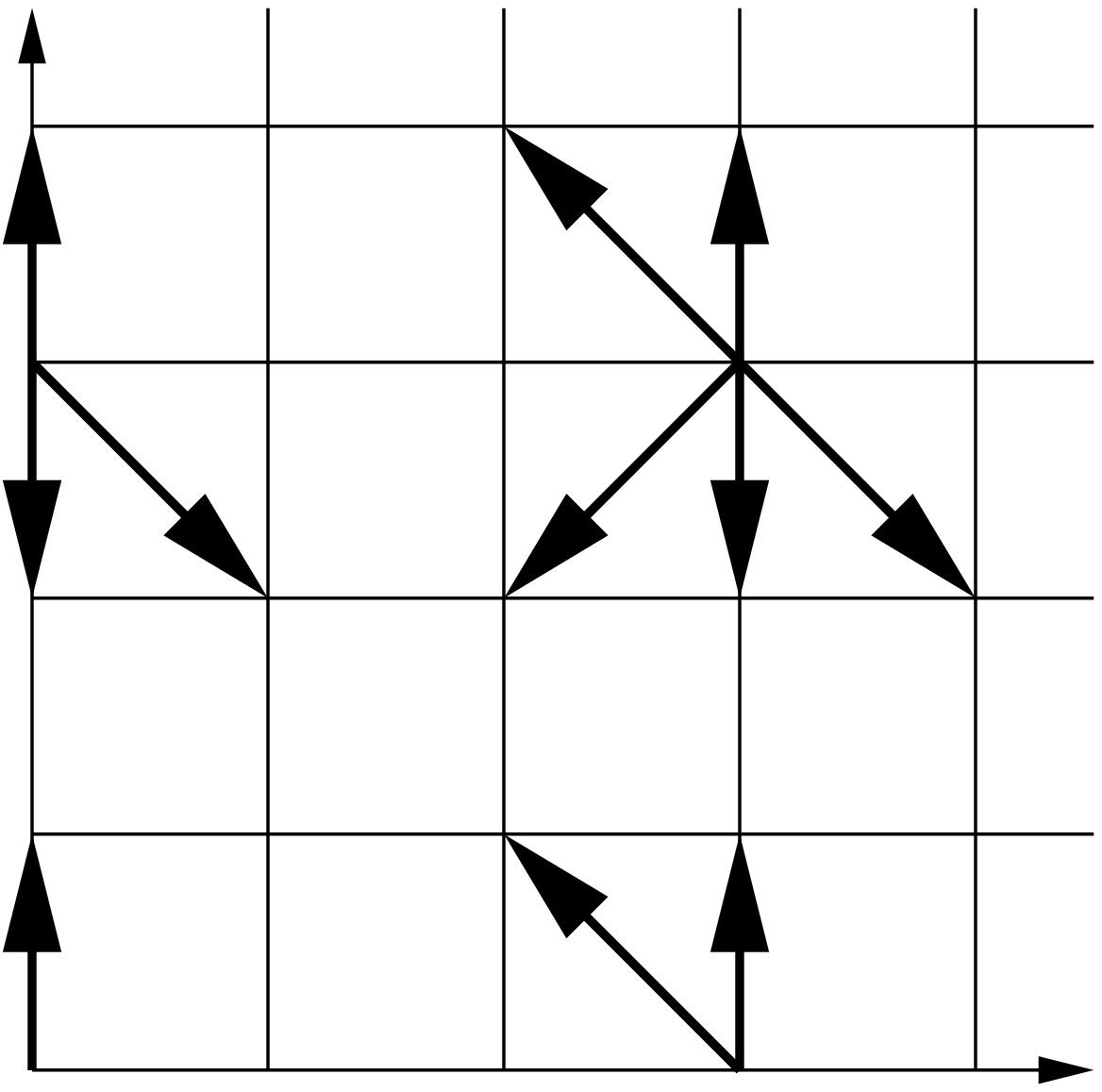}
\end{picture}
\end{center}
\caption{Example of model (with an infinite group) considered here---on the boundary, the jumps are the natural ones: those that would take the walk out ${\bf Z}_{+}^{2}$ are discarded}
\label{Ex}
\end{figure}
     \begin{equation*}
     \label{def_CGF}
          Q(x,y;z)=\sum_{i,j,n\geq 0} q(i,j;n)x^{i}y^{j}z^{n}.
     \end{equation*}
There are then two key challenges:
\begin{enumerate}[label=(\roman{*}),ref={\rm (\roman{*})}]
     \item \label{Challenge_1} Finding an explicit expression for $Q(x,y;z)$;
     \item \label{Challenge_2}
       Determining the nature of $Q(x,y;z)$: is it holonomic
       (i.e., see \cite[Appendix B.4]{FLAJ}, is the vector space
        over ${\bf C}(x,y,z)$---the field of rational functions in the three variables $x,y,z$---spanned by the set of all derivatives of $Q(x,y;z)$ finite
dimensional)? And in that event, is it algebraic, or even rational?
\end{enumerate}
The common approach to treat these problems is to start from a functional equation for the CF, which for the walks with small steps takes the form
(see~\cite{BMM})
     \begin{equation}
     \label{functional_equation}
          K(x,y;z)Q(x,y;z)=K(x,0;z)Q(x,0;z)+K(0,y;z)Q(0,y;z)-K(0,0;z) Q(0,0;z)
            -x y,
     \end{equation}
where
     \begin{equation}
     \label{def_kernel}
          K(x,y;z)=xyz[\textstyle\sum_{(i,j)\in\mathcal{S}}x^{i}y^{j}-1/z]
     \end{equation}
is called the {\em kernel of the walk}. This equation determines $Q(x,y;z)$ through the boundary functions $Q(x,0;z)$, $Q(0,y;z)$ and $Q(0,0;z)$.

Known results regarding both problems \ref{Challenge_1} and \ref{Challenge_2} highlight the notion of the \emph{group of the walk}, introduced by Malyshev \cite{MA,MAL,MALY}. This is the group
     \begin{equation}
     \label{group}
          \langle\xi,\eta\rangle
     \end{equation}
of birational transformations of ${\bf C}(x,y)$, generated by
     \begin{equation}
     \label{xietaf}
          \xi(x,y)= \Bigg(x,\frac{1}{y}\frac{\sum_{(i,-1)\in\mathcal{S}}x^{i}}
          {\sum_{(i,+1)\in\mathcal{S}}x^{i}}\Bigg),
          \qquad  \eta(x,y)=\Bigg(\frac{1}{x}\frac{\sum_{(-1,j)\in\mathcal{S}}y^{j}}
          {\sum_{(+1,j)\in\mathcal{S}}y^{j}},y\Bigg).
     \end{equation}
Each element of $\langle\xi,\eta\rangle$ leaves invariant the jump function $\sum_{(i,j)\in\mathcal{S}}x^{i}y^{j}$.
Further, $\xi^2=\eta^2={\rm Id}$, and $\langle \xi,\eta\rangle$ is a
dihedral group of order even and larger than or equal to four. It
turns out that $23$ of the $79$ walks have a finite group, while the
$56$ others admit an infinite group, see \cite{BMM}.

For $22$ of the $23$ models with finite group,
CFs $Q(x,0;z)$, $Q(0,y;z)$ and $Q(0,0;z)$---and hence $Q(x,y;z)$
 by \eqref{functional_equation}---have been computed
 in \cite{BMM} by means of certain (half-)orbit sums of
 the functional equation \eqref{functional_equation}.
 For the $23$rd model with finite group, known as Gessel's walk
 (see Figure \ref{ExExEx}), the CFs have been expressed by radicals in
  \cite{BK2} thanks to a guessing-proving method
  using computer calculations; they were also found in
  \cite{KRG} by solving some boundary value problems.
  For the $2$ walks with infinite group on the left in
   Figure \ref{The_five_singular_walks}, they have been obtained in \cite{MM2},
    by  exploiting a particular property shared by the $5$ models
      of  Figure \ref{The_five_singular_walks} commonly known
    as {\it singular walks}.
    Finally, in \cite{Ra}, the problem \ref{Challenge_1} was resolved for all $54$
    remaining walks---and in fact for all the $79$ models.
     For the $74$ non-singular walks, this was done via a unified approach:
     explicit integral representations were obtained for CFs $Q(x,0;z)$, $Q(0,y;z)$ and $Q(0,0;z)$
     in certain domains,
      by solving boundary value problems of Riemann-Carleman type.

     In this article we go further, since both functions  $x \mapsto Q(x,0;z)$ and $y \mapsto
     Q(0,y;z)$ are computed on the whole of ${\bf C}$
     as {\it multi-valued} functions  with
     infinitely many meromorphic branches,
     that are made explicit for all $z \in ]0, 1/|\mathcal{S}|[$.
     This result gives not only the most complete continuation of these CFs
     on their complex
     planes along all paths,
      but also permits to establish the nature of these functions, i.e., to
     solve Problem \ref{Challenge_2}.

Problem \ref{Challenge_2} is actually resolved for only $28$ of the
$79$ walks. All $23$ finite group models admit a holonomic CF.
Indeed, the nature of $Q(x,y;z)$ was determined in \cite{BMM} for
$22$ of these walks: $19$ walks turn out to have a holonomic but
non-algebraic CF, while for $3$ walks $Q(x,y;z)$ is algebraic. As
for the $23$rd---again, Gessel's model---, the CF is algebraic
\cite{BK2}. Alternative proofs for the nature of the (bivariate) CF
for these $23$ walks were given in \cite{FR}. For the remaining $56$
walks with an infinite group, not much is known: in \cite{MM2} it
was shown that for $2$ singular walks (namely, the $2$ ones on the
left in Figure \ref{The_five_singular_walks}), the function
$z\mapsto Q(1,1;z)$ has infinitely many poles and, as a consequence
\cite[Appendix B.4]{FLAJ}, is non-holonomic. Accordingly \cite[Appendix B.4]{FLAJ}, the
trivariate function $Q(x,y;z)$ is non-holonomic as well. It is
reasonable to expect that the same approach would lead to the
non-holonomy of all $5$ singular walks, see \cite{MMM,MM2}.
\begin{figure}[t]
\begin{center}
\begin{picture}(427.00,68.00)
\includegraphics{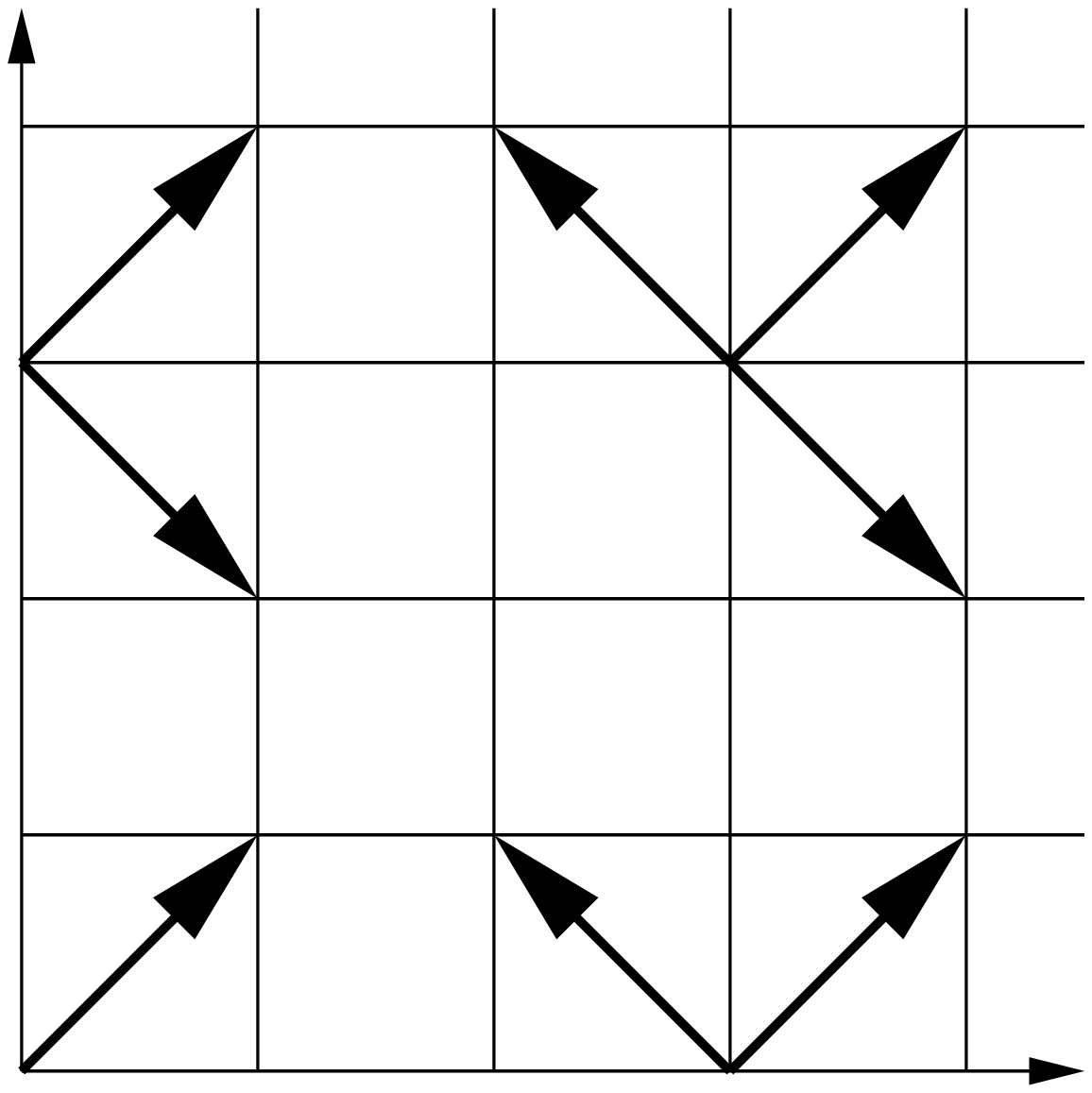}
\hspace{29mm}
\includegraphics{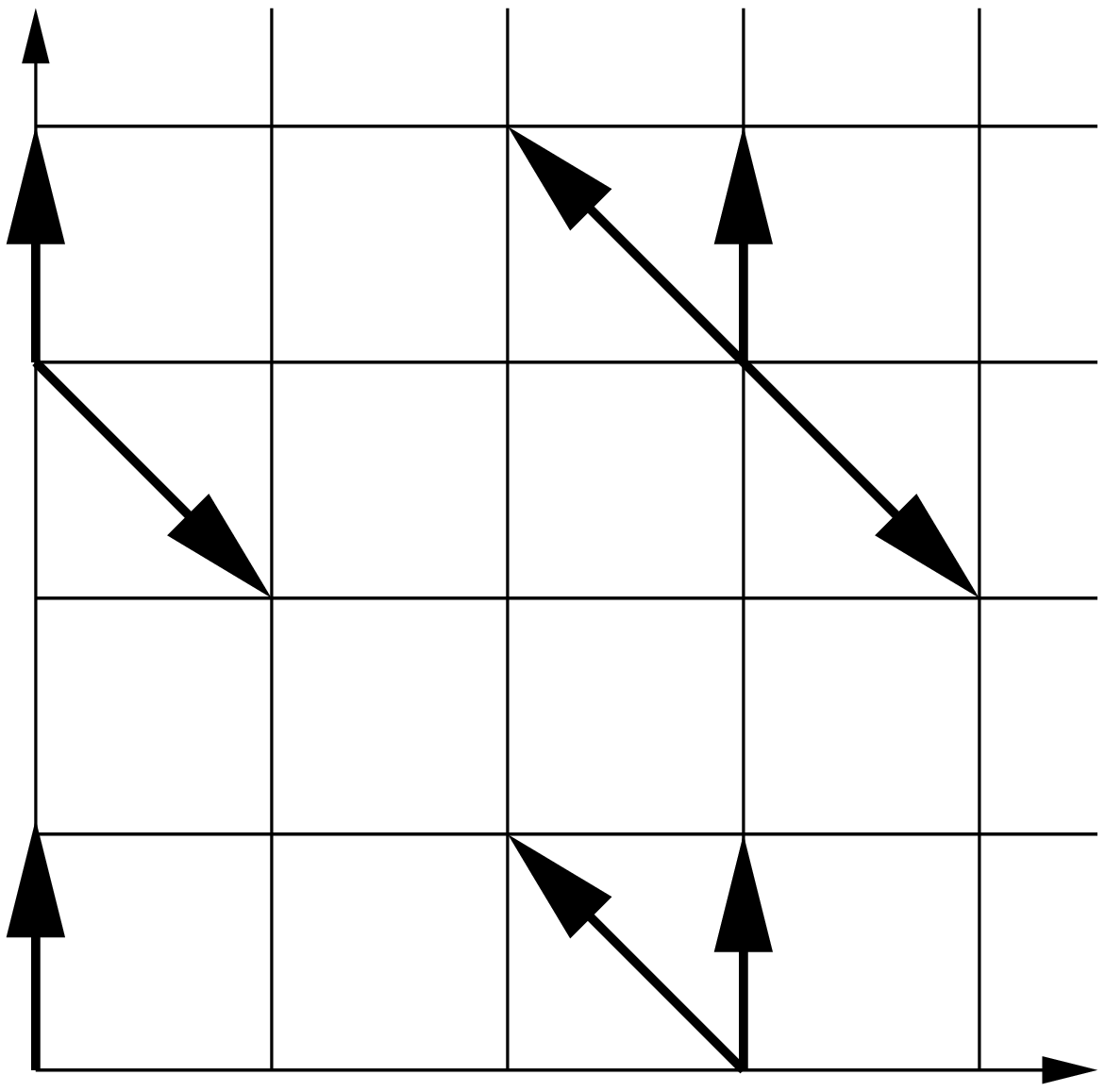}
\hspace{29mm}
\includegraphics{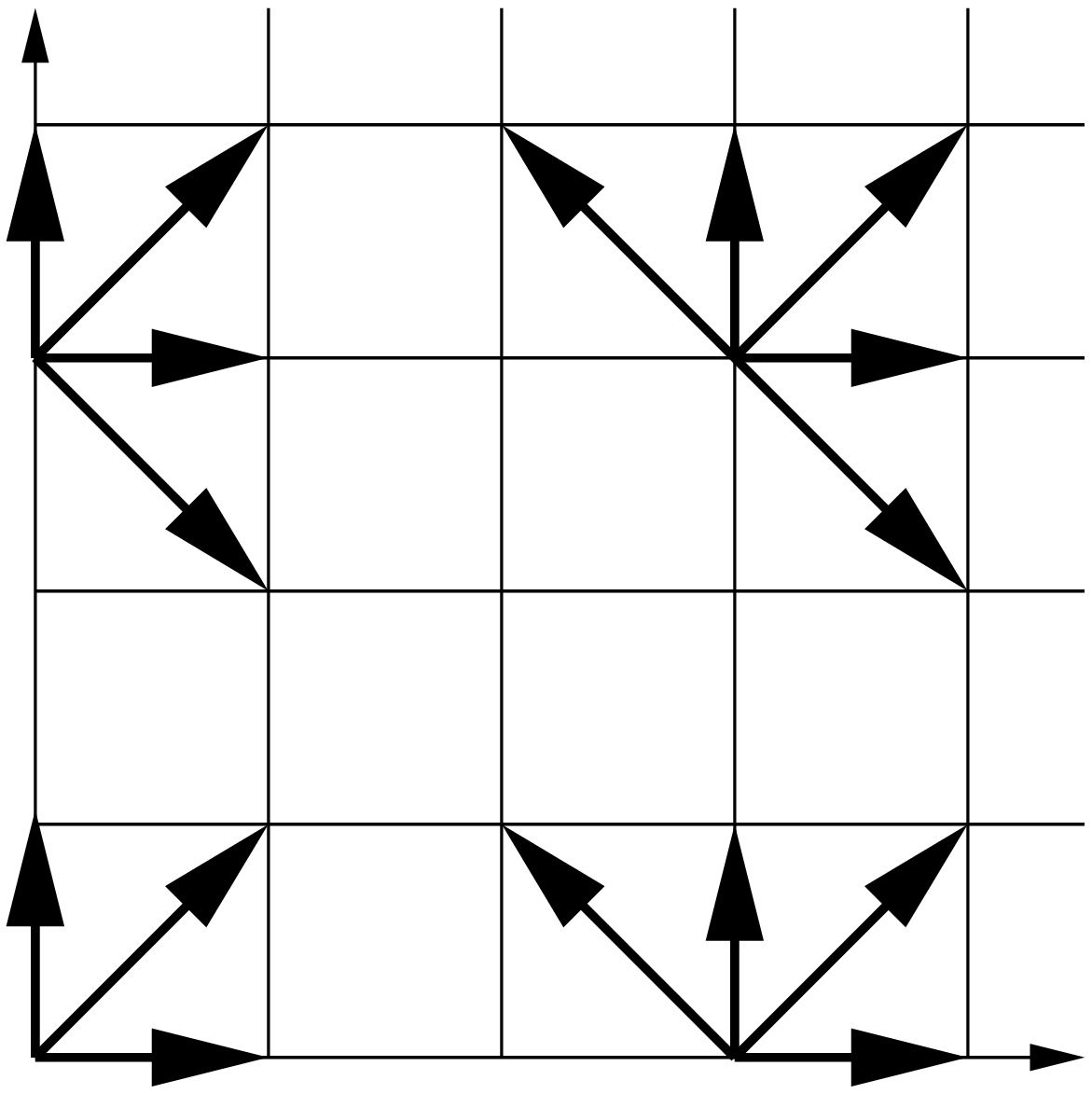}
\hspace{29mm}
\includegraphics{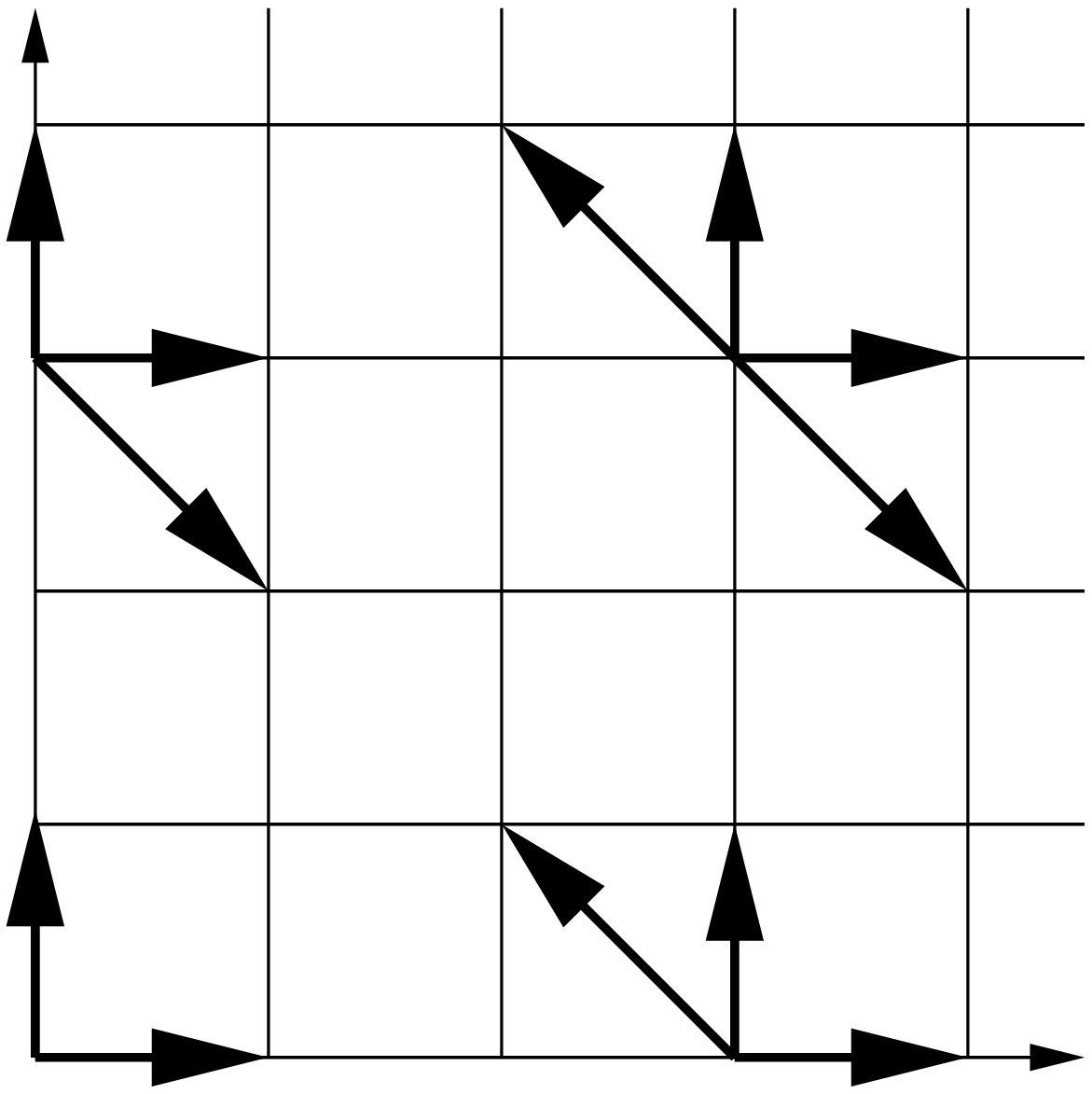}
\hspace{29mm}
\includegraphics{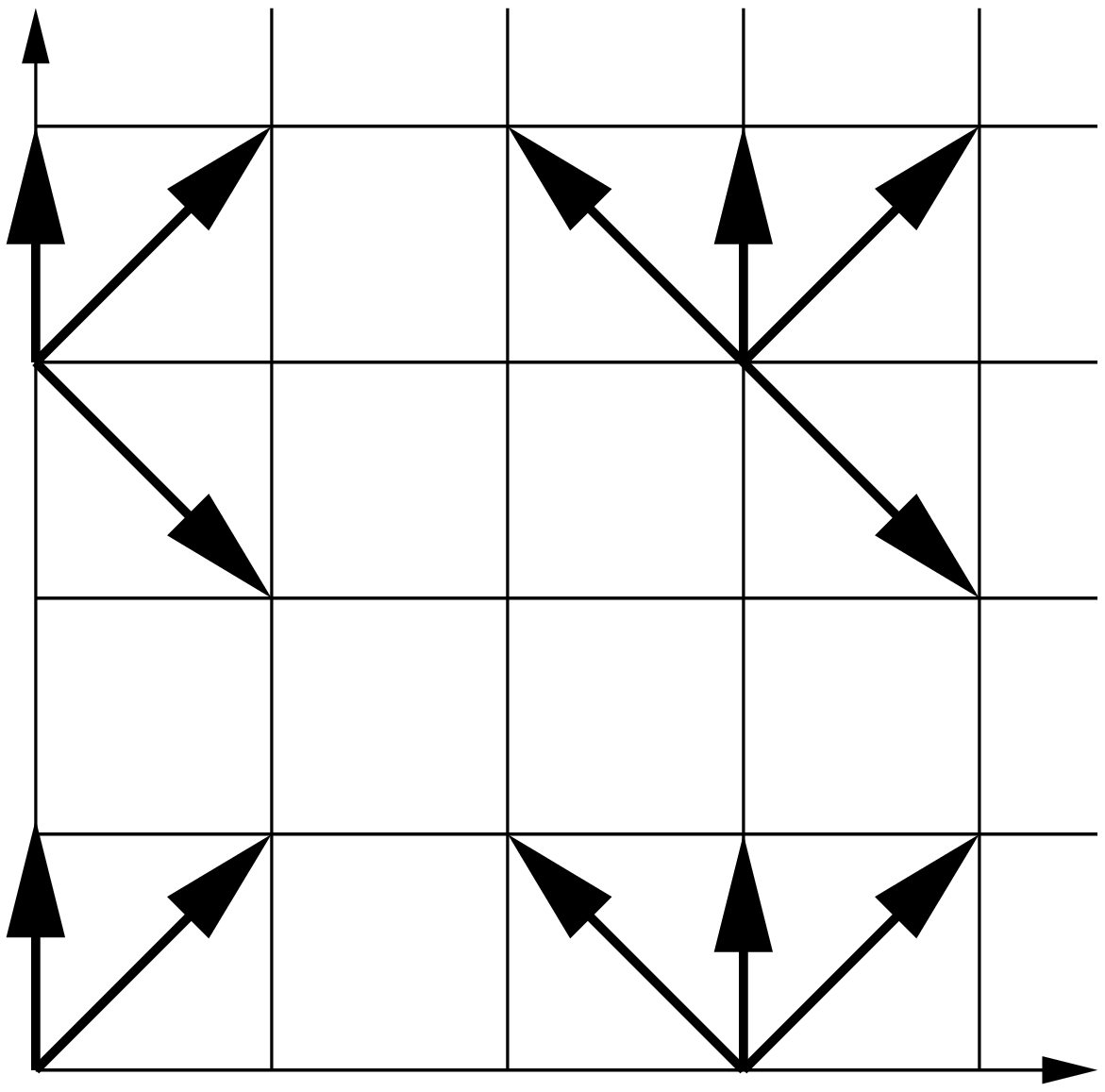}
\end{picture}
\end{center}
\caption{The $5$ singular walks in the classification of~\cite{BMM}}
\label{The_five_singular_walks}
\end{figure}
As for the $51$ non-singular walks with infinite group (all of them
are pictured on Figure~\ref{Allcases}), Bousquet-M\'elou and Mishna
\cite{BMM} conjectured that they also have a non-holonomic CF.
{In this article
  we  prove the following theorem.}

\begin{thm}
\label{main_theorem} For any of the $51$ non-singular walks with
infinite group (\ref{group}), the set $]0, 1/|\mathcal{S}|[$ splits
 into subsets ${\mathcal H}$
 and $]0, 1/|\mathcal{S}|[\setminus {\mathcal H}$ that are both dense in $]0, 1/|\mathcal{S}|[$
   and such that:
\begin{enumerate}[label={\rm (\roman{*})},ref={\rm (\roman{*})}]
\item \label{firstpointthm}
  $x \mapsto Q(x,0;z)$ and $y \mapsto Q(0,y;z)$ are holonomic
for any $z \in {\mathcal H}$;
\item \label{secondpointthm}
$x \mapsto Q(x,0;z)$ and $y \mapsto Q(0,y;z)$ are non-holonomic for
any $z \in ]0, 1/|\mathcal{S}|[\setminus {\mathcal H}$.
\end{enumerate}
\end{thm}

{
Theorem~\ref{main_theorem} \ref{secondpointthm} immediately entails Bousquet-M\'elou and Mishna's conjecture: the~trivariate function $(x,y,z)\mapsto Q(x,y;z)$ is non-holonomic since
the holonomy is stable by specialization of a variable \cite[Appendix B.4]{FLAJ}. Further,
Theorem~\ref{main_theorem} \ref{firstpointthm} goes beyond it: it
suggests that $Q(x,y;z)$, although being non-holonomic, still stays accessible for further
 analysis when $z \in {\mathcal H}$, namely by the use of methods developed in \cite[Chapter 4]{FIM},
 see Remark~\ref{Remark2}. This important set ${\mathcal H}$ will be
 characterized in two different ways, see Corollary~\ref{cor-rnc} and
Remark~\ref{Remark2} below.}

The proof of Theorem \ref{main_theorem}   we shall do here is based
on the above-mentioned construction of the CF $x \mapsto Q(x,0;z)$
(resp.\ $y \mapsto Q(0,y;z)$) as a multi-valued function, that must
now be slightly more detailed.
 First, we prove in this article that for any $z \in
]0, 1/|\mathcal{S}|[$,
  the integral expression of $x\mapsto Q(x,0;z)$
  given in \cite{Ra} in a certain domain of ${\bf C}$
  admits a \emph{direct holomorphic continuation}
  on ${\bf C}\setminus[x_3(z),x_4(z)]$.
  Points $x_3(z),x_4(z)$ are among  four {\it branch points}
   $x_1(z), x_2(z), x_3(z), x_4(z)$ of the two-valued algebraic function
   $x\mapsto Y(x;z)$ defined via the kernel \eqref{def_kernel}
    by the equation $K(x,Y(x;z);z)=0$.
    These branch points are roots of the discriminant \eqref{d_d_tilde}
    of the latter equation, which is of the second order.
     We refer to Section \ref{Riemann_surface}
     for the numbering and for some properties of these branch points.
     We prove next that function $x\mapsto Q(x,0;z)$
     does not admit a direct meromorphic
      continuation on any open domain 
    containing
       the segment $[x_3(z),x_4(z)]$,
but admits  a \emph{meromorphic continuation along any path}
going once through $[x_3(z),x_4(z)]$.
 This way, we obtain a second (and different)
  branch of the function,
  which admits a direct meromorphic continuation on the whole
  cut plane ${\bf C}\setminus ([x_1(z),x_2(z)]\cup[x_3(z),x_4(z)])$.
  Next, if the function  $x\mapsto Q(x,0;z)$
   is continued along a path in ${\bf C}\setminus [x_1(z),x_2(z)]$
    crossing once again $[x_3(z),x_4(z)]$,
    we come across its first branch.
     But its continuation along a path in
      ${\bf C}\setminus [x_3(z),x_4(z)]$
      crossing once $[x_1(z),x_2(z)]$
       leads to a third branch of this CF,
        which is meromorphic on ${\bf C}\setminus ([x_1(z),x_2(z)]\cup [x_3(z),x_4(z)])$.
        Making loops through $[x_3(z),x_4(z)]$ and $[x_1(z),x_2(z)]$,
         successively, we construct $x \mapsto Q(x,0;z)$ as a multi-valued
         meromorphic function on ${\bf C}$ with branch points $x_1(z),x_2(z), x_3(z), x_4(z)$,
         and with (generically) infinitely many branches.
          The analogous construction is valid for $y \mapsto Q(0,y;z)$.

In order to prove Theorem~\ref{main_theorem} \ref{secondpointthm},
 we then show that for any of the $51$ non-singular walks with infinite group
   \eqref{group},  for any $z \in ]0,
1/|\mathcal{S}|[\setminus \mathcal{H}$, {\it the set formed by the poles of
all branches of $x \mapsto Q(x,0;z)$ (resp.\ $y \mapsto Q(0,y;z)$)
is infinite}---and even dense in certain curves, to be specified in
Section~\ref{Section_infinite_group} (see Figure~\ref{CC} for their
pictures).
This is not compatible with holonomy. Indeed, all branches of a
holonomic one-dimensional function must verify the {\it same}
 linear differential equation with polynomial coefficients.
 In particular, the poles of all branches are among the zeros
 of these polynomials, and hence they must be in a finite number.

The rest of our paper is organized as follows. In
Section~\ref{Riemann_surface} we construct the Riemann surface ${\bf
T}$ of genus 1 of the two-valued algebraic functions $X(y;z)$ and
$Y(x;z)$ defined by
     \begin{equation*}
          K(X(y;z),y;z)=0,\qquad K(x,Y(x;z);z)=0.
     \end{equation*}
In Section~\ref{uc} we introduce and study
 the universal covering of ${\bf T}$. It can be viewed as the complex plane ${\bf C}$ split
  into infinitely many parallelograms with edges $\o_1(z) \in i {\bf R}$ and
 $\o_2(z) \in {\bf R}$ that are uniformization periods.
 These periods as well as a new important period
  $\o_3(z)$  are made explicit in (\ref{expression_omega_1_2}) and
  (\ref{expression_omega_3}).
 In Section~\ref{Lifting_uc} we lift CFs $Q(x,0;z)$
 and $Q(0,y;z)$ to some domain of ${\bf T}$,
  and then to a domain on its universal covering.
    In Section~\ref{meroun},
   using a proper lifting of the automorphisms $\xi$ and $\eta$
    defined in \eqref{xietaf} as well as the independence
    of $K(x,0;z)Q(x,0;z)$ and $K(0,y;z)Q(0,y;z)$ w.r.t.\ $y$ and $x$,
    respectively, we continue these functions meromorphically
    on the whole of the universal covering.
    All this procedure has been first carried out by Malyshev
    in the seventies \cite{MA,MAL,MALY},
    at that time to study the stationary probability generating
 functions for random walks with small steps in the quarter
 plane ${\bf Z}_+^2$.
 It is presented in~\cite[Chapter 3]{FIM} for the case of ergodic
 random walks in ${\bf Z}_+^2$, and applies directly for
 our $Q(x,0;1/|\mathcal{S}|)$ and $Q(0,y;1/|\mathcal{S}|)$ if the drift vector
          $(\sum_{(i,j)\in \mathcal{S}}i, \sum_{(i,j)\in
          \mathcal{S}}j)$
has not two positive coordinates. In Sections \ref{uc},
\ref{Lifting_uc} and \ref{meroun} we carry out this procedure for
all $z\in ]0,1/|\mathcal{S}|[$ and all  non-singular walks,
independently of the drift. Then, going back from the universal
covering to the complex plane allows us in
Subsection~\ref{subsecbra} to continue
 $x\mapsto Q(x,0;z)$ and $y\mapsto Q(0,y;z)$ as multi-valued meromorphic functions
 with infinitely many branches.

  For given $z \in ]0,
 1/|\mathcal{S}|[$, the rationality or irrationality of the ratio
 $\o_2(z)/\o_3(z)$
 of the uniformization periods is crucial for the
 nature of $x\mapsto Q(x,0;z)$ and $y\mapsto Q(0,y;z)$.
 Namely, Theorem \ref{ratho} of Subsection~\ref{subsecra}
   proves that if $\o_2(z)/\o_3(z)$ is rational,
 these functions are holonomic.

 For $23$ models of walks with finite
 group $\langle\xi, \eta\rangle$, the ratio $\o_2(z)/\o_3(z)$ turns out to be
 rational and independent of $z$, see Lemma~\ref{lemma_omega_2_3}  below,
 that implies immediately the holonomy of the generating functions.
 In Section~\ref{section_holonomy}
 we gather further results of our approach for the models
 with finite group concerning the set of branches of the generating
 functions and their nature. In particular,
 we recover most of the results of \cite{BK2,BMM,FR,MM2}.

   Section~\ref{Section_infinite_group} is devoted to $51$ models with
   infinite group $\langle\xi,
    \eta\rangle$. For all of them, the
 sets $\mathcal{H}=\{z\in]0, 1/|\mathcal{S}|[: \o_2(z)/\o_3(z) \hbox{ is
 rational}\}$ and $]0, 1/|\mathcal{S}|[\setminus \mathcal{H}=
  \{z \in ]0, 1/|\mathcal{S}|[: \o_2(z)/\o_3(z) \hbox{ is
 irrational}\}$ are proved to be dense in
  $]0, 1/|\mathcal{S}|[$,  see Proposition~\ref{proir}.
{These sets can be also characterized as those
where the group $\langle \xi, \eta \rangle$  restricted to the
curve $\{(x,y) : K(x,y;z)=0\}$ is finite and infinite, respectively, see Remark \ref{Remark1}.}
  By Theorem~\ref{ratho} mentioned above,
$x\mapsto Q(x,0;z)$ and $y\mapsto Q(0,y;z)$ are holonomic for any $z
\in \mathcal{H}$, that proves Theorem \ref{main_theorem} \ref{firstpointthm}.
     In Subsections \ref{subsection71}, \ref{subsection72} and \ref{subsection73},
    we  analyze in detail the branches of
    $x\mapsto Q(x,0;z)$ and $y\mapsto Q(0,y;z)$ for any $z \in
   ]0, 1/|\mathcal{S}|[\setminus \mathcal{H}$
     and prove the following facts (see Theorem \ref{main_tt}):
     \begin{enumerate}[label={\rm (\roman{*})},ref={\rm (\roman{*})}]
          \item The only singularities of the first (main) branches
          of $x\mapsto Q(x,0;z)$ and $y\mapsto Q(0,y;z)$ are two branch points
          $x_3(z),x_4(z)$ and $y_3(z),y_4(z)$, respectively;
          \item All (other) branches have only a finite number of poles;
          \item \label{trois} The set of poles of all these branches is
            infinite for each of these functions, and
            is dense on certain curves; these curves are specified
             in Section~\ref{Section_infinite_group}, and in particular
             are pictured on Figure \ref{CC}
             for all $51$ walks given on Figure~\ref{Allcases};
          \item Poles of branches out of these curves may be only at zeros
             of $x\mapsto K(x,0;z)$ or $y \mapsto K(0,y;z)$, respectively.
     \end{enumerate}
  It follows from \ref{trois} that  $x\mapsto Q(x,0;z)$ and $y\mapsto Q(0,y;z)$
  are non-holonomic for any $z \in ]0, 1/|\mathcal{S}|[\setminus
  \mathcal{H}$.

\section{Riemann surface ${\bf T}$}
\label{Riemann_surface}
\setcounter{equation}{0}

In the sequel we suppose that $z\in ]0,1/|\mathcal{S}|[$, and we drop the dependence of the different quantities w.r.t.\ $z$.

\subsection{Kernel $K(x,y)$}
\label{subsection:kernel}
The kernel $K(x,y)$ defined in \eqref{def_kernel} can be written as
     \begin{equation}
     \label{kernel_pt}
          xyz[\textstyle\sum_{(i,j)\in\mathcal{S}}x^{i}y^{j}-1/z]=\widetilde{a}
          (y)x^{2}+\widetilde{b}(y)x+\widetilde{c}(y)=a(x)y^{2}+b(x)y+c(x),
     \end{equation}
where
     \begin{align*}
          \widetilde{a}(y)&=zy\textstyle \sum_{(+1,j)\in\mathcal{S}}y^{j},
          &\widetilde{b}(y)=-y+&zy\textstyle\sum_{(0,j)\in\mathcal{S}}y^{j},&\widetilde{c}(y)&=zy\textstyle\sum_{(-1,j)\in\mathcal{S}}y^{j}, \\
          a(x)&=\textstyle zx\sum_{(i,+1)\in\mathcal{S}}x^{i},
          &b(x)=-x+&zx\textstyle \sum_{(i,0)\in\mathcal{S}}x^{i},&c(x)&=zx\textstyle\sum_{(i,-1)\in\mathcal{S}}x^{i}.
     \end{align*}
With these notations we define
     \begin{equation}
     \label{d_d_tilde}
          \widetilde{d}(y)=\widetilde{b}(y)^{2}-4\widetilde{a}(y)\widetilde{c}(y),
          \qquad d(x)=b(x)^{2}-4a(x)c(x).
     \end{equation}
If the walk is non-singular, then for any $z\in]0,1/|\mathcal{S}|[$, the polynomial $\widetilde{d}$ (resp.\ $d$) has three or four roots, that we call $y_{\ell}$ (resp.\ $x_{\ell}$). They are such that $|y_{1}|<y_{2}<1<y_{3}<|y_{4}|$ (resp.\ $|x_{1}|<x_{2}<1<x_{3}<|x_{4}|$), with $y_4=\infty$ (resp.\ $x_4=\infty$) if $\widetilde{d}$ (resp.\ ${d}$) has order three: the arguments given in \cite[Part 2.3]{FIM} for the case $z=1/|\mathcal{S}|$ indeed also apply for other values of $z$.

Now we notice that the kernel \eqref{def_kernel} vanishes if and
only if $[\widetilde{b}(y)+
2\widetilde{a}(y)x]^{2}=\widetilde{d}(y)$ or $[b(x)+2a(x)y
]^{2}=d(x)$. Consequently \cite{JS}, the algebraic functions
$X(y)$ and $Y(x)$ defined by
     \begin{equation}
     \label{def_algebraic_functions}
          \textstyle \sum_{(i,j)\in\mathcal{S}}X(y)^{i}y^{j}-1/z=0,\qquad \textstyle\sum_{(i,j)\in\mathcal{S}}x^{i}Y(x)^{j}-1/z=0
     \end{equation}
have two branches, meromorphic on the cut planes ${\bf C}\setminus ([y_{1},y_{2}]\cup[y_{3},y_{4}])$ and ${\bf C}\setminus([x_{1},x_{2}]\cup [x_{3},x_{4}])$,
respectively---note that if $y_4<0$, $[y_3,y_4]$ stands for $[y_3,\infty[\cup\{\infty\}\cup]-\infty,y_4]$; the same holds for $[x_3,x_4]$.

We fix the notations of the two branches of the algebraic functions
$X(y)$ and $Y(x)$ by setting
     \begin{equation}
     \label{def_XX}
          X_{0}(y)=\frac{-\widetilde{b}(y)+ \widetilde{d}(y)^{1/2}}{2\widetilde{a}(y)},
          \qquad X_{1}(y)=\frac{-\widetilde{b}(y)-
          \widetilde{d}(y)^{1/2}}{2\widetilde{a}(y)},
     \end{equation}
as well as
     \begin{equation}
     \label{def_YY}
          Y_{0}(x)=\frac{-b(x)+d(x)^{1/2}}{2a(x)},
          \qquad Y_{1}(x)=\frac{-b(x)-d(x)^{1/2}}{2a(x)}.
     \end{equation}
The following straightforward result holds.
     \begin{lem}
     \label{Properties_X_Y_0}
          For all $y\in{\bf C}$, we have $|X_{0}(y)|
          \leq |X_{1}(y)|$. Likewise, for all $x\in{\bf C}$, we have $|Y_{0}(x)|
          \leq |Y_{1}(x)|$.
     \end{lem}
\begin{proof}
The arguments (via the maximum modulus principle \cite{JS}) given in \cite[Part 5.3]{FIM} for $z=1/|\mathcal{S}|$ also work for $z\in]0,1/|\mathcal{S}|[$.
\end{proof}

\subsection{Riemann surface ${\bf T}$}

We now construct the Riemann surface ${\bf T}$ of the algebraic function $Y(x)$ introduced in \eqref{def_algebraic_functions}. For this purpose we take two Riemann spheres ${\bf C}\cup\{\infty\}$, say ${\bf S}_x^1$ and ${\bf S}_x^2$, cut along the segments $[x_1, x_2]$ and $[x_3, x_4]$, and we glue them together along the borders of these cuts, joining the lower border of the segment $[x_1, x_2]$ (resp.\ $[x_3, x_4]$) on ${\bf S}_x^1$ to the upper border of the same segment on ${\bf S}_x^2$ and vice versa, see Figure \ref{Construction_Riemann}. The resulting surface ${\bf T}$ is homeomorphic to a torus (i.e., a compact Riemann surface of genus $1$) and is projected on the Riemann sphere ${\bf S}$ by a canonical covering map $h_x: {\bf T} \rightarrow {\bf S}$.
\begin{figure}[!t]
\begin{center}
\begin{picture}(00.00,820.00)
\hspace{-102mm}
\includegraphics{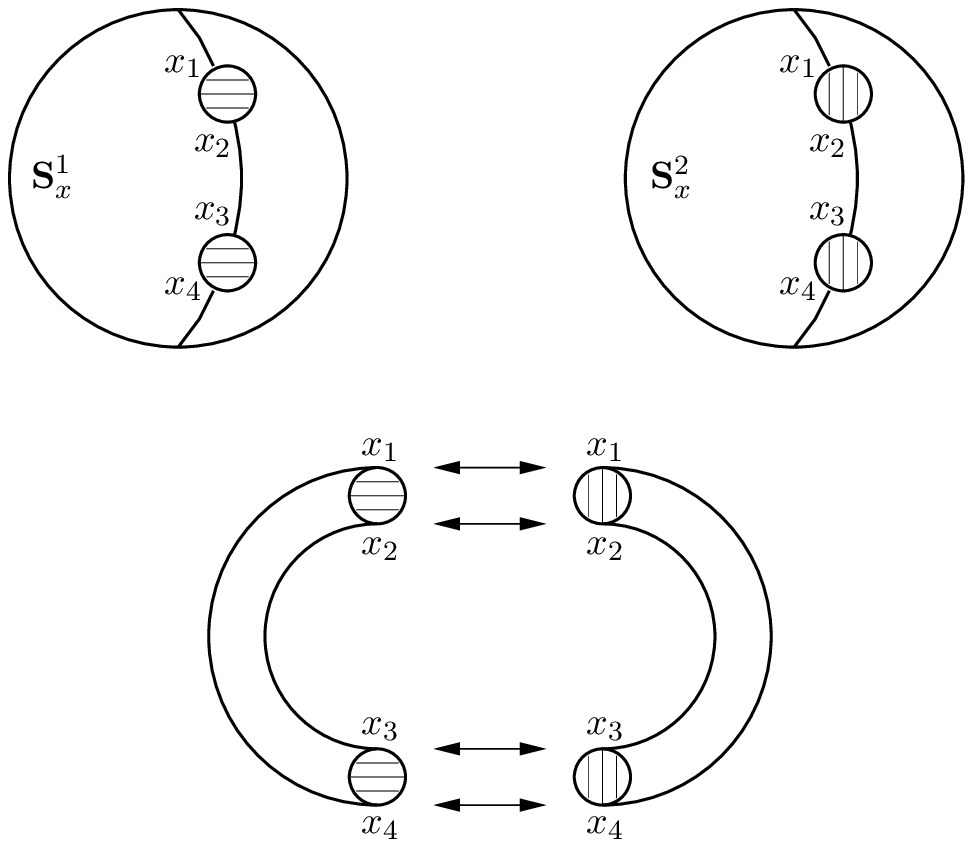}
\end{picture}
\end{center}
\vspace{-200mm}
\caption{Construction of the Riemann surface}
\label{Construction_Riemann}
\end{figure}


In a standard way, we can lift the function $Y(x)$ to ${\bf T}$,
 by setting $Y(s)=Y_\ell(h_x(s))$ if $s\in {\bf S}_x^\ell\subset {\bf
 T}$,
 $\ell\in\{1,2\}$.
Thus, $Y(s)$ is single-valued and continuous on ${\bf T}$. Furthermore, $K(h_x(s), Y(s))= 0$ for any $s\in{\bf T}$. For this reason, we call ${\bf T}$ the Riemann surface of $Y(x)$.

In a similar fashion, one constructs the Riemann surface of the function $X(y)$, by gluing together two copies ${\bf S}_y^1$ and ${\bf S}_y^2$ of the sphere ${\bf S}$ along the segments $[y_1, y_2]$ and $[y_3, y_4]$. It is again homeomorphic to a torus.

Since the Riemann surfaces of $X(y)$ and $Y(x)$ are equivalent, we can work on a {single} Riemann surface ${\bf T}$, but with two different covering maps $h_x,h_y: {\bf T} \rightarrow {\bf S}$. Then, for $s\in{\bf T}$, we set $x(s)= h_x(s)$ and $y(s)= h_y(s)$, and we will often represent a point $s\in {\bf T}$ by the pair of its {\it coordinates} $(x(s), y(s))$. These coordinates are of course not independent, because the equation $K(x(s), y(s))= 0$ is valid for any $s\in{\bf T}$.

\subsection{Real points of ${\bf T}$}
Let us identify the set $\Phi$ of real points of ${\bf T}$, that are
the points $s\in{\bf T}$ where $x(s)$ and $y(s)$ are both real or
equal to infinity. Note that for $y$ real, $X(y)$ is real if
$y\in[y_4,y_1]$ or $y\in[y_2,y_3]$, and complex if $y\in]y_1,y_2[$
or $y\in]y_3,y_4[$, see \eqref{d_d_tilde}. Likewise, for real values of $x$, $Y(x)$ is real
if $x\in[x_4,x_1]$ or $x\in[x_2,x_3]$, and complex if
$x\in]x_1,x_2[$ or $x\in]x_3,x_4[$.
The set $\Phi$ therefore consists of two non-intersecting closed analytic curves $\Phi_0$ and $\Phi_1$, equal to (see Figure \ref{Real_points})
\begin{equation*}
     \Phi_0 = \{s\in{\bf T}: x(s)\in[x_2,x_3]\}=\{s\in{\bf T}: y(s)\in[y_2,y_3]\}
\end{equation*}
and
\begin{equation*}
     \Phi_1=\{s\in{\bf T}: x(s)\in[x_4,x_1]\}=\{s\in{\bf T}: y(s)\in[y_4,y_1]\},
\end{equation*}
and homologically equivalent to a basic cycle on ${\bf T}$---note, however, that the equivalence class containing $\Phi_0$ and $\Phi_1$ is disjoint from that containing the cycle $h_x^{-1}(\{x\in{\bf C}: |x|=1\}$).

\begin{figure}[t]
\begin{center}
\begin{picture}(0.00,715.00)
\hspace{-82mm}
\includegraphics{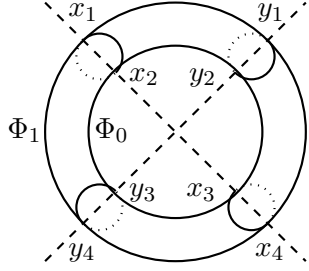}
\end{picture}
\end{center}
\vspace{-217mm}
\caption{Location of the branch points and of the cycles $\Phi_0$ and $\Phi_1$ on the Riemann surface ${\bf T}$}
\label{Real_points}
\end{figure}

\subsection{Galois automorphisms $\xi,\eta$}
We continue Section \ref{Riemann_surface} by introducing two Galois~automorphisms.
Define first, for $\ell\in\{1,2\}$, the incised spheres
     \begin{equation*}
          {\widehat{\bf S}}^\ell_x={\bf S}^\ell_x\setminus ([x_1,x_2]\cup[x_3,x_4]),
          \qquad {\widehat{\bf S}}^\ell_y={\bf S}^\ell_y\setminus ([y_1,y_2]\cup[y_3,y_4]).
     \end{equation*}
For any $s\in {\bf T}$ such that $x(s)$ is not equal to a branch point $x_\ell$, {there is
a unique $s'\neq s\in {\bf T}$ such that $x(s)=x(s')$.}
 Furthermore, if $s\in{\widehat{\bf S}}_x^1$ then $s'\in{\widehat{\bf S}}_x^2$ and vice versa. On the other hand, whenever $x(s)$ is one of the branch points $x_\ell$, $s=s'$. Also, since $K(x(s), y(s))=0$, $y(s)$ and $y(s')$ give the two values of function $Y(x)$ at $x=x(s)=x(s')$. By Vieta's theorem and \eqref{kernel_pt}, $y(s)y(s')=c(x(s))/ a(x(s))$.

Similarly, for any $s\in {\bf T}$ such that $y(s)$ is different from the branch points $y_\ell$, there exists a unique {$s''\neq s \in {\bf T}$} such that $y(s)=y(s'')$. If $s \in{\widehat{\bf S}}_y^1$ then $s'' \in{\widehat{\bf S}}_y^2$ and vice versa. On the other hand, if $y(s)$ is one of the branch points $y_\ell$, we have $s=s''$. Moreover, since $K(x(s), y(s))=0$, $x(s)$ and $x(s'')$ the two values of function $X(y)$ at $y(s)=y(s'')$. Again, by Vieta's theorem and \eqref{kernel_pt}, $x(s)x(s'')= \widetilde c(y(s))/\widetilde a(y(s)) $.

Define now the mappings $\xi : {\bf T}\rightarrow{\bf T}$ and $\eta :{\bf T}\rightarrow {\bf T}$ by
     \begin{equation}
     \label{xie}
          \left\{\begin{array}{lll}\xi s=s'&\mbox{if}&\hspace{-2mm}\phantom{y}x(s)=x(s'),\\
          \eta s=s''&\mbox{if} &\hspace{-2mm}\phantom{x}y(s)=y(s'').
          \end{array}\right.
  \end{equation}
Following \cite{MA,MAL,MALY}, we call them {\it Galois
automorphisms} of ${\bf T}$. Then $\xi^2=\eta^2={\rm Id}$, and
     \begin{equation}
     \label{bue}
          y(\xi s)=\frac{c(x(s))}{a(x(s))}\frac{1}{y(s)},
          \qquad
          x(\eta s)=\frac{\widetilde c(y(s))}{ \widetilde a(y(s))}\frac{1}{x(s)}.
     \end{equation}
Any $s\in{\bf T}$ such that $x(s)=x_\ell$ (resp.\ $y(s)=y_\ell$) is a fixed point for $\xi$ (resp.\ $\eta$).
To {illustrate and to} get some more intuition, it
is helpful to draw {on Figure \ref{Real_points}}
the straight line through the pair of points of $\Phi_0$ where
$x(s)=x_2$ and $x_3$ (resp.\ $y(s)=y_2$ and $y_3$); then points $s$
and $\xi s$ (resp.\ $s$ and $\eta s$) {can be
drawn} symmetric about this straight line.

\subsection{The Riemann surface ${\bf T}$ viewed as a parallelogram whose opposed edges are identified}
Like any compact Riemann surface of genus $1$, ${\bf T}$ is isomorphic to a certain quotient space
     \begin{equation}
     \label{iso_par}
          {\bf C}/(\omega_1{\bf Z}+\omega_2{\bf Z}),
     \end{equation}
where $\omega_1,\omega_2$ are complex numbers linearly independent on ${\bf R}$, see \cite{JS}.
 The set \eqref{iso_par} can obviously be thought as the (fundamental)
 parallelogram $\omega_1[0,1]+\omega_2[0,1]$ whose opposed edges are identified. Up to a {unimodular transform}, $\omega_1,\omega_2$ are unique, see \cite{JS}. In our case, suitable $\omega_1,\omega_2$ will be found in \eqref{expression_omega_1_2}.

If we cut the torus on Figure \ref{Real_points} along $[x_1, x_2]$ and $\Phi_0$, it becomes the parallelogram on the left in Figure \ref{Desc}. On the right in the same figure, this parallelogram is translated to the complex plane, and all corresponding important points are expressed in terms of the complex numbers $\o_1,\o_2$ (see above) and of $\o_3$ (to be defined below, in \eqref{expression_omega_3}).


\begin{figure}[t]
\begin{center}
\begin{picture}(00.00,710.00)
\hspace{-121mm}
\includegraphics{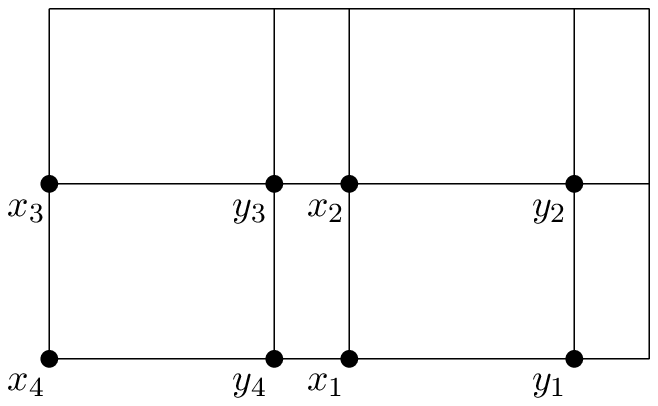}
\hspace{72mm}
\includegraphics{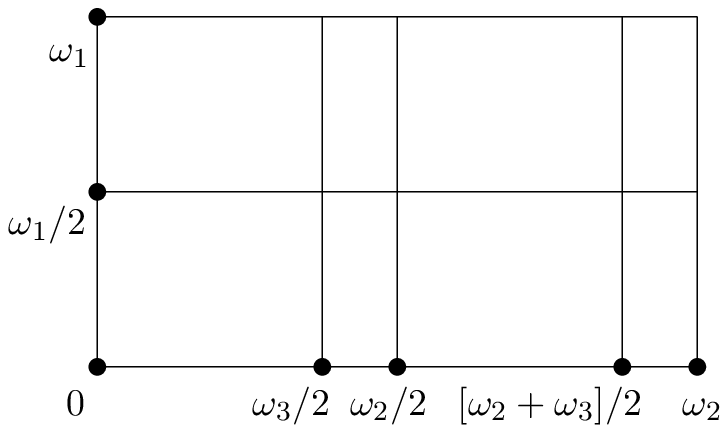}
\end{picture}
\end{center}
\vspace{-205mm}
\caption{The Riemann surface ${\bf C}/(\omega_1{\bf Z}+\omega_2{\bf Z})$ and the
location of the branch points}
\label{Desc}
\end{figure}

\section{Universal covering}
\label{uc}
\setcounter{equation}{0}

\subsection{An informal construction of the universal covering}
The Riemann surface ${\bf T}$ can be considered as a
 parallelogram whose opposite edges are identified, see \eqref{iso_par} and Figure \ref{Desc}.
 The universal covering of ${\bf T}$ can then be viewed as the union of infinitely many such
 parallelograms glued together, as in Figure~\ref{Informal}.

\begin{figure}[t]
\begin{center}
\begin{picture}(00.00,760.00)
\hspace{-105mm}
\includegraphics{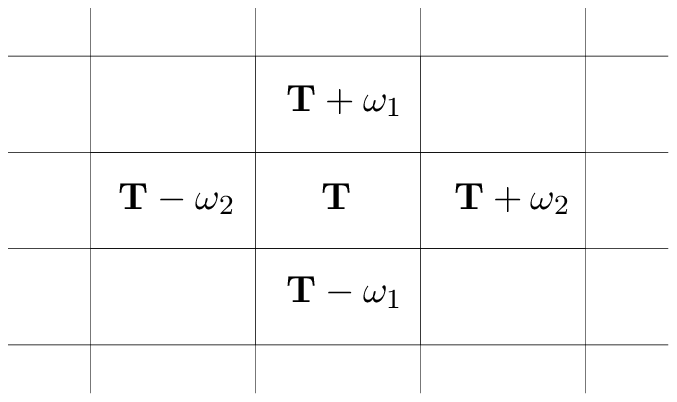}
\end{picture}
\end{center}
\vspace{-230mm}
\caption{Informal construction of the universal covering}
\label{Informal}
\end{figure}

\subsection{Periods and covering map} We now give a proper construction of the universal covering. The Riemann surface ${\bf T}$ being of genus $1$, its universal covering has the form $({\bf C}, \lambda)$, where ${\bf C}$ is the complex plane and $\lambda : {\bf C}\to {\bf T}$ is a non-branching covering map, see \cite{JS}. This way, the surface ${\bf T}$ can be considered as
the additive group ${\bf C}$ factorized by the discrete subgroup $\omega_1{\bf Z}+\omega_2{\bf Z}$, where the periods $\omega_1,\omega_2$ are complex numbers, linearly independent on ${\bf R}$. Any segment of length $|\omega_\ell|$ and parallel to $\omega_\ell$, $\ell \in \{1,2\}$, is projected onto a closed curve on ${\bf T}$ homological to one of the elements of the normal basis on the torus.
We choose {$\lambda([0,\omega_1])$} to be homological to the cut $[x_1, x_2]$ (and hence also to all other cuts $[x_3, x_4]$, $[y_1, y_2]$ and $[y_3, y_4]$); {$\lambda ([0,\omega_2])$} is then homological to the cycles of real points $\Phi_0$ and $\Phi_1$; see Figures \ref{Desc} and \ref{Parallelogram_identified}.

Our aim now is to find the expression of the covering $\lambda$. We will do this by finding, for all $\o \in {\bf C}$,  the explicit expressions of the pair of coordinates $(x(\lambda \omega), y(\lambda \omega))$, that we have introduced in Section \ref{Riemann_surface}. First, the periods $\omega_1,\omega_2$ are obtained in \cite[Lemma 3.3.2]{FIM} for $z=1/|\mathcal{S}|$. The reasoning is
exactly the same for other values of $z$, and we obtain that with $d$ as in \eqref{d_d_tilde},
     \begin{equation}
     \label{expression_omega_1_2}
          \omega_1= i\int_{x_1}^{x_2}
          \frac{\text{d}x}{[-d(x)]^{1/2}},
          \qquad \omega_2 = \int_{x_2}^{x_3} \frac{\text{d}x}{d(x)^{1/2}}.
     \end{equation}
{We also need to introduce}
     \begin{equation}
     \label{expression_omega_3}
          \omega_3 = \int_{X(y_1)}^{x_1}
          \frac{\text{d}x}{d(x)^{1/2}}.
     \end{equation}
Further, we define
     \begin{equation*}
     \label{def_f_x}
          g_x(t)=\left\{\begin{array}{lll}
          \displaystyle d''(x_{4})/6+d'(x_{4})/[t-x_{4}]& \text{if} & x_{4}\neq \infty,\\
          \displaystyle d''(0)/6+d'''(0)t/6 \phantom{{1^1}^{1}}& \text{if} & x_{4}=\infty,\end{array}\right.
     \end{equation*}
as well as
     \begin{equation*}
     \label{def_f_y}
          g_y(t)=\left\{\begin{array}{lll}
          \displaystyle d''(y_{4})/6+d'(y_{4})/[t-y_{4}]& \text{if} & y_{4}\neq \infty,\\
          \displaystyle d''(0)/6+d'''(0)t/6 \phantom{{1^1}^{1}}& \text{if} & y_{4}=\infty,\end{array}\right.
     \end{equation*}
and finally we introduce $\wp(\omega;\omega_1,\omega_2)$, the Weierstrass elliptic function with periods $\omega_1,\omega_2$. Throughout, we shall write $\wp(\omega)$ for $\wp(\omega;\omega_1,\omega_2)$. By definition, see \cite{JS,WW}, we have
     \begin{equation*}
          \wp(\omega)=\frac{1}{\omega^2}+\sum_{(\ell_1,\ell_2)\in{\bf Z}^{2}\setminus\{ (0,0)\}}
          \left[\frac{1}{(\omega-\ell_1\omega_1-\ell_2\omega_2)^{2}}-
          \frac{1}{(\ell_1\omega_1+\ell_2\omega_2)^{2}}\right].
     \end{equation*}

\begin{figure}[t]
\begin{center}
\begin{picture}(00.00,765.00)
\hspace{-105mm}
\includegraphics{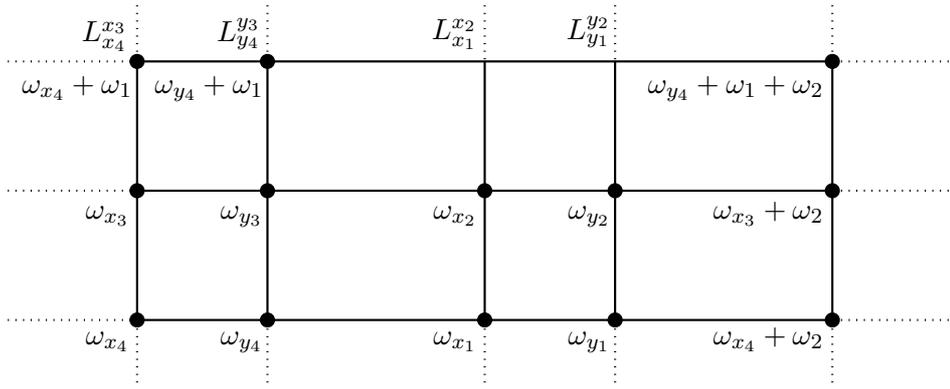} \end{picture}
\end{center}
\vspace{-220mm}
\caption{Important points and cycles on the universal covering}
\label{Parallelogram_identified}
\end{figure}

{\noindent}Then we have the uniformization \cite[Lemma 3.3.1]{FIM}
     \begin{equation}
     \label{expression_uniformization}
          \left\{\begin{array}{l}
          \hspace{-3mm}\phantom{y}x(\lambda \omega)=g_x^{-1}(\wp(\omega)),\\
          \hspace{-3mm}\phantom{x}y(\lambda \omega)=g_y^{-1}(\wp(\omega-\omega_3/2)).
          \end{array}\right.
 \end{equation}
From now on, whenever no ambiguity can arise, we drop the dependence w.r.t.\ $\lambda$, writing $x(\omega)$ and $y(\omega)$ instead of $x(\lambda\omega)$ and $y(\lambda\omega)$, respectively. The coordinates $x(\omega),y(\omega)$ defined in \eqref{expression_uniformization} are elliptic:
     \begin{equation}
     \label{peri}
          x(\omega+\omega_\ell)=x(\omega), \qquad  y(\omega+\omega_\ell)=y(\omega),\qquad  \forall\ell\in\{1,2\},\qquad \forall \omega\in{\bf C}.
     \end{equation}
Furthermore,
     \begin{equation*}
          \left\{\begin{array}{l}
          \hspace{-1.2mm}x(0)=x_4\\
          \hspace{-1mm}y(0)=Y(x_4)
          \end{array}\right.\hspace{-2mm},
          \ \
          \left\{\begin{array}{l}
          \hspace{-1.2mm}x(\omega_1/2)=x_3\\
          \hspace{-1mm}y(\omega_1/2)=Y(x_3)
          \end{array}\right.\hspace{-2mm},
          \ \
          \left\{\begin{array}{l}
          \hspace{-1.2mm}x(\omega_2/2)=x_1\\
          \hspace{-1mm}y(\omega_2/2)=Y(x_1)
          \end{array}\right.\hspace{-2mm},
          \ \
          \left\{\begin{array}{l}
          \hspace{-1.2mm}x([\omega_1+\omega_2]/2)=x_2\\
          \hspace{-1mm}y([\omega_1+\omega_2]/2)=Y(x_2)
          \end{array}\right.\hspace{-2mm}.
     \end{equation*}
Let us denote the points $0,\omega_1/2,\omega_2/2,[\omega_1+\omega_2]/2$ by $\omega_{x_4}, \omega_{x_3}, \omega_{x_1}, \omega_{x_2}$, respectively, see Figures \ref{Desc} and \ref{Parallelogram_identified}. Let
     \begin{equation*}
          L_{x_4}^{x_3}=\omega_{x_4}+\omega_1{\bf R},
          \qquad
          L_{x_1}^{x_2}=\omega_{x_1}+\omega_1{\bf R}.
     \end{equation*}
Then $\lambda L_{x_4}^{x_3}$ (resp.\ $\lambda L_{x_1}^{x_2}$) is the cut of ${\bf T}$ where ${\bf S}_x^1$ and ${\bf S}_x^2$ are glued together, namely, $\{s\in{\bf T}: x(s) \in  [x_3, x_4] \}$ (resp.\ $\{s\in{\bf T}: x(s) \in  [x_1, x_2] \}$).


Moreover, by construction we have (see again Figures \ref{Desc} and \ref{Parallelogram_identified})
      \begin{equation*}
          \left\{\begin{array}{l}
          \hspace{-1.2mm}x(\omega_3/2)=X(y_4)\\
          \hspace{-1mm}y(\omega_3/2)=y_4
          \end{array}\right.\hspace{-2mm},
          \ \
          \left\{\begin{array}{l}
          \hspace{-1.2mm}x([\omega_1+\omega_3]/2)=X(y_3)\\
          \hspace{-1mm}y([\omega_1+\omega_3]/2)=y_3
          \end{array}\right.\hspace{-2mm},
          \ \
          \left\{\begin{array}{l}
          \hspace{-1.2mm}x([\omega_2+\omega_3]/2)=X(y_1)\\
          \hspace{-1mm}y([\omega_2+\omega_3]/2)=y_1
          \end{array}\right.\hspace{-2mm},
     \end{equation*}
and
     \begin{equation*}
          \left\{\begin{array}{l}
          \hspace{-1.2mm}x([\omega_1+\omega_2+\omega_3]/2)=X(y_2)\\
          \hspace{-1mm}y([\omega_1+\omega_2+\omega_3]/2)=y_2
          \end{array}\right.\hspace{-2mm}.
     \end{equation*}
We denote the points $\omega_3/2, [\omega_1 +\omega_3]/2,[\omega_2+\omega_3]/2,  [\omega_1+\omega_2 +\omega_3]/2$ by $\omega_{y_4}, \omega_{y_3}, \omega_{y_1}, \omega_{y_2}$, respectively. Let
     \begin{equation*}
          L_{y_4}^{y_3}=\omega_{y_4}+\omega_1{\bf R},
          \qquad
          L_{y_1}^{y_2}=\omega_{y_1}+\omega_1{\bf R}.
     \end{equation*}
Then $\lambda L_{y_4}^{y_3}$ (resp.\ $\lambda L_{y_1}^{y_2}$) is  the cut of ${\bf T}$ where ${\bf S}_y^1$ and ${\bf S}_y^2$  are glued together, that is to say $\{s\in{\bf T}: y(s) \in  [y_3, y_4] \}$ (resp.\ $\{s\in{\bf T}: y(s) \in  [y_1, y_2] \}$).

The distance between $L_{x_4}^{x_3}$ and $L_{y_4}^{y_3}$ is the same as between $L_{x_1}^{x_2}$ and $L_{y_1}^{y_2}$; it equals $\o_3/2$.


\subsection{Lifted Galois automorphisms $\widehat \xi, \widehat \eta$}
Any conformal automorphism $\zeta$ of the surface ${\bf T}$ can be continued as a conformal automorphism $\widehat \zeta=\lambda^{-1} \zeta \lambda$ of the universal covering {\bf C}. This continuation is not unique, but it will be unique if we fix some $\widehat \zeta \omega_0 \in {\lambda^{-1} \zeta\lambda\omega_0}$ for a given point $\omega_0 \in {\bf C}$.

According to \cite{FIM}, we define $\widehat \xi, \widehat \eta$ by
choosing their fixed points to be $\omega_{x_2},\omega_{y_2}$,
respectively. Since any conformal automorphism of ${\bf C}$ is an
affine function of $\omega$ \cite{JS} and since $\widehat \xi\,^2 =
\widehat \eta\,^2={\rm Id}$, we have
     \begin{equation}
     \label{hatx}
          \widehat \xi \omega=-\omega +2\omega_{x_2},
          \qquad
          \widehat \eta \omega=- \omega+2\omega_{y_2}.
     \end{equation}
It follows that $\widehat \eta \widehat \xi $ and $ \widehat \xi \widehat \eta$ are just the shifts via the real numbers $\omega_3$ and $-\omega_3$, respectively:
     \begin{equation}
     \label{O3}
          \widehat\eta\widehat\xi \omega=\omega+ 2(\omega_{y_2}-\omega_{x_2})=\omega+\omega_3,
          \qquad
          \widehat \xi \widehat \eta \omega=\omega+ 2(\omega_{x_2}-\omega_{y_2})=\omega-\omega_3.
     \end{equation}
By \eqref{xie} and \eqref{bue} we have
     \begin{equation}
     \label{hat}
          x(\widehat \xi \omega)=x(\omega),
          \quad
          y(\widehat \xi \omega)=\frac{c(x(\omega))}{a(x(\omega))}\frac{1}{y(\omega)},
          \quad
          x(\widehat \eta \omega)=\frac{\widetilde c(y(\omega))}{ \widetilde a(y(\omega))}\frac{1}{x(\omega)},
          \quad y(\widehat \eta \omega)=y(\omega).
     \end{equation}
Finally, $\widehat \xi L_{x_1}^{x_2}=L_{x_1}^{x_2}$, $\widehat \xi L^{x_3}_{x_4}=L^{x_3}_{x_4}+\o_2$ and $\widehat \eta  L_{y_1}^{y_2}=L_{y_1}^{y_2}$, $\widehat \eta L^{y_3}_{y_4}=L^{y_3}_{y_4}+\o_2$.

\section{Lifting of $x \mapsto Q(x,0)$ and $y\mapsto Q(0,y)$ to the universal covering}
\label{Lifting_uc}
\setcounter{equation}{0}

\subsection{Lifting to the Riemann surface ${\bf T}$}
We have seen in Section \ref{Riemann_surface} that for any $z\in ]0,1/|\mathcal{S}|[$, exactly two branch points of $Y(x)$ (namely, $x_1$ and $x_2$) are in the unit disc. For this reason, and by construction of the surface ${\bf T}$, the set $\{s\in{\bf T} : |x(s)|=1\}$ is composed of two cycles (one belongs to ${\bf S}_x^1$ and the other to ${\bf S}_x^2$) homological to the cut $\{s\in {\bf T} : x(s)\in [x_1, x_2]\}$. The domain $\mathscr{D}_x=\{s\in{\bf T} : |x(s)|< 1 \}$ is bounded  by these two cycles, see Figure \ref{Domd}, and contains the points $s\in{\bf T}$ such that $x(s) \in [x_1, x_2]$. Since the function $x \mapsto K(x,0)Q(x,0)$ is holomorphic in the unit disc,
we can lift it to $\mathscr{D}_x \subset {\bf T}$ as
     \begin{equation*}
          r_x(s)=K(x(s),0)Q(x(s), 0),\qquad \forall s \in \mathscr{D}_x.
     \end{equation*}

\begin{figure}[!t]
\begin{center}
\begin{picture}(00.00,725.00)
\hspace{-58.5mm}
\includegraphics{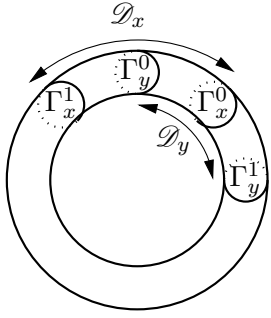}
\end{picture}
\end{center}
\vspace{-220mm}
\caption{
Location of the domains $\mathscr{D}_x$ and $\mathscr{D}_y$ on the Riemann surface ${\bf T}$
}
\label{Domd}
\end{figure}

In the same way, the domain $\mathscr{D}_y=\{s\in{\bf T} : |y(s)|< 1\}$ is bounded by $\{s\in{\bf T} : |y(s)|=1\}$, which consists in two cycles homological to the cut $\{s\in {\bf T} : y(s)\in [y_1, y_2]\}$, see Figure \ref{Domd}, and which contains the latter.
We lift the function $y \mapsto K(0,y)Q(0,y)$ to $\mathscr{D}_y \subset {\bf T}$ as
     \begin{equation*}
          r_y(s)=K(0,y(s))Q(0,y(s)),\qquad \forall s \in \mathscr{D}_y.
     \end{equation*}

It is shown in \cite[Lemma 3]{Ra} that for any $z \in ]0, 1/|\mathcal{S}|[$ and any $x$ such that $|x|= 1$, we have $|Y_0(x)|< 1$ and $|Y_1(x)|>1$. Hence, the cycles that constitute the boundary of $\mathscr{D}_x$ are
     \begin{equation*}
          \Gamma^0_x= \{s\in{\bf T} : |x(s)|=1,\, |y(s)|< 1\},
          \qquad
          \Gamma^1_x=\{s\in{\bf T} : |x(s)|=1,\, |y(s)|>1\}.
     \end{equation*}
We thus have $\Gamma^0_x \in \mathscr{D}_y$  and $\Gamma^1_x \notin \mathscr{D}_y$, see Figure \ref{Domd}. In the same way, for any $z \in ]0,1/|\mathcal{S}|[$ and any $y$ such that $|y|=1$, we have $|X_0(y)|< 1$ and $|X_1(y)|>1$. Therefore, the cycles composing the boundary of $\mathscr{D}_y$ are
     \begin{equation*}
          \Gamma^0_y= \{s\in{\bf T} : |y(s)|=1,\, |x(s)|< 1\},
          \qquad
          \Gamma^1_y=\{s \in{\bf T}: |y(s)|=1,\, |x(s)|>1\}.
     \end{equation*}
Furthermore, $\Gamma^0_y \in \mathscr{D}_x$ and $\Gamma^1_y \notin \mathscr{D}_x$, see Figure \ref{Domd}.

It follows that $\mathscr{D}_x \cap \mathscr{D}_y=\{s\in{\bf T} : |x(s)|<1,\,|y(s)|< 1\}$ is not empty, simply connected and bounded by $\Gamma^0_x$ and $\Gamma^0_y$. Since for any $s \in {\bf T}$, $K(x(s), y(s))=0$, and since the main equation \eqref{functional_equation} is valid on $\{(x,y) \in {\bf C}^2: |x| < 1,\, |y| < 1\}$, we have
     \begin{equation}
     \label{qqq}
          r_x(s)+r_y(s)-K(0,0)Q(0,0)-x(s)y(s)=0, \qquad \forall s\in \mathscr{D}_x \cap \mathscr{D}_y.
     \end{equation}



\subsection{Lifting to the universal covering ${\bf C}$}
The domain $\mathscr{D}$ lifted on the universal covering consists of infinitely many curvilinear strips shifted by $\o_2$:
     \begin{equation*}
          \lambda^{-1}\mathscr{D}_x=\bigcup _{n\in{\bf Z}}\Delta_x^n,
          \qquad
          \Delta_x^n\subset\omega_1{\bf R}+]n\omega_2,(n+1)\omega_2[,
     \end{equation*}
and, likewise,
     \begin{equation*}
          \lambda^{-1}\mathscr{D}_y=\bigcup _{n\in{\bf Z}}\Delta_y^n,
          \qquad
          \Delta_y^n\subset\omega_1{\bf R}+\omega_3/2+]n\omega_2,(n+1)\omega_2[.
     \end{equation*}
Let us consider these strips for $n=0$, that we rename
     \begin{equation*}
          \Delta_x=\Delta_x^0,\qquad \Delta_y=\Delta_y^0.
     \end{equation*}
The first is bounded by $\widehat \Gamma^1_x\subset \lambda^{-1}\Gamma^1_x$ and by $\widehat \Gamma^0_x\subset \lambda^{-1}\Gamma^0_x$, while the second is delimited by $\widehat \Gamma^0_y\subset \lambda^{-1} \Gamma^0_y$ and by $\widehat\Gamma^1_y\subset \lambda^{-1} \Gamma^1_y$.

Further, note that the straight line
$L_{x_1}^{x_2}$ (resp.\ $L_{y_1}^{y_2}$)
 defined in Section \ref{uc} is invariant
 w.r.t.\ $\widehat \xi$ (resp.\ $\widehat \eta$) and belongs to $\Delta_x$ (resp.\ $\Delta_y$).

Then, by the facts that $\xi \Gamma^1_x=\Gamma^0_x$ and $\eta
\Gamma^1_y=\Gamma^0_y$, and by our choice \eqref{hatx} of the definition of $\widehat\xi$ and $\widehat\eta$ on the universal
covering, we have $\widehat \xi \,\widehat \Gamma^1_x=\widehat
\Gamma^0_x$ and $\widehat \eta\, \widehat \Gamma^1_y=  \widehat
\Gamma^0_y$. In addition,
     \begin{equation}
     \label{fghj}
          \widehat \xi \omega \in\Delta_x, \quad\forall\omega \in \Delta_x,
          \qquad
          \widehat \eta \omega \in\Delta_y, \quad\forall\omega \in \Delta_y.
     \end{equation}
Moreover, since $ \Gamma^0_y \in \mathscr{D}_x$, $\Gamma^1_y \notin \mathscr{D}_x$ and $\Gamma^0_x \in \mathscr{D}_y$, $\Gamma^1_x \notin \mathscr{D}_y$, we have $\widehat\Gamma^0_y \in \Delta_x$, $\widehat  \Gamma^1_y \notin \Delta_x$ and $\widehat \Gamma^0_x \in \Delta_y$, $\widehat  \Gamma^1_x \notin \Delta_y$. It follows that $\Delta_x \cap \Delta_y$ is a non-empty strip bounded by $\widehat  \Gamma^0_x$ and $\widehat \Gamma^0_y$, and that
     \begin{equation*}
          \Delta=\Delta_x \cup \Delta_y
     \end{equation*}
is simply connected, as in Figure \ref{Fig_Delta}.


Let us lift the functions $r_x(s)$ and $r_y(s)$ holomorphically to
$\Delta_x$ and $\Delta_y$, respectively: we put
     \begin{equation}
     \label{QQ}
          \left\{\begin{array}{ll}
          r_x(\omega)=r_x(\lambda \omega)=K(x(\omega),0)Q(x(\omega), 0),& \qquad\forall \omega\in \Delta_x,\\
          r_y(\omega)\hspace{0.12mm}=r_y(\lambda \omega)\hspace{0.12mm}=K(0, y(\omega))Q(0, y(\omega)),& \qquad\forall \omega \in\Delta_y.
          \end{array}\right.
     \end{equation}
It follows from \eqref{qqq} and \eqref{QQ} that
     \begin{equation}
     \label{sqs}
          r_x(\omega)+r_y(\omega)-K(0,0)Q(0,0)-x(\omega)y(\omega) =0,\qquad \forall\omega \in \Delta_x \cap \Delta_y.
     \end{equation}
Equation \eqref{sqs} allows us to continue functions $r_x(\omega)$
and $r_y(\omega)$ meromorphically on $\Delta$: we put
     \begin{equation}
     \label{QQR}
          \left\{\begin{array}{ll}
          r_x(\omega)=-r_y(\omega)+K(0,0)Q(0,0)+x(\omega)y(\omega),& \qquad\forall \omega\in \Delta_y,\\
          r_y(\omega)=-r_x(\omega)+K(0,0)Q(0,0)+x(\omega) y(\omega),& \qquad\forall \omega \in\Delta_x.
          \end{array}\right.
     \end{equation}
Equation \eqref{sqs} is then valid on the whole of $\Delta$. We summarize all facts above in the next result.
\begin{thm}
\label{thmde}
The functions
     \begin{equation*}
          r_x(\o)=\left\{\begin{array}{lll}
          \phantom{-}K(x(\o),0)Q(x(\o),0) & \text{if} &\o \in \Delta_x, \\
          -K(0, y(\o))Q(0, y(\o))+ K(0,0)Q(0,0) +x(\o)y(\o) & \text{if} & \o \in \Delta_y,
          \end{array}\right.
     \end{equation*}
and
     \begin{equation*}
          r_y(\o)=\left\{\begin{array}{lll}
          \phantom{-}K(0, y(\o))Q(0,y(\omega)) & \text{if} &\o \in \Delta_y, \\
          -K(x(\o),0)Q(x(\o),0)+ K(0,0)Q(0,0)+ x(\o)y(\o) & \text{if} & \o \in \Delta_x,
          \end{array}\right.
     \end{equation*}
are meromorphic in $\Delta$. Furthermore,
     \begin{equation}
     \label{sqs1}
          r_x(\omega)+r_y(\omega)-K(0,0)Q(0,0)-x(\omega)y(\omega) =0,
          \qquad \forall \o\in\Delta.
     \end{equation}
\end{thm}

\begin{figure}[t]
\begin{center}
\begin{picture}(00.00,725.00)
\hspace{-105mm}
\includegraphics{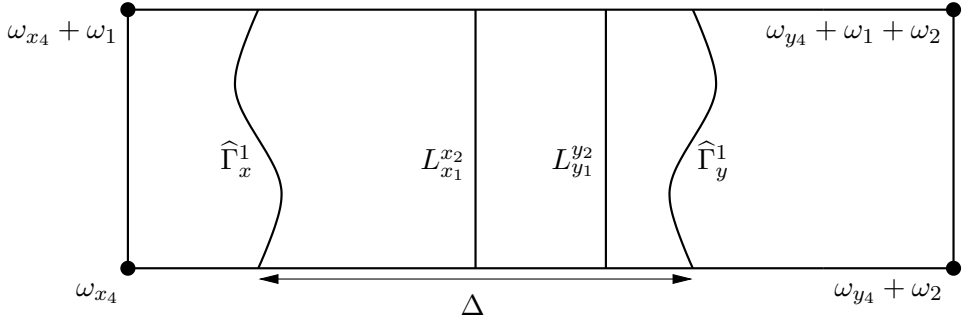}
\end{picture}
\end{center}
\vspace{-215mm}
\caption{Location of $\Delta=\Delta_x\cup\Delta_y$}
\label{Fig_Delta}
\end{figure}

\section{Meromorphic continuation of $x \mapsto Q(x,0)$ and $y \mapsto Q(0,y)$ on the universal covering}
\label{meroun}
\setcounter{equation}{0}

\subsection{Meromorphic continuation}
In Theorem \ref{thmde}  we saw that $r_x(\omega)$ and $r_y(\omega)$ are meromorphic on $\Delta$. We now continue these functions
meromorphically from $\Delta$ to the whole of~${\bf C}$.

\begin{thm}
\label{thm2} The functions $r_x(\omega)$ and $r_y(\omega)$ can be
continued meromorphically to the whole of ${\bf C}$. Further, for
any $\omega \in {\bf C}$, we have
  \begin{align}
       &r_x(\omega-\omega_3)=r_x(\omega)+y(\omega)[x(-\omega+2\omega_{y_2})-x(\omega)],
       \label{cont}
       \\
       &r_y(\omega+\omega_3)\hspace{0.12mm}=r_y(\omega)\hspace{0.12mm}+x(\omega)[y(-\omega+2\omega_{x_2})-y(\omega)],
       \label{cont1}
       \\
       &r_x(\omega)+r_y(\omega)-K(0,0)Q(0,0)-x(\omega)y(\omega) =0,
       \label{sqs2}
       \\
       &\left\{\begin{array}{cc}
       r_x(\widehat \xi \omega)=r_x(\omega), \\r_y(\widehat \eta\omega)=r_y(\omega),
       \end{array}\right.
       \label{xieta}
       \\
       &\left\{\begin{array}{cc}
       r_x(\omega+\omega_1)=r_x(\omega), \\
       r_y(\omega+\omega_1)=r_y(\omega).
       \end{array}\right.       \label{buzz}
  \end{align}
  \end{thm}

For the proof of Theorem \ref{thm2}, we shall need the following lemma.

\begin{lem} We have
\label{lemde}
     \begin{equation}
     \label{delta}
          \bigcup_{n\in {\bf Z}}(\Delta+n\omega_3) ={\bf C}.
     \end{equation}
\end{lem}

\begin{proof}
It has been noticed in Section \ref{Lifting_uc} that
 $\widehat \xi \,\widehat \Gamma^1_x=\widehat\Gamma^0_x\in \Delta_y$. By \eqref{fghj}, $\widehat \eta\, \widehat \Gamma^0_x \in \Delta_y \subset\Delta$, so that, by \eqref{O3},
     \begin{equation*}
          \widehat \Gamma^1_x+\omega_3=\widehat \eta \widehat \xi \,\widehat \Gamma^1_x \in \Delta.
     \end{equation*}
In the same way, $\widehat \Gamma^1_y-\o_3 \in \Delta$. It follows that $\Delta \cup (\Delta+\o_3)$ is a simply connected domain, see Figure \ref{Fig_Delta}. Identity \eqref{delta} follows.
\end{proof}

\begin{proof}[Proof of Theorem \ref{thm2}]
For any $\omega \in \Delta$, by Theorem \ref{thmde} we have
  \begin{equation}
     \label{sdhj1}
          r_x(\omega)+
          r_y(\omega)-K(0,0)Q(0,0)-x(\omega) y(\omega)=0.
     \end{equation}
For any $\omega\in\Delta$ close enough to the cycle $\widehat
\Gamma^1_x$, we have that $\widehat \xi \omega \in \Delta_y$ since
 $\widehat \xi \,\widehat \Gamma^1_x=\widehat\Gamma^0_x\in
 \Delta_y$. Then
 $\omega+\omega_3=\widehat \eta \widehat \xi \omega \in\Delta_y$
  by \eqref{fghj}. We now
  compute $r_y(\widehat \eta \widehat \xi \omega)$ for any such $\omega$.
   Equation \eqref{sqs1}, which is valid in $\Delta\supset\Delta_y$, gives
     \begin{equation}
     \label{sdhj}
          r_x(\widehat \xi \omega)+r_y(\widehat \xi \omega)-K(0,0)Q(0,0)-x(\widehat \xi \omega) y(\widehat \xi \omega)=0.
     \end{equation}
By \eqref{hat}, $x(\widehat \xi \omega)=x(\omega)$. For our $\omega \in \Delta_x$, by \eqref{fghj} we have $\widehat \xi \omega \in \Delta_x$, so that Theorem \ref{thmde} yields
     \begin{equation*}
          r_x(\widehat \xi \omega)=
          K(x(\widehat \xi \o),0)Q(x(\widehat \xi \omega), 0)=
          K(x(\o),0)Q(x(\omega), 0)=r_x(\omega).
     \end{equation*}
If we now combine the last fact together with Equation \eqref{sdhj1}, Equation \eqref{sdhj} and identity $x(\widehat \xi \omega)=x(\omega)$, we obtain that
     \begin{equation*}
          r_y(\widehat \xi \omega)=r_y(\omega)+x(\omega)[y(\widehat \xi\omega)-y(\omega)].
     \end{equation*}
Since $\widehat \xi \omega \in \Delta_y$,
 then by \eqref{fghj} we have $\widehat \eta \widehat \xi \omega
 \in\Delta_y$.
 Equation \eqref{hat} and  Theorem \ref{thmde} entail
     \begin{equation*}
          r_y(\widehat \eta \widehat \xi \omega)=
          K(0, y(\widehat \eta \widehat \xi \omega) )
          Q(0, y(\widehat \eta \widehat \xi \omega))=
         K(0, y( \widehat \xi \omega ))
         Q(0, y(\widehat \xi \omega))=r_y(\widehat \xi \omega).
     \end{equation*}
 Finally, for all $\omega \in \Delta $ close enough to $\widehat \Gamma^1_x$
   we have
     \begin{equation*}
          r_y(\widehat \eta \widehat \xi \omega)=r_y(\omega)+
          x(\omega)[y(\widehat \xi\omega)-y(\omega)].
     \end{equation*}
Using \eqref{O3}, we obtain exactly Equation \eqref{cont1}. Thanks to Theorem \ref{thmde} and Lemma \ref{lemde}, this equation shown for any $\omega \in \Delta$ close enough to $\widehat \Gamma^1_x$ allows us to continue $r_y$ meromorphically  from $\Delta$ to the whole of ${\bf C}$. Equation \eqref{cont1} therefore stays valid for any $\omega \in {\bf C}$. The function $r_y(\widehat \eta\omega)=r_y(-\omega+\omega_{y_2})$ is then also meromorphic on ${\bf C}$. Since these functions coincide in $\Delta_y$, then by the principle of analytic continuation \cite{JS} they do on the whole of ${\bf C}$. In the same way, we prove Equation \eqref{cont} for all $\omega \in \Delta_y$ close enough to $\widehat\Gamma^1_y$. Together with Theorem \ref{thmde} and Lemma \ref{lemde} this allows us to continue $r_x(\omega)$ meromorphically to the whole of ${\bf C}$. By the same continuation argument, the identity $r_x(\omega)=r_x(\widehat \xi \omega)$ is valid everywhere on ${\bf C}$. Consequently Equation \eqref{sqs2}, which a priori is satisfied in $\Delta$, must stay valid on the whole of ${\bf C}$. Since $x(\omega)$ and $y(\omega)$ are $\omega_1$-periodic, it follows from Theorem \ref{thmde} that $r_x(\omega)$ and $r_y(\omega)$ are $\omega_1$-periodic in $\Delta$. The vector $\omega_3$ being real, by \eqref{cont} and \eqref{cont1} these functions stay $\omega_1$-periodic on the whole of ${\bf C}$.
\end{proof}

\subsection{Branches of $x\mapsto Q(x,0)$ and $y\mapsto Q(0,y)$}
\label{subsecbra}
 The restrictions of $r_x(\o)/K(x(\o),0)$ on
     \begin{equation}
     \label{mkl}
          \mathscr{M}_{k,\ell}=\omega_1[\ell,\ell+1[+\omega_2[k/2,(k+1)/2[
     \end{equation}
for $k,\ell \in {\bf Z}$ provide all branches on ${\bf C}\setminus
([x_1,x_2]\cup[x_3,x_4])$ of $Q(x,0)$ as follows:
     \begin{equation}
     \label{branches}
          Q(x,0)=\{r_x(\o)/K(x(\o),0):
           \o \text{ is the (unique) element of } \mathscr{M}_{k,\ell} \text{ such that } x(\o)=x\}.
     \end{equation}
Due to the $\o_1$-periodicity of $r_x(\o)$ and $x(\o)$, the
restrictions of these functions on $\mathscr{M}_{k,\ell}$ do not
depend on $\ell\in {\bf Z}$, and therefore determine the same branch
as on $\mathscr{M}_{k,0}$ for any~$\ell$. Furthermore, thanks to
\eqref{xieta}, \eqref{hatx} and \eqref{hat} the restrictions of
$r_x(\o)/K(x(\o),0)$ on $\mathscr{M}_{-k+1,0}$ and on
$\mathscr{M}_{k,0}$ lead to the same branches for any $k\in{\bf Z}$.
Hence, the restrictions of $r_x(\o)/K(x(\o),0)$ to
$\mathscr{M}_{k,0}$ with $k\geq 1$ provide all different branches of
this function. The analogous statement holds for the restrictions of
$r_y(\o)/K(0, y(\o))$ on
     \begin{equation}
     \label{nkl}
          \mathscr{N}_{k,\ell}=\omega_3/2+\omega_1[\ell,\ell+1[+\omega_2]k/2,(k+1)/2]
     \end{equation}
for $k,\ell \in {\bf Z}$, namely:
\begin{equation}
     \label{branches1}
          Q(0,y)=\{r_y(\o)/K(0,y(\o)):
           \o \text{ is the (unique) element of } \mathscr{N}_{k,\ell} \text{ such that } y(\o)=y\}.
     \end{equation}
The restrictions on $\mathscr{N}_{k,\ell}$ for $\ell\in {\bf Z}$
give the same branch as on $\mathscr{N}_{k,0}$. For any $k \in {\bf
Z}_+$ the restrictions
on $\mathscr{N}_{-k+1,0}$
and on $\mathscr{N}_{k,0}$ determine the same branches. Hence, the
restrictions of $r_y(\o)/K(0,y(\o))$ on $\mathscr{N}_{k,0}$ with
$k\geq 1$ provide all different branches of $y \mapsto  Q(0,y)$.

\subsection{Ratio $\o_2/\o_3$}
\label{subsecra}

{
\begin{rem}
\label{Remark1}
  For any $z \in ]0, 1/|\mathcal{S}|[ $ the value $\o_2/\o_3$
   is rational  if and only if the group $\langle \xi, \eta \rangle
   $ restricted to the  curve $\{(x,y)\in{\bf C}\cup\{\infty\}^2 : K(x,y)=0\}$ is
   finite, see \cite[Section 4.1.2]{FIM} and \cite[Proof of Proposition 4]{Ra}.
\end{rem}
}

   The rationality or irrationality of the quantity $\o_2/\o_3$
   is crucial for the nature of the functions
    $x \mapsto Q(x,0)$ and $y \mapsto Q(0,y)$ for a given $z$.
Indeed, the following theorem holds true.
\begin{thm}
\label{ratho}
   For any $z \in ]0, 1/|\mathcal{S}|[$ such that
 $\o_2/\o_3$ is rational, the functions
 $x \mapsto Q(x,0)$ and $y \mapsto Q(0,y)$ are holonomic.
 \end{thm}

\begin{proof}
The proof of Theorem \ref{ratho} is completely similar to that of Theorems 1.1 and 1.2 in \cite{FR}, so here we just recall the main ideas. The proof actually consists in applying \cite[Theorem 4.4.1]{FIM}, which entails that if $\o_2/\o_3$ is rational, the function $Q(x,0)$ can be written as
\begin{equation*}
Q(x,0)=w_1(x)+\widetilde\Phi(x)\phi(x)+w(x)/r(x),
\end{equation*}
where $w_1$ and $r$ are rational functions, while $\phi$ and $w$ are algebraic. Further, in \cite[Lemma 2.1]{FR} it  is shown that $\widetilde\Phi$ is holonomic. Accordingly, $Q(x,0)$ is also holonomic. The argument for $Q(0,y)$ is similar. Notice that Theorem 4.4.1 in \cite{FIM} is proved for $z=1/|\mathcal{S}|$ only, but in \cite{FR} it is observed that this result also holds for $z \in ]0, 1/|\mathcal{S}|[$.
\end{proof}

For all $23$ models of walks with finite group \eqref{group}, the
ratio $\omega_2/\omega_3$ is rational and {\it independent} of $z$.
This fact, which is specified in Lemma~\ref{lemma_omega_2_3} below,
implies the holonomy of the functions $x \mapsto Q(x,0)$ and $y
\mapsto Q(0,y)$ for all $z \in ]0, 1/|\mathcal{S}|[$ by Theorem
\ref{ratho}, and  also leads to some more profound analysis of the
models with a finite group. This analysis is the topic of  the
Section~\ref{section_holonomy}.

 For all $51$ non-singular models of walks with infinite group,
    $\o_2/\o_3$ takes rational and irrational values on subsets
    $\mathcal{H}$ and $]0, 1/|\mathcal{S}|[\setminus \mathcal{H}$,
 respectively, which are dense on
  $]0, 1/|\mathcal{S}|[$,
      as it will be proved in Proposition~\ref{proir} below.
   For any $z\in \mathcal{H}$,  $x \mapsto Q(x,0)$ and $y \mapsto Q(0,y)$
     are holonomic by Theorem~\ref{ratho}.
 For all $z \in ]0, 1/|\mathcal{S}|[\setminus \mathcal{H} $, properties of the
branches of $x \mapsto Q(x,0)$ and $y \mapsto Q(0,y)$ (in
particular, the set of their poles) will be studied in detail in
Section~\ref{Section_infinite_group}; the non-holonomy will be
derived from this analysis.

\section{Finite group case}
\setcounter{equation}{0}

\label{section_holonomy}

\begin{figure}[t]
\begin{center}
\begin{picture}(340.00,68.00)
\includegraphics{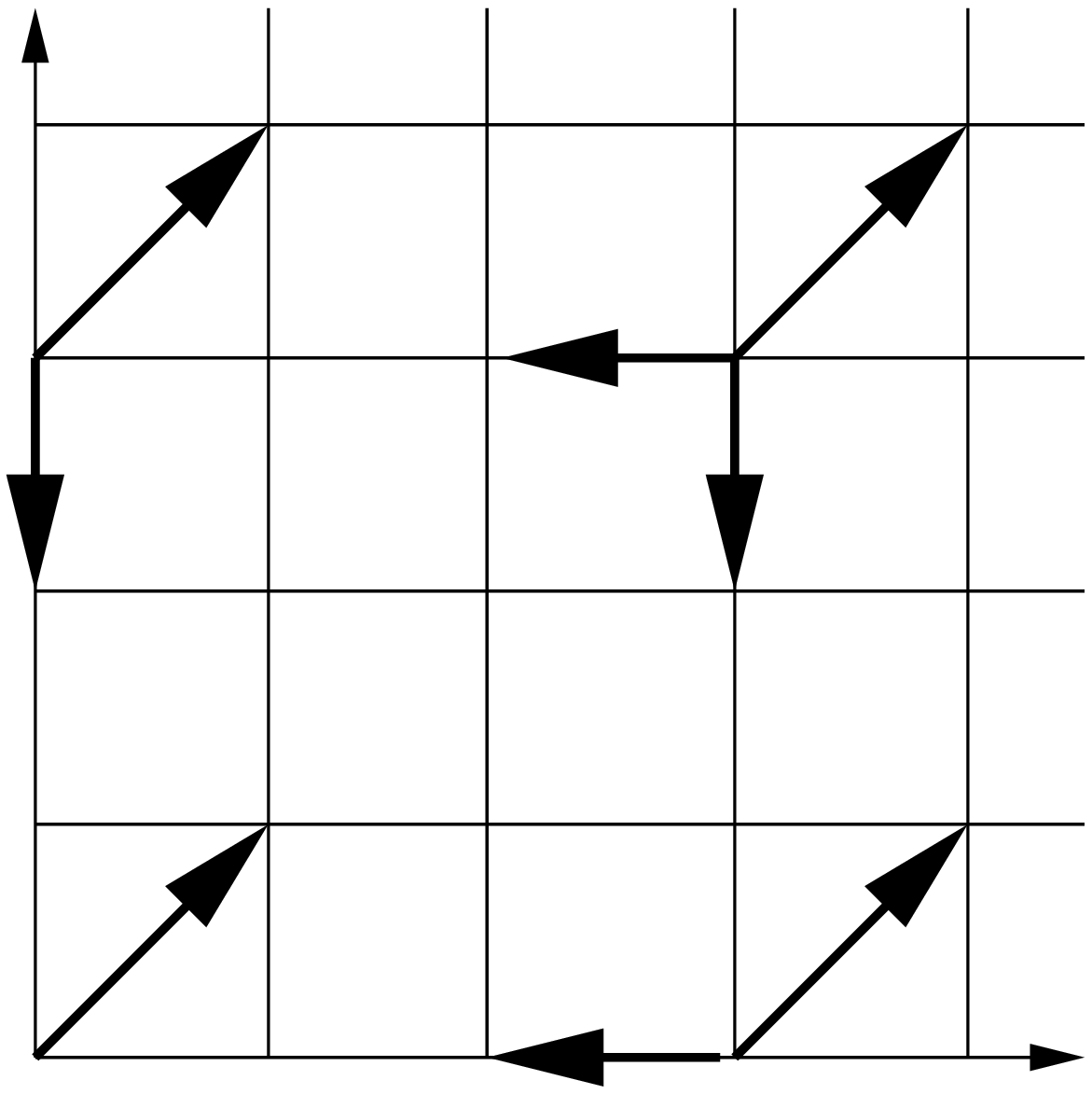}
\hspace{45mm}
\includegraphics{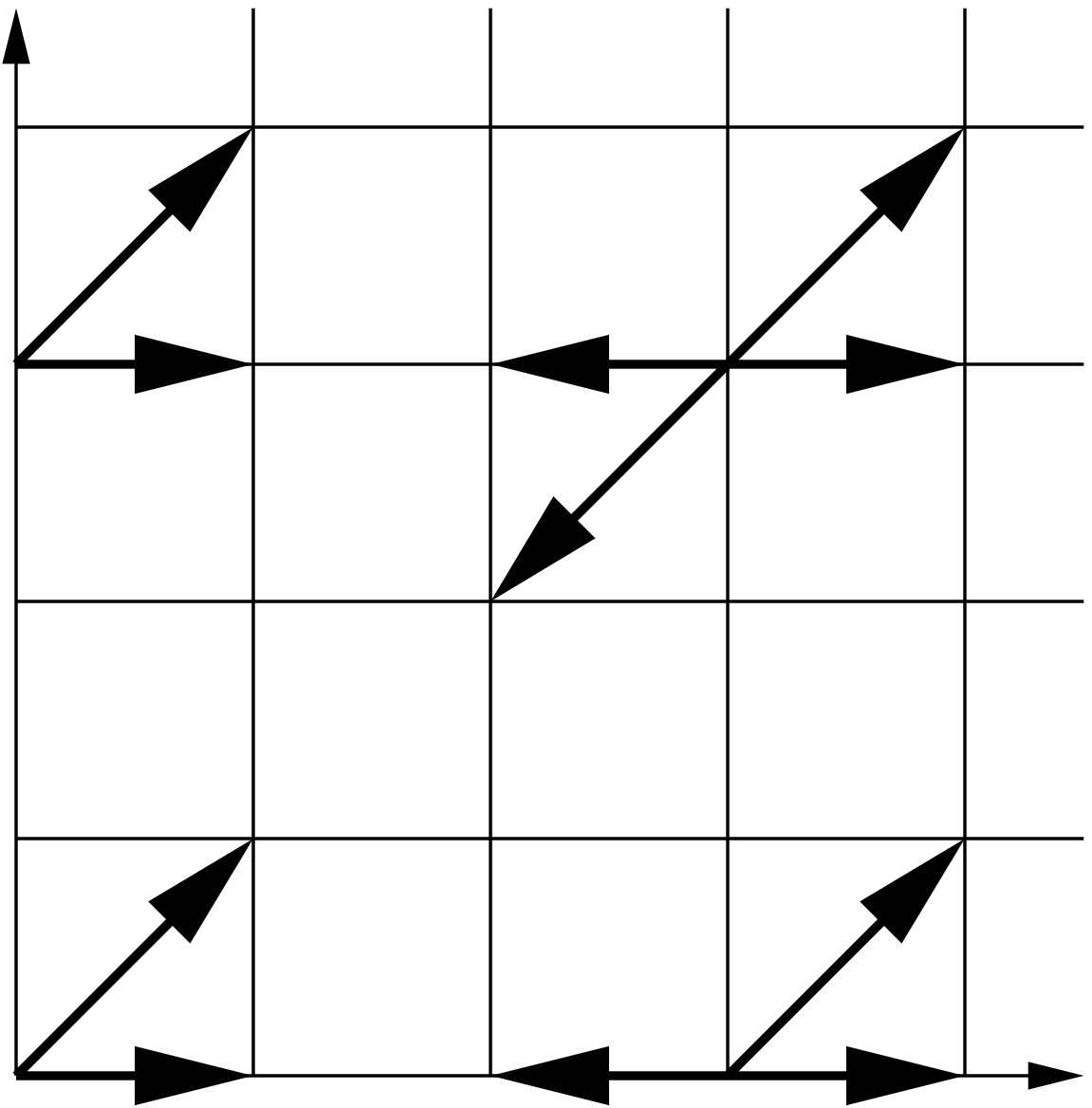}
\hspace{45mm}
\includegraphics{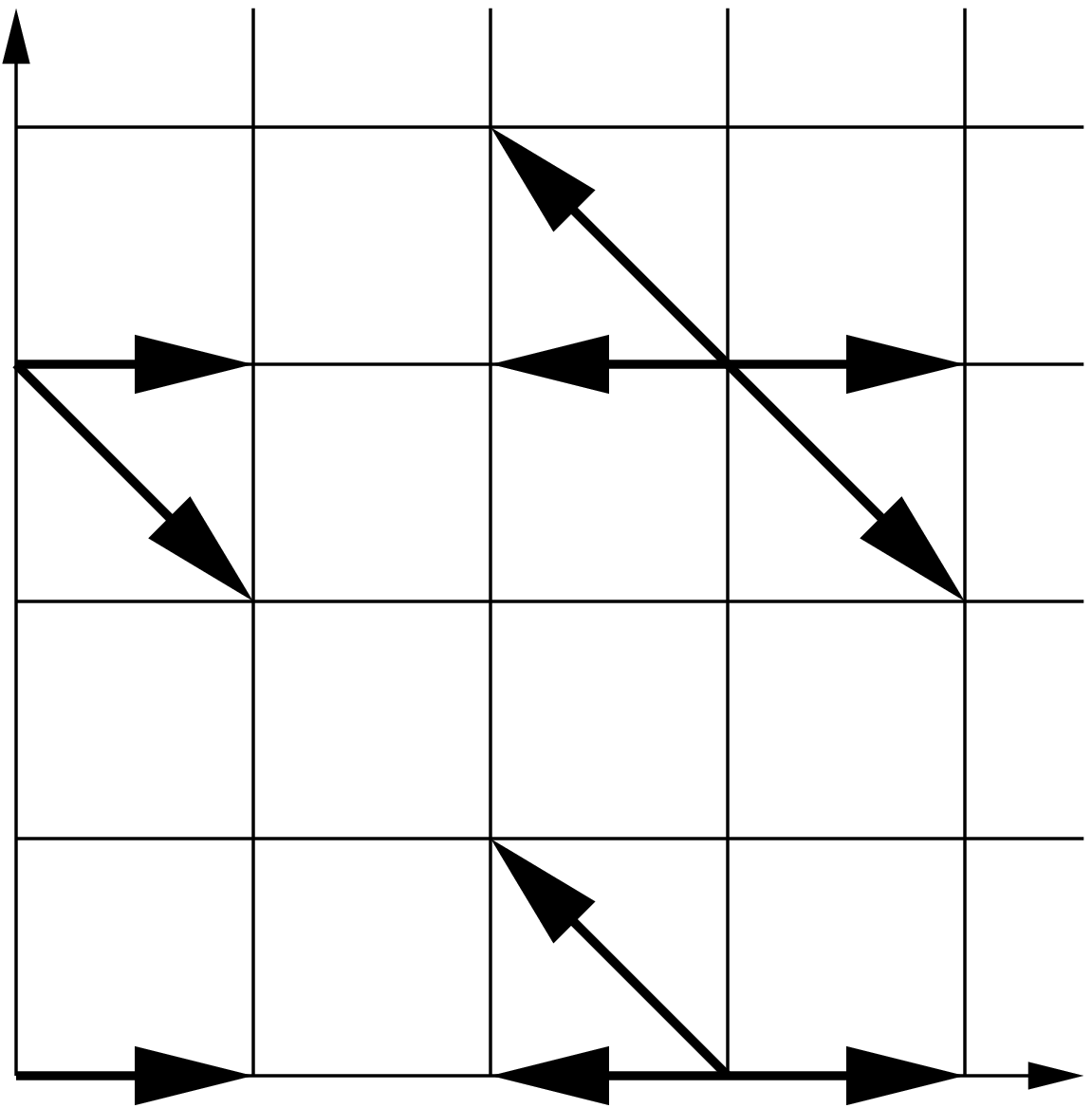}
\end{picture}
\end{center}
\caption{Three famous examples, known as Kreweras', Gessel's and Gouyou-Beauchamps' walks, respectively}
\label{ExExEx}
\end{figure}


Define the \emph{covariance} of the model as
     \begin{equation}
     \label{def_covariance}
          \textstyle \sum_{(i,j)\in\mathcal{S}} i j-[\sum_{(i,j)\in\mathcal{S}} i ][\sum_{(i,j)\in\mathcal{S}} j]=\sum_{(i,j)\in\mathcal{S}} i j.
     \end{equation}
The equality above follows from the fact that for each of the $23$ models with a finite group, $\sum_{(i,j)\in\mathcal{S}} i =0$ or $\sum_{(i,j)\in\mathcal{S}} j=0$, see \cite{BMM}. Lemma \ref{lemma_omega_2_3} below is proved in \cite[Proposition 5]{Ra}.
\begin{lem}
\label{lemma_omega_2_3} For all $23$ models with finite group
\eqref{group}, $\omega_2/\omega_3$ is rational and independent of
$z$. More precisely:
\begin{enumerate}
     \item[---] For the walks with a group of order $4$, $\omega_2/\omega_3=2$;
     \item[---] For the walks with a group of order $6$ and such that the covariance is negative
     (resp.\ positive),  $\omega_2/\omega_3=3$ (resp.\ $3/2$);
     \item[---] For the walks with a group of order $8$ and a negative
     (resp.\ positive) covariance,  $\omega_2/\omega_3=4$ (resp.\ $4/3$).
\end{enumerate}
\end{lem}

In the sequel, we note $\omega_2/\omega_3=k/\ell$; then, $2k$ is the order of the group. Since $k\omega_3=\ell\omega_2$, we obviously always have
     \begin{equation*}
          r_x(\omega +\ell\omega_2)-r_x(\omega)=
          \sum_{1\leq m\leq k}
          r_x(\omega +m \omega_3)-r_x(\omega+(m-1)\omega_3).
     \end{equation*}
It follows from \eqref{cont} and from properties \eqref{hatx}, \eqref{O3} and \eqref{hat} of the Galois automorphisms that
     \begin{align}
          r_x(\omega +\ell\omega_2)-r_x(\omega)=&\sum_{1\leq m\leq k} (xy) (\omega +m \omega_3)-(xy) (\widehat \eta (\o+m\o_3))\nonumber\\
          =&\sum_{1\leq m\leq k} (xy)((\widehat \eta \widehat \xi)^m \o)-(xy)(\widehat\xi (\widehat \eta \widehat \xi)^{m-1}\o))\nonumber\\
          =&\sum_{\theta \in\langle \widehat\xi,\widehat\eta\rangle}(-1)^{\theta} xy(\theta(\omega)),\label{key_equation}
     \end{align}
{where $(-1)^{\theta}$ is the signature of $\theta$; in other words, $(-1)^{\theta}=(-1)^{\ell(\theta)}$, where $\ell(\theta)$ is the length of $\theta$, i.e., the smallest $\ell$ such that we can write $\theta = \theta_1\circ \cdots \circ \theta_\ell$, with $\theta_1,\ldots,\theta_\ell$ equal to $\widehat\xi$ or $\widehat\eta$.}
The same identity with the opposite sign holds for $r_y$. The quantity \eqref{key_equation} is the {\it orbit-sum} of the
function $xy$ under the group $\langle
\widehat\xi,\widehat\eta\rangle$, and is denoted by
$\mathscr{O}(\omega)$. It satisfies the property hereunder, which is
proved in \cite{BMM}.
\begin{lem}
In the finite group case, the orbit-sum $\mathscr{O}(\omega)$ is identically zero if and only if the covariance \eqref{def_covariance} is positive.
\end{lem}
We therefore come to the following corollary.
\begin{cor}
In the finite group case, the functions $x \mapsto Q(x,0)$ and $y \mapsto Q(0,y)$ have a finite number of different branches if and only if the covariance \eqref{def_covariance} is positive.
\end{cor}

After the lifting to the universal covering done in Theorem \ref{thm2}, results of \cite{BK2, BMM} concerning the nature of the functions $x\mapsto Q(x,0)$ and $y \mapsto Q(0,y)$ in all finite group cases can now be established by very short reasonings. For the sake of completeness, we show how this works.

\begin{prop}[\cite{BK2,BMM}]
\label{prop_zero_orbit-sum} For all models with a finite group and a
positive covariance \eqref{def_covariance}, $x\mapsto Q(x,0)$ and
$y\mapsto Q(0,y)$ are algebraic.
\end{prop}
\begin{prop}[\cite{BMM}]
\label{prop_non-zero_orbit-sum} For all models with a finite group
and a negative or zero covariance \eqref{def_covariance}, $x\mapsto
Q(x,0)$ and $y\mapsto Q(0,y)$ are holonomic and non-algebraic.
\end{prop}

Proofs of both of these propositions involve the following lemma.
\begin{lem}
\label{lemma_properties_wp} Let ${\wp}$ be a Weierstrass elliptic
function with certain periods $\overline{\omega},\widehat{\omega}$.
\begin{enumerate}[label={\rm (P\arabic{*})},ref={\rm (P\arabic{*})}]
\item \label{differential_equation} We have
      \begin{equation*}
          \wp'(\omega)^{2}=4[\wp(\omega)-\wp(\overline{\omega}/2)]
                            [\wp(\omega)-\wp([\overline{\omega}+\widehat{\omega}]/2)]
                            [\wp(\omega)-\wp(\widehat{\omega}/2)],\qquad  \forall \omega\in{\bf C}.
     \end{equation*}
\item \label{principle_transformation} Let $p$ be some positive integer. The Weierstrass elliptic function with periods $\overline{\omega},\widehat{\omega}/p$ can be written in terms of ${\wp}$ as
     \begin{equation*}
          {\wp}(\omega)+\sum_{\ell=1}^{p-1}[{\wp}(\omega+\ell\widehat{\omega}/p)-{\wp}(\ell\widehat{\omega}/p)],\qquad  \forall \omega\in{\bf C}.
     \end{equation*}
\item \label{addition_theorem} We have the addition theorem:
     \begin{equation*}
          \wp({\omega}+\widetilde{\omega})=-\wp({\omega})
          -\wp(\widetilde{\omega})+\frac{1}{4}\left[\frac{\wp'({\omega})-
          \wp'(\widetilde{\omega})}{\wp({\omega})-
          \wp(\widetilde{\omega})}\right]^{2},\qquad  \forall \omega,\widetilde{\omega}\in{\bf C}.
     \end{equation*}
\item \label{elliptic_field} For any elliptic function $f$ with periods $\overline{\omega},\widehat{\omega}$, there exist two rational functions $R$ and $S$ such that
     \begin{equation*}
          f(\omega)=R(\wp(\omega))+\wp'(\omega)S(\wp(\omega)),\qquad  \forall \omega\in{\bf C}.
     \end{equation*}
\item \label{zeta_elliptic_function} There exists a function $\Phi$ which is $\overline{\omega}$-periodic and such that $\Phi(\omega+\widehat{\omega})=\Phi(\omega)-1$, $\forall \omega\in{\bf C}$.
\end{enumerate}
\end{lem}
\begin{proof}
Properties \ref{differential_equation}, \ref{addition_theorem} and
\ref{elliptic_field} are most classical, and can be found, e.g., in
\cite{JS,WW}. For \ref{principle_transformation} we refer to
\cite[page 456]{WW}, and for \ref{zeta_elliptic_function}, see
\cite[Equation (4.3.7)]{FIM}. Note that the function $\Phi$ in
\ref{zeta_elliptic_function} can be constructed via the zeta
function of Weierstrass.
\end{proof}

\begin{proof}[Proof of Proposition \ref{prop_zero_orbit-sum}]
If the orbit-sum $\mathscr{O}(\omega)$ is zero, Equation
\eqref{key_equation} implies that $r_x(\omega)$ is
$\ell\omega_2$-periodic. In particular, the property
\ref{elliptic_field} of Lemma \ref{lemma_properties_wp} entails that
there exist two rational functions $R$ and $S$ such that
     \begin{equation}
     \label{q_x_algebraic}
          r_x(\omega)=R(\wp(\omega;\omega_1,\ell\omega_2))+\wp'(\omega;\omega_1,\ell\omega_2)S(\wp(\omega;\omega_1,\ell\omega_2)).
     \end{equation}
Further, the property \ref{principle_transformation} together with
the addition formula \ref{addition_theorem} of Lemma
\ref{lemma_properties_wp} gives that
$\wp(\omega;\omega_1,\ell\omega_2)$ is an algebraic function of
$\wp(\omega)$---we recall that $\wp(\omega)$ denotes the Weierstrass
function $\wp(\omega;\omega_1,\omega_2)$. Due to Lemma
\ref{lemma_properties_wp} \ref{differential_equation},
$\wp'(\omega)$ is an algebraic function of $\wp(\omega)$ too, so
that $\wp'(\omega;\omega_1,\ell\omega_2)$ is also an algebraic
function of $\wp(\omega)$. Thanks to \eqref{q_x_algebraic}, we get
that $r_x(\omega)$ is algebraic in $\wp(\omega)$. Since
$\wp(\omega)$ is a rational function of $x(\omega)$, see
\eqref{expression_uniformization}, we finally obtain that
$r_x(\omega)$ is algebraic in $x(\omega)$.
  Then $q_x(\o)=r_x(\o)/K(x(\o), 0)$ is algebraic in $x(\o)$,
  and so is $q_y(\o)$ in $y(\o)$.
\end{proof}

\begin{proof}[Proof of Proposition \ref{prop_non-zero_orbit-sum}]
In this proof we have $\ell=1$, see Lemma \ref{lemma_omega_2_3}.
Thanks to Lemma \ref{lemma_properties_wp}
\ref{zeta_elliptic_function}, there exists a function $\Phi$ which
is $\omega_1$-periodic and such that
$\Phi(\omega+\omega_2)=\Phi(\omega)-1$. In particular, transforming
\eqref{key_equation} we can write
     \begin{equation*}
          r_x(\omega+\omega_2)+\Phi(\omega+\omega_2)\mathscr{O}(\omega+\omega_2)=r_x(\omega)+\Phi(\omega)\mathscr{O}(\omega).
     \end{equation*}
This entails that $r_x(\omega)+\Phi(\omega)\mathscr{O}(\omega)$ is
elliptic with periods $\omega_1,\omega_2$. In particular, for the
same reasons as in the proof of Proposition
\ref{prop_zero_orbit-sum}, this is an algebraic function of
$x(\omega)$. The function $\mathscr{O}(\omega)$ is obviously also
algebraic in $x(\omega)$. As for the function $\Phi(\omega)$, it is
proved in \cite[page 71]{FIM} that it is a non-algebraic function
of $x(\omega)$. Moreover, it is shown in \cite[Lemma 2.1]{FR} that
it is holonomic in $x(\omega)$. Hence $r_x(\o)$ is holonomic in
$x(\o)$ but not algebraic.  The same is true for
$q_x(\o)=r_x(\o)/K(x(\o),0)$ in $x(\o)$ and  for $q_y(\o)$ in
$y(\o)$.
\end{proof}


\section{Infinite group case}
\label{Section_infinite_group}
\setcounter{equation}{0}

  It has been shown in \cite[Part 6.2]{Ra} that for all $51$ models
   with infinite group, $\o_2/\o_3$ takes irrational values for
   infinitely many $z$. The next proposition states a more complete
   result.

\begin{prop}
\label{proir} For all $51$ walks with infinite group, the
 sets $\mathcal{H}=\{z \in ]0, 1/|\mathcal{S}|[: \o_2/\o_3 \hbox{ is
 rational}\}$ and $]0, 1/|\mathcal{S}|[\setminus \mathcal{H}=
  \{z \in ]0, 1/|\mathcal{S}|[: \o_2/\o_3 \hbox{ is
 irrational}\}$ are dense in
  $]0, 1/|\mathcal{S}|[$.
\end{prop}

\begin{proof}
 {The function $\o_2/\omega_3$ is clearly real continuous function on $]0,
  1/|\mathcal{S}|[$.  In fact, it has been noticed in \cite[Part 6.2]{Ra} that the function
$\o_2/\o_3$ is expandable in power series in a neighborhood of any
point of the interval $]0, 1/|\mathcal{S}|[$.
   Thus it suffices to find just one segment within $]0,
  1/|\mathcal{S}|[$ where this function is not constant.}

  Proposition~\ref{nncc} below gives the asymptotic of $\o_2/\o_3$
  as  $z\to 0$: for any of $51$ models there exist some rational $L>0$
     and some $\widetilde L\ne 0$ such that as $z>0$ goes to $0$,
\begin{equation*}
     \o_2/\o_3=L+\widetilde L/\ln z+ O((1/\ln z)^2).
     \end{equation*}
  {This immediately implies that} this function
    is not constant on a small enough interval in a right neighborhood of
    $0$ and concludes the proof.

 Note however that there is another way to conclude the proof that does not need the full power of Proposition~\ref{nncc}: it is enough to show {(as done in the proof of Proposition~\ref{nncc})} that $\o_2/\o_3$ converges to a rational positive constant $L$ as $z \to 0$ for all $51$ models. {Indeed, then, since $\o_2/\o_3$ necessarily takes irrational values for some $z\in]0, 1/|\mathcal{S}|[$ (see \cite[Part 6.2]{Ra}), there exists an interval within $]0, 1/|\mathcal{S}|[$ where the ratio $\o_2/\o_3$ is not constant.}
\end{proof}

In  Subsections \ref{subsection71}, \ref{subsection72} and \ref{subsection73} we thoroughly analyze the branches of
$x\mapsto Q(x,0)$ and $y\mapsto Q(0,y)$ for $z$ such that
$\o_2/\o_3$ is irrational, and we prove in particular that their set of
poles is infinite and dense on the curves given on Figure \ref{CC},
see Theorem \ref{main_tt}.
   Then the following corollary is immediate.
\begin{cor}
\label{cor-rnc}
  Let  $\mathcal{H}=\{z \in ]0, 1/|\mathcal{S}|[: \o_2/\o_3 \hbox{ is
 rational}\}$ and $]0, 1/|\mathcal{S}|[\setminus \mathcal{H}=
  \{z \in ]0, 1/|\mathcal{S}|[: \o_2/\o_3 \hbox{ is
 irrational}\}$.
     \begin{enumerate}[label={\rm (\roman{*})},ref={\rm (\roman{*})}]
     \item \label{cor-rnc1} For all $z\in \mathcal H$, $x \mapsto Q(x,0)$ and $y \mapsto Q(0,y)$ are holonomic;
     \item\label{cor-rnc2} For all $z\in \mathcal ]0, 1/|\mathcal{S}|[\setminus \mathcal H$, $x \mapsto Q(x,0)$ and $y \mapsto Q(0,y)$ are non-holonomic.
     \end{enumerate}
     \end{cor}
\begin{proof}
The statement \ref{cor-rnc1} follows from Theorem \ref{ratho}, and \ref{cor-rnc2} comes from Theorem \ref{main_tt} \ref{inet} below as explained in the Introduction.
\end{proof}

{
\begin{rem}
\label{Remark2}
   It follows from Remark \ref{Remark1} that
$\mathcal{H}$ can be characterized as the set of $z \in ]0,
1/|\mathcal{S}|[$ such that the group $\langle \xi, \eta \rangle$
restricted to the curve $\{(x,y)\in({\bf C}\cup\{\infty\})^2 : K(x,y;z)=0\}$
 is finite. Then methods developed in \cite[Chapter 4]{FIM}
    specifically for the finite group case should be efficient for
      further analysis of $x\mapsto Q(x,0)$ and $y \mapsto Q(0,y)$
      for any fixed $z \in \mathcal{H}$.
  \end{rem} }

 The analysis of the poles
being rather technical, we start first with an informal study.

%

\subsection{Poles of the set of branches of $x \mapsto Q(x,0)$ and
  $y \mapsto Q(y,0)$ for irrational $\o_2/\o_3$: an informal study}
  \label{subsection71}

Let us fix $z \in ]0, 1/|\mathcal{S}|[$ such that $\o_2/\o_3$
is irrational. We first informally explain why the set of poles of
all branches of $x\mapsto Q(x,0)$ and $y \mapsto Q(0,y)$ could be
dense on certain curves in this case. We shall denote by $\Re
\omega$ and $\Im \omega$ the real and imaginary parts of
$\omega\in{\bf C}$, respectively. Let $\Pi_x$ and $\Pi_y$ be the parallelograms
defined by
     \begin{equation}
     \label{def_Pi_x_y}
          \Pi_x = \mathscr{M}_{0,0} \cup \mathscr{M}_{0,1}=\omega_1[0,1[+\omega_2[0,1[,
          \qquad
          \Pi_y = \mathscr{N}_{0,0} \cup \mathscr{N}_{0,1}=\omega_3/2+\omega_1[0,1[+\omega_2]0,1],
     \end{equation}
with  notations \eqref{mkl} and \eqref{nkl}.
 Function $r_x(\o)$ (resp.\ $r_y(\o)$)
 on $\Pi_x$ (resp.\ $\Pi_y$) defines the first
  (main) branch of $x \mapsto  Q(x,0)$
   (resp.\ $y \mapsto  Q(0,y)$) twice via \eqref{branches} (resp.\ \eqref{branches1}).

Denote by $f_y(\o)=x(\o)[y(-\o+2\o_{x_1})-y( \o)]$ the function used in the meromorphic
 continuation procedure \eqref{cont1}.
  Assume that at some $\o_0 \in \Pi_y$, $r_y(\o_0)\ne \infty$ and $f_y(\o_0)=\infty$.
  Further, suppose that
\begin{equation}
\label{bbb}
\nexists \o \in \Pi_y:\quad \Im \o=\Im \o_0,\quad
 f_y(\o)=\infty.
 \end{equation}
   By \eqref{cont1}, for any $n\geq 1$ we have
   \begin{equation}
   \label{qdf1}
   r_y(\o_0+n\o_3)=r_y(\o_0)+ f_y(\o_0)+\sum_{k=1}^{n-1}f_y(\o_0+k
   \o_3).
  \end{equation}
     We have $r_y(\o_0)+ f_y(\o_0)=\infty$ by our assumptions.
   If $\o_2/\o_3$ is irrational, then for any $k\geq 1$
 there is no $p \in {\bf Z}$ such that $\o_0+k \o_3=\o_0+p\o_2$.
 Function $f_y$ being $\o_2$-periodic, it follows from this
 fact and assumption \eqref{bbb}
 that $f_y(\o_0+k\o_3)\ne \infty$ for any $k \geq 1$.  Hence
 by (\ref{qdf1}), $r_y(\o_0+n\o_3)=\infty$ for all $n \geq 1$.
  Due to irrationality of $\o_2/\o_3$,
  for any $n\geq 1$ there exists a unique $\o_n(\o_0) \in
    \Pi_y$ and
    $p \in {\bf Z}$ such that $\o_0+n \o_3=\o_n(\o_0)+p\o_2$,
  and the set $\{\o_n(\o_0)\}_{n\geq 1}$ is dense on the curve
  \begin{equation}
  \label{iy}
   \mathcal{I}_y(\o_0)=y(\{\o: \Im \o=\Im \o_0,\,  \o \in
   \Pi_y\})\subset {\bf C} \cup \{\infty\}.
   \end{equation}
  By definition \eqref{branches1}, the set of poles of all
   branches of $y \mapsto  K(0,y)Q(0,y)$ is dense on the curve
   $\mathcal{I}_y(\o_0)$. The number of zeros of $y \mapsto K(0,y)$ being
   at most two, the same conclusion holds true for $y \mapsto Q(0,y)$.

\medskip

Let us now identify the points $\o_0$ in $\Pi_y$ where $f_y(\o_0)$
is infinite. They are (at most) six such points $a_1,a_2,a_3,a_4,
b_1,b_2 \in \Pi_y$, which correspond to the following pairs
$(x(\o_0), y(\o_0))$:
     \begin{equation}
     \label{aaaabb}
          a_1= (x^\s, \infty),
          \ a_4=(x^\s, y^\s),
          \ a_2=(x^\ss, \infty),
          \ a_3=(x^\ss,y^\ss),
          \ b_1=(\infty, y^\d),
          \ b_2=(\infty, y^\dd).
     \end{equation}
Here by \eqref{def_XX} and \eqref{def_YY}
\begin{equation*}
\begin{array}{ccccccc}
 x^\s\hspace{-1.5mm}&\hspace{-1.5mm}=\hspace{-1.5mm}&\hspace{-1.5mm}
 \displaystyle \lim_{y \to \infty} \frac{
-\widetilde b(y)+[\widetilde b(y)^2-4\widetilde a(y)\widetilde
c(y)]^{1/2}}{2\widetilde a(y)},&\ & x^\ss
\hspace{-1.5mm}&\hspace{-1.5mm}
=\hspace{-1.5mm}&\hspace{-1.5mm}\displaystyle\lim_{y \to \infty}
\frac{ -\widetilde b(y)-[\widetilde b(y)^2-4\widetilde
a(y)\widetilde c(y)]^{1/2}}{2\widetilde a(y)},

\\

y^\s\hspace{-1.5mm}&\hspace{-1.5mm}=\hspace{-1.5mm}&\hspace{-1.5mm}\displaystyle\lim_{x\to
x^\s}\frac{ - b(x)+ [b(x)^2-4 a(x) c(x)]^{1/2}}{2 a(x)},&\ &
y^\ss\hspace{-1.5mm}&\hspace{-1.5mm}
=\hspace{-1.5mm}&\hspace{-1.5mm}\displaystyle\lim_{x\to x^\ss}\frac{
- b(x)+[ b(x)^2-4 a(x) c(x)]^{1/2}}{2 a(x)},

\\

y^\d\hspace{-1.5mm}&\hspace{-1.5mm}=\hspace{-1.5mm}&\hspace{-1.5mm}\displaystyle\lim_{x\to
\infty}\frac{ - b(x)+[b(x)^2-4 a(x) c(x)]^{1/2}}{2 a(x)},&\ &
y^\dd\hspace{-1.5mm}&\hspace{-1.5mm}=\hspace{-1.5mm}&\hspace{-1.5mm}\displaystyle\lim_{x\to
\infty}\frac{ - b(x)-[ b(x)^2-4 a(x) c(x)]^{1/2}}{2 a(x)}.
\end{array}
\end{equation*}
where $a,b,c, \widetilde a, \widetilde b, \widetilde c$ are
introduced in \eqref{kernel_pt}.

For most of $51$ models of walks, assumption \eqref{bbb} holds true
for none of these points, so that the previous reasoning does not
work: some poles of $f_y$ could be compensated in the sum
(\ref{qdf1}).
 Furthermore, it may happen for some of these points
  that not only $f_y(\o_0)=\infty$ but also
 $r_y(\o_0)=\infty$, and consequently $f_y(\o)+r_y(\o)$
     may have no pole at $\o=\o_0$.
 For these reasons we need to inspect {more closely} the location
 of these six points for each of the $51$ models and their contribution
 to the set of poles via (\ref{qdf1}).

\subsection{Functions $x \mapsto Q(x,0)$ and $y \mapsto Q(0,y)$
 for irrational $\o_2/\o_3$}
 \label{subsection72}

 In addition to the notation \eqref{iy}, define the curve
\begin{equation}
  \label{ix}
   \mathcal{I}_x(\o_0)=x(\{\o: \Im \o=\Im \o_0,\,\o \in \Pi_x\})
   \subset {\bf C} \cup \{\infty\}.
   \end{equation}
We now formulate the main theorem of this section.
\begin{thm}
\label{main_tt} For all $51$ non-singular walks with infinite group
\eqref{group} given on Figure \ref{Allcases} and any $z$ such that
$\o_2/\o_3$ is irrational, the following statements hold.
\begin{enumerate}[label={\rm (\roman{*})},ref={\rm (\roman{*})}]
\item \label{onsi} The only singularities on ${\bf C}$
of the first branch of $x \mapsto Q(x,0)$ (resp.\ $y \mapsto Q(0,y)$)
are the branch points $x_3$ and $x_4$ (resp.\ $y_3$ and $y_4$).


\item \label{ebfp}
Each branch of $x \mapsto Q(x,0)$ (resp.\ $y\mapsto Q(0,y)$)
is meromorphic on ${\bf C}$ with a finite number of poles.

\item \label{inet}
The set of poles on ${\bf C}$ of all branches of $x \mapsto Q(x,0)$
(resp.\ $y \mapsto Q(0,y)$) is infinite. With the notations
(\ref{ix}), (\ref{iy}) above
  and points $a_1,b_1$ defined in
\eqref{aaaabb}, it is dense on the following curves (see Figure
\ref{CC}):

\begin{enumerate}[label={\rm (iii.\alph{*})},ref={\rm (iii.\alph{*})}]
\item \label{inet1}
 For the walks of Subcase I.A in Figure \ref{Allcases},:
 $\mathcal{I}_x(a_1)$ and $\mathcal{I}_x(b_1)$ for $x \mapsto
Q(x,0)$; $\mathcal I_y(a_1)$ and $\mathcal I_y(b_1)$ for~$y \mapsto
Q(0,y)$.

\item \label{inet2}
For the walks of Subcases I.B and I.C in Figure \ref{Allcases}:
   $\mathcal{I}_x(a_1)$ and ${\bf R} \setminus ]x_1,x_4[$
for $x \mapsto Q(x,0)$;  $\mathcal{I}_y(a_1)$ and $[y_4, y_1]$
for $y \mapsto Q(0,y)$.

\item \label{inet3}
For the walks of Subcase II.A in Figure \ref{Allcases}:
$\mathcal{I}_x(b_1)$ and $[x_4,x_1]$ for $x \mapsto Q(x,0)$;
$\mathcal{I}_y(b_1)$ and ${\bf R} \setminus ]y_1,y_4[$ for $y
\mapsto Q(0,y)$.

\item \label{inet4}
For the walks of Subcases II.B, II.C, II.D
and Case III in Figure \ref{Allcases}:
  ${\bf R} \setminus ]x_1,x_4[$ for $x \mapsto Q(x,0)$;  ${\bf R} \setminus ]y_1,y_4[$ for $y
\mapsto Q(0,y)$.
\end{enumerate}

\item \label{ipoles} Poles of branches of
 $x \mapsto Q(x,0)$ and $y\mapsto Q(0,y)$ out of these curves
  may be only at zeros of $K(x,0)$ and $K(0,y)$, respectively.
\end{enumerate}
\end{thm}



Before giving the proof of Theorem \ref{main_tt}, we need to
introduce some additional tools. If the value of $\omega_2/\o_3$ is
irrational, for any $\o_0 \in {\bf C}$ and any $n \in {\bf Z}_+$,
there exists a unique $\o_n^y(\o_0) \in\Pi_y$  (resp.\ $\o_n^x(\o_0)
\in\Pi_x$) as well as a unique number $p_y \in {\bf Z}$ (resp.\ $p_x
\in {\bf Z}$) such that $\o_0+n\o_3=p_y\o_2+\o_n^y(\o_0)$ (resp.\
$\o_0+n\o_3=p_x\o_2+\o_n^x(\o_0)$). With these notations we can
state the following lemma.

\begin{lem}
\label{vspom}
Let $z$ be such that $\omega_2/\o_3$ is irrational.
\begin{enumerate}[label={\rm (\alph{*})},ref={\rm (\alph{*})}]
\item \label{AAA} For all $n \neq  m$,
we have $\o_n^x(\o_0) \neq \o_m^x(\o_0)$ and $\o_n^y(\o_0) \neq \o_m^y(\o_0)$.
\item \label{BBB} The set $\{\o_n^x(\o_0)\}_{n  \in {\bf Z}_+}$ (resp.\ $\{\o_n^y(\o_0)\}_{n  \in {\bf Z}_+}$) is dense on the segment $\{\o \in \Pi_x : \Im\o=\Im \o_0\}$ (resp.\ $\{\o \in \Pi_y : \Im\o=\Im \o_0\}$).
\end{enumerate}
\end{lem}

\begin{proof}
Both \ref{AAA} and \ref{BBB} are direct consequences of the irrationality of $\omega_2/\o_3$.
\end{proof}

In the next definition, we introduce a partial order in $\Pi_y$.

\begin{defn}
For any $\o,\o' \in \Pi_y$, we write $\o \ll \o'$ if for some $n\in
{\bf Z}_+$ and some $p \in {\bf Z}$, $\o+n\o_3=\o'+p\o_2$.
\end{defn}
If $\o \ll \o'$ (and if $\o_2/\o_3$ is irrational), both $n$ and $p$
are unique and sometimes we shall write $\o \ll_n \o'$. In
particular, for any $\o \in \Pi_y$, we have $\o \ll \o$, since
$\o\ll_0\o$.

\begin{defn}
\label{def_equiv}
If either $\o \ll \o'$ or $\o \ll \o'$,
we say that $\o$ and $\o'$ are ordered, and we write $\o \sim \o'$.
\end{defn}

Let us denote by $f_x$ and $f_y$ the (meromorphic) functions used in the
meromorphic~continuation procedures \eqref{cont} and \eqref{cont1}, namely, by using \eqref{hatx}:
     \begin{equation*}
          f_x(\o)=y(\omega)[x(\widehat \eta \o)-x(\omega)],
          \qquad f_y(\o)=x(\omega)[y(\widehat \xi \o)-y(\omega)].
     \end{equation*}
 The following lemma will be the key tool for the proof of
 Theorem~\ref{main_tt}.

 \begin{figure*}[t]
\begin{center}
{\rm Theorem \ref{main_tt} \ref{inet1}, i.e., Subcase I.A }
\end{center}
\begin{center}
\begin{picture}(00.00,732.00)
\hspace{-105mm}
\includegraphics{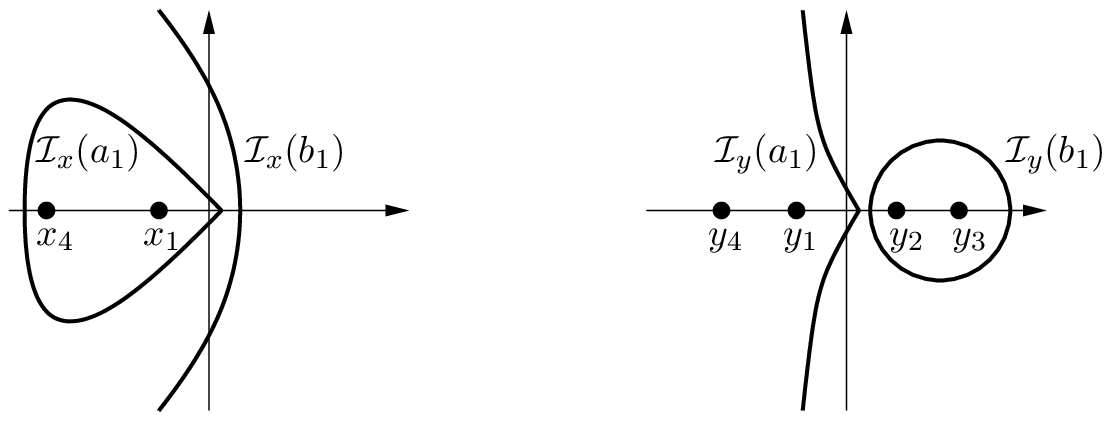}
\end{picture}
\vspace{-223.5mm}


\end{center}
\end{figure*}

\begin{figure*}[t]
\begin{center}
{\rm Theorem \ref{main_tt} \ref{inet2}, i.e., Subcases I.B and I.C}
\end{center}
\begin{center}
\begin{picture}(00.00,732.00)
\hspace{-105mm}
\includegraphics{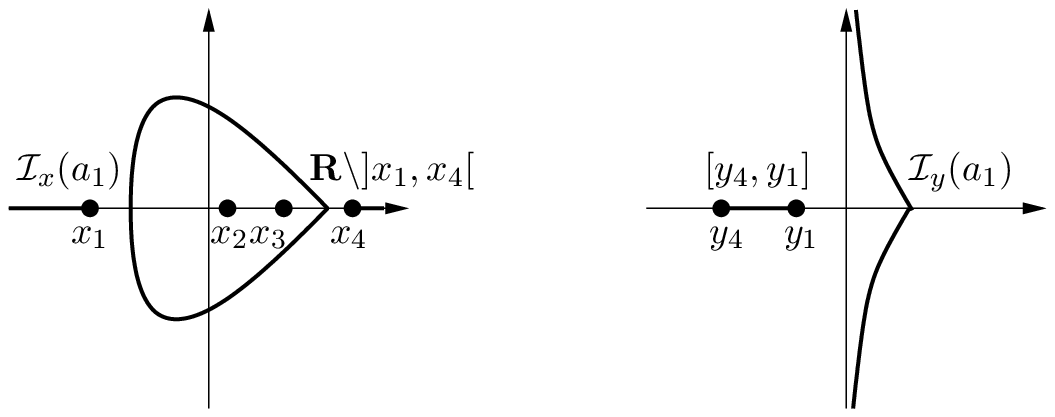}
\end{picture}
\vspace{-223.5mm}


\end{center}
\end{figure*}

\begin{figure*}[t]
\begin{center}
{\rm Theorem \ref{main_tt} \ref{inet3}, i.e., Subcase II.A}
\end{center}
\begin{center}
\begin{picture}(00.00,732.00)
\hspace{-105mm}
\includegraphics{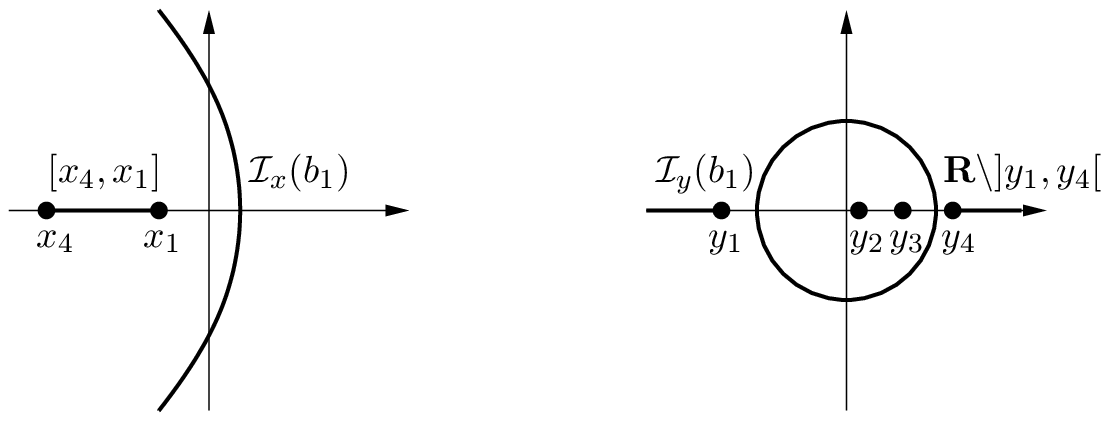}
\end{picture}
\vspace{-223.5mm}


\end{center}
\end{figure*}

\begin{figure}[t]
\begin{center}
{\rm Theorem \ref{main_tt} \ref{inet4}, i.e., Subcases II.B, II.C,
II.D and Case III}
\end{center}
\begin{center}
\begin{picture}(00.00,732.00)
\hspace{-105mm}
\includegraphics{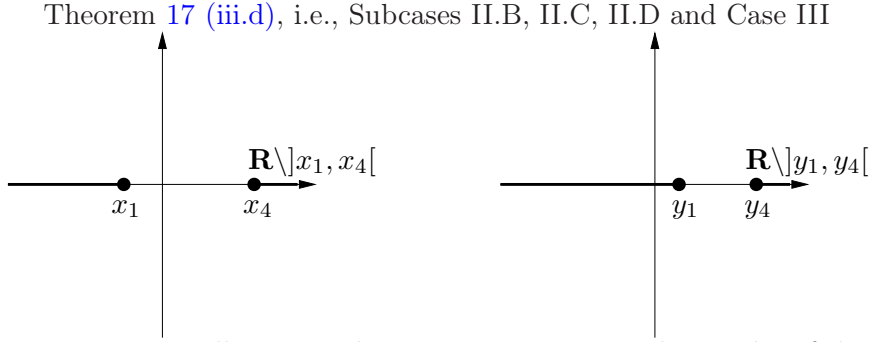}
\end{picture}
\vspace{-222.5mm}


\end{center}
\caption{For walks pictured on Figure \ref{Allcases},
curves where
poles of the set of branches of $x\mapsto Q(x,0)$ and $y\mapsto Q(0,y)$
 are dense}
\label{CC}
\end{figure}

\begin{lem}
\label{dop} Let $z$ be such that $\omega_2/\o_3$ is irrational; let $\o_0 \in \Pi_y$ be such that $r_y(\o_0)\neq\infty$,
and let
     \begin{equation*}
          \mathcal{A}(\o_0)=\{\o \in \Pi_y: \Im \o=\Im \o_0,\, f_y(\o)=\infty\}.
     \end{equation*}
Assume that $\o_0 \in \mathcal{ A}(\o_0)$  and  that for some $\o^1,
\ldots, \o^k \in \mathcal{A}(\o_0)$:
\begin{enumerate}[label={\rm (\Alph{*})},ref={\rm (\Alph{*})}]
\item \label{(i)}  $\o_0 \ll_{n_1} \o^1 \ll_{n_2} \cdots \ll_{n_k} \o^k$;
\item \label{(ii)} $\lim_{\o \to \o_0} \{ f_y(\o)+f_y(\o+n_1 \o_3)+ f_y(\o+ n_2 \o_3)
+\cdots +f_y(\o+ n_k \o_k)\} =\infty$;
\item \label{(iii)} there is no other $\o \in \mathcal{A}(\o_0)$ such that $\o_0 \ll \o$.
\end{enumerate}
   Then
the set of poles of all branches of  $x \mapsto Q(x,0)$ (resp.\ $y\mapsto Q(0,y)$) is dense on the curve
  $\mathcal{I}_x(\o_0)$ (resp.\ $\mathcal{I}_y(\o_0)$)
  defined in \eqref{ix} (resp.\ \eqref{iy}).
\end{lem}

\begin{proof}
By Equation \eqref{cont1} of Theorem \ref{thm2}, we have, for any $n\in{\bf Z}_+$ and any $\o \in \Pi_y$,
     \begin{equation}
     \label{n_it}
          r_y(\o+n\o_3)=r_y(\o)+f_y(\o)+f_y(\o+ \o_3)+ f_y(\o+2\o_3)+\cdots +f_y(\o_0+(n-1)\o_3).
     \end{equation}
Let $\o_0$ be as in the statement of Lemma \ref{dop}.
 Due to assumption \ref{(iii)}, Lemma \ref{vspom} \ref{AAA}
   and the $\o_2$-periodicity of $f_y$,
  the set $\{\o_0 + n \o_3\}_{n>n_1+\cdots +n_k}$ does
  not contain any point $\o$ where $f_y(\o)=\infty$.
   Further, by the assumptions \ref{(i)} and \ref{(iii)},
    Lemma \ref{vspom} \ref{AAA} and also by the $\o_2$-periodicity of
    $f_y$, the set
    $\{\o_0 + n \o_3\}_{0 \leq n\leq n_1+\cdots +n_k}$
   contains exactly  $k+1$ poles of $f_y$ that are
     $\o_0, \o_0+n_1\o_3, \ldots, \o_0+n_k\o_3$.
    Then, by \eqref{n_it}, assumption \ref{(ii)} and the fact that
     $r_y(\o_0)\ne \infty$, we reach the conclusion
     that for any $n>n_1+\cdots+n_k$, the point $\o_0+n \o_3$
     is a pole of $r_y(\o)$.

Due to Equation \eqref{sqs2}, any $\o$ pole of $r_y$ such that $x(\o)y(\o) \neq \infty$
 is also a pole of $r_x$.
Define now $\mathcal{B}$, the set of (at most twelve) points
  in $\Pi_y$ where either $x(\o)=\infty$, $y(\o)=\infty$,
   $K(x(\o),0)=0$ or $K(0, y(\o))=0$.
  Introduce also $M=\max\{m\geq  0: \o_0 \ll_m \o \hbox{ for some }\o \in \mathcal{B}\}$---with the usual convention $M=-\infty$ if $\o_0 \ll \o$ for none $\o\in \mathcal{B}$.
   If $n >\max(M, n_1+\cdots +n_k)$, the points $\o_0+n \o_3$ are poles of $r_x$ as well, and both $K(x(\o_0+n\o_3),0)$ and $K(0,y(\o_0+n \o_3))$ are non-zero.
   By Lemma \ref{vspom} \ref{BBB}
   and definitions~\eqref{branches} and \eqref{branches1}, Lemma~\ref{dop}
    follows.
\end{proof}

\begin{figure}[t]
\begin{center}
\begin{picture}(00.00,725.00)
\hspace{-105mm}
\includegraphics{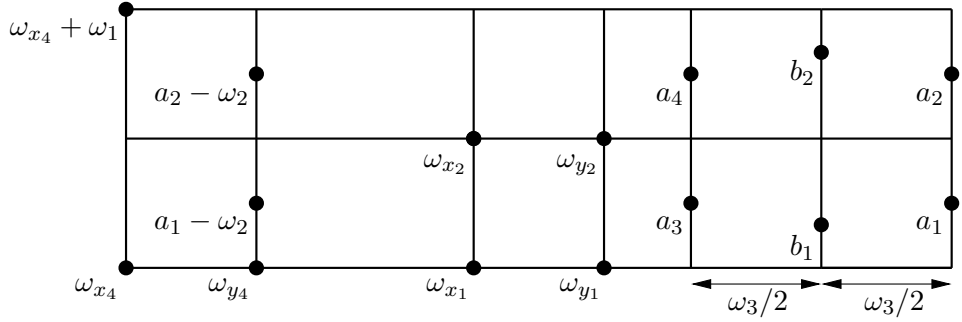}
\end{picture}
\end{center}
\vspace{-213mm} \caption{Location of $a_1,a_2,a_3,a_4, b_1,b_2$ if
$y_4<0$ and $x_4<0$, i.e.,~Subcase~I.A} \label{Locloc}
\end{figure}

 We are now ready to give the proof of Theorem \ref{main_tt}.

\begin{proof}[Proof of Theorem \ref{main_tt}]
Functions $f_y(\o)$ and $y(\o)$ being $\o_2$-periodic,
 it follows that both of them have no pole at
 any $\o$ with $0\leq \Im \o<\o_1$ and $\Im \o \notin \{\Im a_1, \Im a_2, \Im a_3, \Im a_4, \Im b_1, \Im b_2\}$. Then, by \eqref{cont1},
     \begin{equation}
     \label{xx}
          \forall \o \in\bigcup_{k=0}^{\infty} \mathscr{N}_{k,0} \ \hbox{with } \Im \o \ne\{\Im a_1, \Im a_2, \Im a_3, \Im a_4, \Im b_1, \Im b_2\},\quad r_y(\o) \ne \infty.
     \end{equation}
Function $x(\o)$ being $\o_2$-periodic, it has no pole at any $\o$ such that $0\leq \Im \o<\o_1$ and $\Im \o \notin \{\Im b_1, \Im b_2\}$. Then Equation \eqref{sqs2} and the fact that $ \cup_{k=1}^{\infty}\mathscr{M}_{k,0} \subset\cup_{k=0}^{\infty} \mathscr{N}_{k,0}$ imply that
\begin{equation}
\label{yy} \forall \o \in
\bigcup_{k=1}^{\infty} \mathscr{M}_{k,0}\ \hbox{with }\Im \o \ne
\{\Im a_1, \Im
 a_2, \Im a_3, \Im a_4, \Im b_1, \Im b_2\},\quad r_x(\o) \ne \infty.
\end{equation}
      In order to prove Theorem \ref{main_tt} \ref{onsi},
 we shall prove the following proposition.
\begin{prop}
\label{ph1}  For all  $51$ models,
   for any $\o \in \mathscr{M}_{1,0}$ (resp.\ $\o \in
\mathscr{N}_{1,0}$)  with $x(\o)\neq \infty$ (resp.\ $y(\o)\neq
\infty$) and
  $\Im \o \in \{\Im a_1, \Im
 a_2, \Im a_3, \Im a_4, \Im b_1, \Im b_2\}$,
 we have $r_x(\o)\neq \infty$ (resp.\ $r_y(\o) \neq \infty$).
\end{prop}
  The proof of this proposition is postponed to the next subsection.
 By this proposition, (\ref{xx}) and (\ref{yy}),
     the only singularities of the first branches of $K(x,0)Q(x,0)$
(resp.\ $K(0,y)Q(0,y)$)  may be only among the branch
      points $x_1, x_2, x_3, x_4$
 (resp.\ $y_1, y_2, y_3, y_4$).
 Let us recall that
       the function $x \mapsto Q(x,0)$ is initially defined as a series
   $\sum_{i, n \geq 0} q(i,0; n)x^{i}z^{n}$. The {elementary estimate} $\sum_{i \geq 0} q(i,0; n)\leq |\mathcal{S}|^n$
   implies that for any $z \in ]0, 1/|\mathcal{S}|[ $ and $x \in {\bf C}$ with $|x|\leq
1$ this series is absolutely is convergent. Since $|x_1|<1$,
$|x_2|<1$, and since also $K(x,0)$ is a polynomial with
      (at most two) roots
that are smaller or equal to $1$ by absolute value,
 the only singularities of the first branch of $x \mapsto Q(x,0)$ are
  the branch points $x_3$ and $x_4$, that are out of the unit disc.
     By the same arguments the analogous statement holds true for $y \mapsto Q(0,y)$.
 This finishes the proof of Theorem \ref{main_tt} \ref{onsi}.

 Since for any $p\in{\bf Z}_+$,
  there  exist only finitely many  $\o \in \cup_{k=1}^p\mathscr{M}_{k,0}$ (resp.\ $\o \in \cup_{\ell=1}^p\mathscr{N}_{0,\ell}$)
   where $f_x(\o)=\infty$
    (resp.\ $f_y(\o)=\infty$),
       Theorem \ref{main_tt} \ref{ebfp}
    immediately follows from the meromorphic continuation procedure
    of $r_x$ and $r_y$ done in Section \ref{meroun},
    namely Equations \eqref{cont} and \eqref{cont1} as well as the definitions
\eqref{branches} and \eqref{branches1}.

  The
following proposition proves Theorem \ref{main_tt} \ref{inet}.
\begin{prop}
\label{ph2}
  The poles of $x \mapsto  Q(x,0)$ (resp.\ $y \mapsto  Q(0,y)$)
  are dense on the six curves
$\mathcal{I}_x(\o_0)$ (resp.\ $\mathcal{I}_y(\o_0)$) with
$\o_0\in\{a_1,a_2,a_3,a_4,b_1,b_2\}$.
 For any of $51$ models, the set of these curves coincide with the one claimed in
 Theorem~\ref{main_tt}~\ref{inet}.
 \end{prop}
The proof of this proposition is postponed to the next subsection as
well, it will be based on Lemma \ref{dop} with $\o_0$ appropriately
chosen among $a_1,a_2,a_3,a_4,b_1,b_2$.

The last statement  \ref{ipoles} of the theorem follows immediately
from
  (\ref{xx}) and (\ref{yy}), Proposition \ref{ph2}
   and definitions \eqref{branches} and
\eqref{branches1}. \end{proof}

\subsection{Proof of Propositions \ref{ph1} and \protect\ref{ph2}}
\label{subsection73}

 To start with the proofs of Propositions \ref{ph1} and \ref{ph2},
we need to study closer the location of points (\ref{aaaabb})
$a_1,a_2,a_3,a_4,b_1,b_2$ on $ \Pi_y$. It depends heavily on the
signs of $x_4$ and $y_4$,
 see Figures \ref{Locloc}, \ref{LoclocICIIBC}, \ref{LoclocIIA},
 \ref{LoclocIID} and \ref{LoclocIII}.
  Let us recall that  $0<x_3<\infty$ and $0<y_3<\infty$, see Section~\ref{Riemann_surface}.

      If $y_4<0$ (resp.\ $x_4<0$),
      the point $y=\infty$ (resp.\ $x=\infty$)
      obviously belongs to the real cycle $
      ]y_3,\infty[\cup\{\infty\}\cup]\infty,y_4[$
      (resp.\ $]x_3,\infty[\cup\{\infty\}\cup ]\infty,x_4[$)
      of the complex sphere ${\bf S}$.
      By construction of the Riemann surface ${\bf T}$ and of its universal covering,
      the points $a_1, a_2$ (resp.\ $b_1, b_2$) then lie on the open interval
      $\{\o : \o \in L_{y_4}^{y_3}+\o_2,\, 0<\Im \o<\o_1\}$
      (resp.\ $\{\o : \o \in L_{x_4}^{x_3}+\o_2,\, 0<\Im \o <\o_1 \}$).
      These points are symmetric w.r.t.\ the center of the interval, namely $\o_1/2+\o_2+\o_3/2$
      (resp.\ $\o_1/2+\o_2$).
      The points $\o$ corresponding to $a_3$ and $a_4$ are on
      the open interval $\{\o : \o \in L_{y_4}^{y_3}+\o_2-\o_3,\, 0<\Im \o<\o_1\}$
       and are symmetric w.r.t.\ the center $\o_1/2+\o_2-\o_3/2$ as well.
        Furthermore, $a_4 +\o_3=a_2$ and $a_3+\o_3 =a_1$,
         so that $a_4 \ll_1 a_2$ and $a_3 \ll_1 a_1$, see  Figure~\ref{Locloc}.
         Finally, we have $\Im a_4=\Im a_2 \neq \Im a_1= \Im_3$,
         hence for any $a \in \{a_2,a_4\}$ and any $a' \in \{a_1,a_3\}$,
         $a \not\sim a'$ (in the sense of Definition \ref{def_equiv}).

      If $y_4>0$ or $y_4=\infty$
          (resp.\ $x_4>0$ or $x_4=\infty$), the point $y=\infty$
          (resp.\ $x=\infty$) is on
          $]y_4, \infty]\cup\{\infty\}\cup ]\infty,y_1[$ (resp.\ $]x_4, \infty]\cup\{\infty\}\cup
          ]\infty,x_1[$). Accordingly, the points $a_1,a_2$ (resp.\ $b_1,b_2$) and also $a_3,a_4$
           are on the segment $]\o_3/2, \omega_2+\omega_3/2]$.
           Their location on this segment
           will be specified latter.

Therefore, Propositions \ref{ph1} and \ref{ph2} must
be proved separately for eight subclasses of $51$ models according to the signs of $x_4$ and $y_4$:
 these are those of the walks pictured on Figure~\ref{Allcases}, Subcases I.A, I.B, I.C, II.A, II.B, II.C, II.D and Case III.

The following remark gives a geometric interpretation of this
classification.
\begin{rem}
\label{Remark3} Let ${\bf 1}_{(i,j)}$ be $1$ if
$(i,j)\in\mathcal{S}$,
  otherwise $0$. Then $x_4>0$ (resp.\ $<0$, $=\infty$)
  if and only if ${\bf 1}_{(1,0)}^2-4{\bf 1}_{(1,1)}{\bf
  1}_{(1,-1)}>0$ (resp.\ $<0$, $=0$), see Equation \eqref{d_d_tilde}.
  A symmetric statement holds for $y_4$.
\end{rem}
   As an example, Remark \ref{Remark3} implies that $x_4<0$ if and only
    if $(1,1)\in \mathcal S$ and $(1,-1)\in \mathcal S$.


\renewcommand\thesubsubsection{Subcase \Roman{subsection}.\Alph{subsubsection}}

\medskip

\noindent\underline{\it Case {\rm I}: $y_4<0$, Subcase {\rm I.A}:
$x_4<0$.}
This assumption yields $x^\s\neq x^\ss$; $x^\s,x^\ss \neq \infty$;
$y\neq y^\dd$; $y^\d,y^\dd\neq \infty$; $y^\s,y^\ss\neq \infty$. The
location of the six points $a_1,a_2,a_3,a_4,b_1,b_2$ is already
described above and is pictured on Figure \ref{Locloc}.

\medskip

We first show that for all $\o \in \{\o: \Im \o =\Im a_3,\, \o_{y_1}\leq \Re \o< \o_{y_4}+\o_2\}$, we have
 $r_y(\o)\neq \infty$. The proof consists in three steps.

\medskip

{\noindent}\textit{Step 1.}  Let us first prove that
 $r_y(a_3)\neq \infty$ and $r_y(a_4)\neq \infty$.
  If $|y^\ss|<1$, then $a_3 \in \Delta_y$
 (see Section \ref{Lifting_uc} for the definition of $\Delta_y$)
 and it is immediate from Theorem~\ref{thmde}
 that $r_y(a_3)\neq \infty$.
  If $|y^\ss|\geq 1$, then by Lemma~\ref{lemde}
  there exists $n \in {\bf Z}_+$ such that $a_3-n\o_3 \in \Delta$,
  and by Equation \eqref{cont1} of Theorem \ref{thm2},
     \begin{equation}
     \label{before_reformulation}
          r_y(a_3)=r_y(a_3-n \o_3)+\sum_{k=n}^{1}f_y(a_3-k\o_3).
     \end{equation}
Introducing, for any $\o_0$, the set
     \begin{equation}
     \label{ooo}
          \mathscr{O}^{\Delta}(\o_0)=\{\o_0-\o_3, \o_0-2\o_3,\ldots,\o_0-n_{\o_0}\o_3\},
          \quad n_{\o_0}=\inf\{\ell\geq 0: \o_0-\ell\o_3\in \Delta\},
     \end{equation}
we can rewrite \eqref{before_reformulation} as
     \begin{equation}
     \label{samwino}
          r_y(a_3)=r_y(a_3-n_{a_3} \o_3)+\sum_{\omega \in\mathscr{O}^{\Delta}(a_3)}f_y(\o).
     \end{equation}
In \eqref{samwino}, the quantity $r_y(a_3-n_{a_3} \o_3)$ is defined
thanks to Theorem \ref{thmde}. It may be infinite, but only if $y(a_3-n_{a_3} \o_3)=\infty$. In this case we must have
$a_3-n_{a_3} \o_3=a_1-\omega_2$. But since $a_3+\o_3=a_1$, we then have
$(n_{a_3}+1)\o_3=\omega_2$ which is impossible, due to irrationality
of $\o_2/\o_3$. Hence $r_y(a_3-n_{a_3} \o_3)\ne \infty$.
 Further, we immediately have (see indeed Figure
\ref{Locloc}) that $a_2,a_4,a_2-\o_2,a_4-\o_2 \notin
\mathscr{O}^{\Delta}(a_3)$. Moreover, since either $\Im b_1 \neq \Im
a_3$, or $\Im b_1=\Im a_3$ but then $a_3+\o_3/2=b_1$ (see again
Figure \ref{Locloc}), we also have that $b_1, b_2, b_1-\o_2,
b_2-\o_2 \notin \mathscr{O}^{\Delta}(a_3)$. Finally $a_1=a_3+\o_3
\notin \mathscr{O}^{\Delta}(a_3)$ and $a_1-\o_2=a_3+\o_3-\o_2
\notin \mathscr{O}^{\Delta}(a_3)$, since $\o_2/\o_3$ is
irrational. Thus $f_y(a_3-k\o_3)\neq \infty$ for any
$k\in\{1,\ldots, n_{a_3}\}$. Accordingly, $r_y(a_3)\neq \infty$ and
by the same arguments, $r_y(a_4)\neq \infty$.

\medskip

{\noindent}\textit{Step 2.} We now show that
$r_y(a_1-\o_2)+f_y(a_1-\o_2)\neq \infty$
 and $r_y(a_2-\o_2)+f_y(a_2-\o_2)\neq \infty$.
  By Equation \eqref{sqs2},
$r_y(a_1-\o_2)=-r_x(a_1-\o_2)+K(0,0)Q(0,0)+x(a_1-\o_2)y(a_1-\o_2)$
  and $f_y(a_1-\o_2)=x(a_1-\o_2)[y(a_4)- y(a_1-\o_2)]$; hence
\begin{equation}
 \label{sjk}
 r_y(a_1-\o_2)+f_y(a_1-\o_2)=-r_x(a_1-\o_2)+K(0,0)Q(0,0)+
    x(a_1-\o_2)y(a_4).
    \end{equation}
It follows from Equation \eqref{sqs2}
  and from the first
step that $r_x(a_4)\neq \infty$, since $x(a_4)y(a_4)=x^\s y^\s\neq
\infty$.
 Then, by \eqref{hatx} and \eqref{xieta}
 we get that $r_x(a_1-\o_2)=r_x(\widehat \xi a_4)=r_x(a_4)\neq \infty $.
  Furthermore, $x(a_1-\o_2)y(a_4)=x^\s y^\s\ne \infty$.
  Finally, thanks to \eqref{sjk},
       $r_y(a_1-\o_2)+f_y(a_1-\o_2)\neq \infty$ and
 by the same arguments, $r_y(a_2-\o_2)+f_y(a_2-\o_2)\neq \infty$.

\medskip

{\noindent}\textit{Step 3.} Let us now take any $\o_0$ in $\{\o: \Im
\o =\Im a_3,\, \o_{y_1}\leq \Re \o< \o_{y_4}+\o_2\}$. If $\o_0 \in
\Delta$, then $\o_0 \in \Delta_y$.
Indeed, it is proved in Section \ref{Lifting_uc} that
the domain $\Delta_y$ (resp.\ $\Delta_x$),
which is bounded by $\widehat \Gamma_y^0$ and $\widehat\Gamma_y^1$ (resp.\ $\widehat\Gamma_x^0$ and $\widehat\Gamma_x^1$), is centered around $L_{y_1}^{y_2}$ (resp.\ $L_{x_1}^{x_2}$). Furthermore, $\widehat\Gamma_y^0 \in \Delta_x$ and $\widehat\Gamma_y^1 \notin \Delta_x$ (resp.\ $\widehat\Gamma_x^0 \in \Delta_y$ and $\widehat\Gamma_x^1 \notin \Delta_y$). It follows that for any $\o_0 \in\Delta\setminus\Delta_y$, $\Re \o_0<  \o_{y_1}$. Then, by Theorem \ref{thmde}, $r_y(\o_0)\neq \infty$. If $\o_0 \notin \Delta$, with \eqref{ooo} and \eqref{samwino} we have
     \begin{equation*}
          r_y(\o_0)=r_y(\o_0-n_{\o_0} \o_3)+\sum_{\omega \in\mathscr{O}^{\Delta}(\o_0)}f_y(\o).
     \end{equation*}
For the same reasons as in the first step,
we have that $a_2,a_4,a_2-\o_2, a_4-\o_2 \notin \mathscr{O}^{\Delta}(\o_0)$.
 If $\Re \o_0< \Re a_3$, for obvious reasons
 $\mathscr{O}^{\Delta}(\o_0)$
 cannot contain $a_3$.
 If $\Re a_3 \leq \Re \o_0 < \o_{y_4}+\o_2$,
 it can neither contain $a_3$, since $\o_{y_4}+\o_2- \Re a_3=\o_3$,
 and hence $\Re \o_0- \o_3 < \Re a_3$.
  If $\Im b_1 \neq \Im a_3$,
  or $\Im b_1= \Im a_3$ and $\Re \o_0< \Re b_1$,
   it cannot contain $b_1$. If $\Im b_1=\Im a_3$
   and $\Re b_1 \leq \Re \o_0 < \o_{y_4}+\o_2$,
    then $\Re  \o_0 - \Re b_1 \leq \o_{y_4}+\o_2-\Re b_1=\o_3/2<\o_3$,
    and $b_1\notin \mathscr{O}^{\Delta}(\o_0)$.

If $a_1-\o_2 \notin \mathscr{O}^{\Delta}(\o_0)$,
 then we have $r_y(\o_0-n_{\o_0}\o_3)\neq \infty$
 and $f_y(\o_0-k\o_3)\neq \infty$ for all $k\in\{1,\ldots, n_{\o_0}\}$,
  so that $r_y(\o_0)\neq\infty$ by \eqref{samwino}.

If $a_1-\o_2 \in \mathscr{O}^{\Delta}(\o_0)$,
 then for some $j\in\{1,\ldots n_{\o_0}\}$, we have $\o_0-j\o_3=a_1-\o_2$.
    Then
     $r_y(\o_0-n_{\o_0} \o_3)+\sum_{k=n}^{j+1} f_y(\o_0-k \o_3)=r_y(a_1-\o_2)$
      and thus by \eqref{samwino},
     \begin{equation*}
     r_y(\o_0)=r_y(a_1-\o_2)+f_y(a_1-\o_2)+ \sum_{k=j-1}^{1}
     f_y(\o_0-k\o_3).
     \end{equation*}
     The first term here is finite by the second step and
      $f_y(\o_0-k\o_3) \ne \infty$ for $k\in\{1,\ldots,j-1\}$ by all properties
     said above, so that $r_y(\o_0)\ne \infty$.

\medskip

So far we have proved that for all
 $\o \in \{\o: \Im \o =\Im a_3,\, \o_{y_1}\leq \Re \o< \o_{y_4}+\o_2\}$,
 $r_y(\o)\neq \infty$. In the same way, we obtain that
 $r_y(\o)\neq \infty$ for $\o \in \{\o: \Im \o =\Im a_4,\,  \o_{y_1}\leq \Re \o<  \o_{y_4}+\o_2\}$.

\medskip

Since by \eqref{hatx},
     \begin{align*}
          &\widehat \eta \{\o: \Im \o =\Im a_3,\, \o_{y_1}\leq \Re \o<  \o_{y_4}+\o_2\}= \{\o: \Im \o =\Im a_4,\,\o_{y_4}< \Re \o \leq  \o_{y_1}\},\\
          &\widehat \eta \{\o: \Im \o =\Im a_4,\, \o_{y_1}\leq \Re \o<  \o_{y_4}+\o_2\}= \{\o: \Im \o =\Im a_3,\,\o_{y_4}< \Re \o \leq  \o_{y_1}\},
     \end{align*}
Equation \eqref{xieta} implies that $r_y(\o)\neq \infty$ on the
segments $\{\o\in \Pi_y : \Im  \o=a_3,a_4\}$, except for their ends
$a_1,a_2$.
    The segments $\{\o : \Im \o=a_3,a_4,\, \o_{x_1}\leq \Re \o \leq \o_{x_4}+\o_2\}$
 do not contain any point where $y(\o)=\infty$.
 It follows from Equation \eqref{sqs2}
 that $r_x(\o)\neq \infty$
     on these segments except for
 points where $x(\o)=\infty$ if they exist.
      This last fact happens if and only if $\Im b_1=\Im a_3$ and only
  at the ends $b_1,b_2$ of the segments.

If $\Im b_1 \neq \Im a_3$,
 we can show exactly in the same way that
 $r_y(\o)\neq\infty$ on the two segments
 $\{\o\in \Pi_y : \Im   \o=b_1,b_2\}$ and that
$r_x(\o)\neq\infty$ on the segments $\{\o : \Im\o=b_1,b_2,\,\o_{x_1}\leq \Re \o \leq \o_{x_4}+\o_2\}$, except for their ends
$b_1,b_2$.
  This concludes the
  proof of Proposition~\ref{ph1}.

        We proceed with the proof of Proposition \ref{ph2}.
   Let us verify the assumptions of Lemma \ref{dop} for
   $\o_0=a_3,a_4,b_1,b_2$.
We have proved that $r_y(a_3),r_y(a_4),r_y(b_1), r_y(b_2)\ne
\infty$, $a_3 \ll_1 a_1$, $a_4 \ll_1 a_2$ and that the pairs $\{a_1,a_3\}$
and $\{a_2, a_4\}$ are not ordered.  Let us now show that for any
$k\in\{3,4\}$ and $\ell\in\{1,2\}$, it is impossible to have $a_k\sim b_\ell$.
If $\Im b_\ell \neq \Im a_3, \Im a_4$, this is obvious.
If $\Im b_\ell =\Im a_3$, then it is enough to note that
$b_\ell-a_3=\o_3/2$ and $a_1-b_1=\o_3/2$ (see Figure \ref{Locloc}).
From the irrationality of $\o_2/\o_3$, it follows that
$b_\ell\not\sim a_1,a_3$ and in the same way $b_\ell \not\sim
a_2,a_4$.
  Then there is no other $\o \in \Pi_y$ except for $a_1$ (resp.\ $a_2$)
   such that $a_3 \ll \o$ (resp.\ $a_4 \ll \o$)
 and $f_y(\o)=\infty$. There is no $\o \in \Pi_y$
   such that $b_\ell \ll \o$
 and $f_y(\o)=\infty$, $\ell=1,2$.
 Hence, Lemma \ref{dop} could be applied to any of four points
  $\o_0=a_3,a_4,b_1,b_2$
  if the assumption \ref{(ii)} of this lemma is satisfied for these points.
   It is then immediate that $\lim_{\o \to b_\ell}
      f_y(\o)=\lim_{\o \to b_\ell} x(\o)[y(\widehat \xi \o)-y(\o)]=\infty$,
      $\ell\in\{1,2\}$,
          since $x(\o) \to \infty$  and the other term converges to
       $\pm[y^\d-y^\dd]\ne 0$.
   Let us verify that $\lim_{\o \to a_3}
   \{f_y(\o)+f_y(\o+\o_3)\}=\infty$. We have
\begin{align*}
\lim_{\o \to a_3} \{f_y(\o)+f_y(\o+\o_3)\}=&
\lim_{\o \to a_3} \{x(\o)[y(\widehat \xi\o)-y(\o)]+x(\widehat \eta \widehat
\xi\o)[y(\widehat \xi \widehat \eta \widehat \xi\o)-y(\widehat \eta \widehat
\xi\o)]\}\\=&\lim_{\o \to a_3}\{x(\widehat \eta \widehat \xi\o)y(\widehat \xi \widehat \eta
\widehat \xi\o)-x(\o)y(\o)\}\\+&\lim_{\o \to a_3}\{x(\o)y(\widehat
\xi\o)-x(\widehat \eta \widehat \xi\o)y(\widehat \eta \widehat \xi\o)\}.
\end{align*}
    The first term above converges to $x^\s y^\s- x^\ss y^\ss$.
 By \eqref{hat} the second term equals the limit of the product
 $y(\widehat \xi \omega) [x(\widehat \xi \o)-x(\widehat \eta \widehat \xi \o)]$.
    If $\o \to a_3$, then $\widehat \xi \o \to a_2-\o_2$
  so that the first term in the product converges to
$y(a_2-\o_2)=y(a_2)=\infty$.
  The second
  term of this product converges to $x(a_2-\o_2)-x(a_1)=x^\ss-x^\s$
    which is different from $0$ as $x^\ss\ne x^\s$.
  Then assumption \ref{(ii)} is satisfied for $\o_0=a_3$
   and in the same way for $\o_0=a_4$.
    Lemma~\ref{dop} applies to any of the four points
   $\o_0=a_3,a_4,b_1,b_2$. But  by \eqref{hat},
  $\mathcal{I}_x(a_3)=\mathcal{I}_x(a_4)= \mathcal{I}_x(a_1)=\mathcal{I}_x(a_2)$,
$\mathcal{I}_y(a_3)=\mathcal{I}_y(a_4)=\mathcal{I}_y(a_1)=\mathcal{I}_y(a_2)$,
$\mathcal{I}_x(b_1)=\mathcal{I}_x(b_2)$,
$\mathcal{I}_y(b_1)=\mathcal{I}_y(b_2)$
      so that poles of $x \mapsto  Q(x,0)$ are dense on the curves
$\mathcal{I}_x(a_1)$ and $\mathcal{I}_x(b_1)$ and those of $y
\mapsto Q(0,y)$ are dense on the curves $\mathcal{I}_y(a_1)$ and
$\mathcal{I}_y(b_1)$. Proposition \ref{ph2} is proved.

\medskip

\noindent\underline{\it Case {\rm I}: $y_4<0$, Subcase {\rm I.B}:
$x_4=\infty$.} This assumption implies that $x^\s\neq x^\ss$;
$x^\s,x^\ss \neq \infty$;
 $y^\d=y^\dd\neq \infty$; $y^\s,y^\ss\ne\infty$.

 The points $a_1,a_2,a_3,a_4$ are located as in the previous case,
 see Figure \ref{Locloc}.
  Consequently we have the following facts:
   $r_y(\o)\neq\infty$ on the segments $\{\o \in \Pi_y :
\Im\o=a_3,a_4\}$, except for their ends $\o=a_1,a_2$;
$r_x(\o)\neq\infty$ on the segments $\{\o: \Im\o=a_3,a_4,\,\o_{x_1}\leq \Re \o \leq \o_{x_4}+\o_2\}$.
 Lemma \ref{dop}
applies to $\o_0=a_3,a_4$ as in the previous case, as $x^\s \ne
x^\ss$.  Then the set of poles of all branches of $x\mapsto Q(x,0)$
(resp.\ $y\mapsto Q(0,y)$) is dense on $\mathcal{I}_x(a_1)$ (resp.
$\mathcal{I}_y(a_1)$) where  $\mathcal{I}_x(a_1)=\mathcal{I}_x(a_2)=
\mathcal{I}_x(a_3)=\mathcal{I}_x(a_4)$
(resp.~$\mathcal{I}_y(a_1)=\mathcal{I}_y(a_2)=\mathcal{I}_y(a_3)=\mathcal{I}_y(a_4)$).

Since $x_4=\infty$, we have that $b_1=b_2=\o_{x_4}+\o_2$.
 Take any $\o_0$ with $\Im \o_0=0$ such that $\o_{y_1}\leq  \Re \o_0 \leq
  \o_{y_4}+\o_2$. Then $y(\o_0)\ne \infty$.
    Let us show that $r_y(\o_0)\neq \infty$.
  If $\o_0 \in \Delta$, then by the same reasons as in Subcase I.A
  $\o_0 \in \Delta_y$ and $r_y(\o_0)\ne \infty$.
       If $\o_0 \notin \Delta$, consider the set $\mathscr{O}^{\Delta}(\o_0)$
    defined as in \eqref{ooo} and \eqref{samwino}.
    Clearly $b_1-\o_2=\o_{x_4} \notin
    \mathscr{O}^{\Delta}(\o_0)$. Since $b_1+\o_3/2=\o_{y_4}+\o_2$, we have
       $\o_0-\o_3\leq \o_{y_4}+\o_2 -\o_3 <b_1$  and  then
      $b_1 \notin
    \mathscr{O}^{\Delta}(\o_0)$. Hence  $
    \mathscr{O}^{\Delta}(\o_0)$ does not contain any point where
      $y(\o)$ or $f_y(\o)$ is infinite. Thus  by \eqref{samwino},
       $r_y(\o_0)\ne \infty$.

We have $\widehat \eta\{ \o : \Im \o=0,\, \o_{y_1}\leq  \Re \o_0 \leq
  \o_{y_4}+\o_2\}=\{\o : \Im \o=\o_1,\, \o_{y_4}\leq  \Re \o_0
\leq
  \o_{y_1}\}$. Then
by \eqref{xieta} and \eqref{buzz}, we get that
  $r_y(\o)\neq\infty$ for all $\o \in \Pi_y$ with $\Im \o=0$.
The segment $\{\o : \Im \o=0,\, \o_{x_1}\leq \o \leq \o_{x_4}+\o_2\}$
  does not contain any point with $y(\o)=\infty$.
   By \eqref{sqs2}
   this gives $r_x(\o)\neq\infty$ for all $\o$ on this segment
    except for the points where $x(\o)= \infty$ (that is only at $\o=\o_{x_4}+\o_2=b_1$),
   and this concludes the proof of Proposition \ref{ph1}.

    We have proved in particular that $r_y(\o_0)\ne \infty$ for $\o_0=b_1$.
       Furthermore, there is no $\o \in \Pi_y$ such that $b_1\ll \o$
   and $f_y(\o)=\infty$. Finally
 \begin{equation}
 \label{xqx}
 \lim_{\o \to b_1}f_y(\o)= \lim_{\o \to b_1} x(\o)[y(\widehat \xi
 \o)-y(\o)]=\lim_{\o \to b_1} x(\o)
     \frac{[b(x(\o))^2-4a(x(\o))c(x(\o))]^{1/2}}{a(x(\o))},
 \end{equation}
where $x(\o)\to \infty$ as $x \to b_1$. For all models in Subcase
I.B $\deg a(x)=2$, $\deg b(x)=1$ and
     $\deg c(x)=1$, so that \eqref{xqx} is of the order
     $O(|x(\o)|^{1/2})$. Thus
        $\lim_{\o \to b_1}f_y(\o)=\infty$.
   By Lemma~\ref{dop} with $\o_0=b_1$,
   the poles of $x\mapsto  Q(x,0)$ and those of $y\mapsto  Q(0,y)$ are
   dense on $\mathcal{I}_x(b_1)=\mathcal{I}_x(b_2)$ and $\mathcal{I}_y(b_1)=\mathcal{I}_y(b_2)$,
   respectively. They are the intervals of the real line claimed in
  Theorem \ref{main_tt} \ref{inet}. Proposition~\ref{ph2} is proved.

\medskip

\noindent\underline{\it Case {\rm I}: $y_4<0$, Subcase {\rm I.C}:
$x_4>0$.} The statements and results about $a_1,a_2,a_3,a_4$ are the
same as in Subcases I.A and I.B, see Figure \ref{Locloc} for
 their location.


We now locate $b_1,b_2$. By definition (see Section
\ref{Riemann_surface}), the values $y_1, y_2, y_3,y_4$ are the roots
of
     \begin{equation*}
          \widetilde d(y)=(\widetilde b(y)-2[\widetilde a(y)\widetilde c(y)]^{1/2})
          (\widetilde b(y)+2[\widetilde a(y)\widetilde c(y)]^{1/2})=0.
     \end{equation*}
Hence, for two of these roots
 $\widetilde b(y)=-2[\widetilde a(y)\widetilde c(y)]^{1/2}$ and then
 $X(y)\geq 0$ (see \eqref{def_XX}),
  and for the two others $\widetilde b(y)=2[\widetilde a(y)\widetilde c(y)]^{1/2}$
  and then $X(y)\leq 0$.
  But $X(y_2)$ and $X(y_3)$ are on the segment $[x_2, x_3]\subset ]0,\infty[$.
  Thus $X(y_1)\leq 0$ and $X(y_4)\leq 0$.
Since $x(b_1)=x(b_1-\o_2)=\infty$, $x_4=x(\o_{x_4})>0$ and
$X(y_4)=x(\o_{y_4})<0$, it follows that $b_1-\o_2 \in ]\o_{x_4},
\o_{y_4}[$,
in such a way that $b_1 \in ]\o_{x_4}+\o_2, \o_{y_4}+\o_2[$.
 Also, $b_2=\widehat \xi (b_1-\o_2)-\o_1=
    2(\o_{x_4}+\o_2)-b_1$ is symmetric to $b_1$
    w.r.t.\ $\o_{x_4}+\o_2$.
     Since $x(\o_{y_1})=X(y_1)\leq 0$ and
       $x(\o_{x_4}+\o_2)=x_4>0$,
       it follows that $\o_{y_1}<b_2<\o_{x_4}+\o_2$, see Figure~\ref{LoclocICIIBC}.

\begin{figure}[t]
\begin{center}
\begin{picture}(00.00,645.00)
\hspace{-112.5mm}
\includegraphics{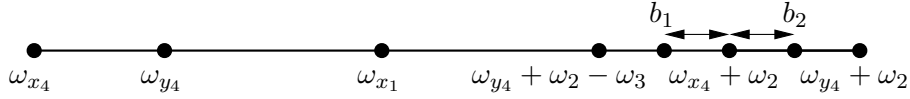}
\end{picture}
\end{center}
\vspace{-213mm} \caption{Location of $b_1,b_2$ if $x_4>0$, Subcases
I.C and II.C} \label{LoclocICIIBC}
\end{figure}

Now we show that for any $\o_0$
 with $\Im \o_0=0$ and
 $\o_{y_1} \leq \Re \o_0 \leq \o_{y_4}+\o_2$, $r_y(\o_0)\neq \infty$.
 Note that $y(\o_0)\ne \infty$.
 If $\o_0 \in \Delta$,
 by the same arguments as in Subcase I.A, $\o_0 \in
 \Delta_y$ and $r_y(\o_0)\ne \infty$.
 If $\o_0 \notin \Delta$,
 then consider $\mathscr{O}^\Delta(\o_0)$ with the notation \eqref{ooo}.

  Note that $b_1-\o_2 \notin\Delta$.
      For this, it is enough to prove that
  $b_1-\o_2\notin \Delta_x$ and that $b_1-\o_2\notin \Delta_y$.
  First, $b_1-\o_2 \notin \Delta_x$,
  since $x(b_1)=x(b_1-\o_2)=\infty$.
  Furthermore, $\Delta_y$ is centered w.r.t.\ $L_{y_1}^{y_2}$,
  and $\o_{y_4} \notin \Delta_y$ (since $|y_4|>1$).
  Hence the point $b_1-\o_2 < \o_{y_4}$ cannot be in $\Delta_y$.

   Since $\Delta \cap \{\o\in {\bf C} : \Im \o=0\}$ is
    an open interval containing $\o_{y_1}$, and since $b_1-\o_2
<\o_{y_1}\leq \o_0$, it follows that
$b_1-\o_2<\o_0-n_{\omega_0}\o_3$ (see \eqref{ooo}), so that
$b_1-\o_2 \notin \mathscr{O}^\Delta(\o_0)$. Obviously
$b_2-\o_2\notin \mathscr{O}^\Delta(\o_0)$. Furthermore, since
$\o_{y_4}+\o_2-b_2=\o_3/2+ \o_{x_4}+\o_2-b_2<\o_3/2+\o_3/2=\o_3$,
 it follows that $\o_0-\o_3<b_2<b_1$
 for any such $\o_0$.
 Hence $b_1,b_2 \notin\mathscr{O}^\Delta(\o_0)$.
 Thus for any $\o \in \mathscr{O}^\Delta(\o_0)$,
   $y(\o)\ne \infty$ and $f_y(\o) \ne \infty$.
  By \eqref{samwino}, $r_y(\o_0)\neq \infty$.
This implies, exactly as in Subcase I.B---by \eqref{xieta} and
\eqref{buzz}---, that
  $r_y(\o)\neq\infty$ for all $\o \in \Pi_y$ such that $\Im \o=0$.
   By \eqref{sqs2}
   this gives $r_x(\o)\neq\infty$ for all $\o$ with $\Im
   \o=0$ and $\o_{x_1} \leq \Im \o \leq \o_{x_4}+\o_2$, except for
   points $\o$ where $x(\o)=\infty$ (that happens for
   $\o=b_2$ only),
    and this concludes
   the proof of Proposition \ref{ph1}.

   In particular, we proved that $r_y(b_1)\neq \infty$
and also $r_y(b_2)\neq\infty$.  Since $y^\d \ne y^\dd$, we have
$\lim_{\o \to b_1}f_y(\o)=\infty$ and also $\lim_{\o \to
b_2}f_y(\o)=\infty$ by the same arguments
 as in Subcase I.A. If $b_1$ and $b_2$ are not ordered,
Lemma \ref{dop} applies to both of these points. If $b_1 \ll b_2$
(resp.\ $b_2 \ll b_1$), then there is no  $\o \in \Pi_y$ such that
  $\o \neq b_2$ (resp.\ $\o \neq b_1$),
$f_y(\o)=\infty$ and $b_2 \ll \o$ (resp.\ $b_1 \ll \o$). Hence
Lemma \ref{dop} applies to $\o_0=b_2$ (resp.\ $\o_0=b_1$).
 Thus the set of poles
 of all branches of $x\mapsto Q(x,0)$ (resp.\ $y\mapsto Q(0,y)$) is
dense on the curves $\mathcal{I}_x(b_2)$ and $\mathcal{I}_y(b_2)$
  (resp.\ $\mathcal{I}_x(b_1)$ and $\mathcal{I}_y(b_1)$). We conclude the proof of Proposition \ref{ph2}
 with the observation that $\mathcal{I}_x(b_2)=\mathcal{I}_x(b_1)$,
 while $\mathcal{I}_y(b_2)=\mathcal{I}_y(b_1)$ are intervals of the real line
  as claimed in Theorem \ref{main_tt} \ref{inet}.

\medskip

\noindent\underline{\it Case {\rm II}: $y_4>0$, location of
$a_1,a_2,a_3,a_4$.} We first exclude Subcase II.D where $x_4=\infty$
and $Y(x_4)=\infty$, and we locate the points $a_1,a_2,a_3,a_4$. In
this case we have $x^\s\neq x^\ss$. Note that $x_1, x_2, x_3,x_4$
are the roots of the equation
     \begin{equation*}
          d(x)=(b(x)-[a(x)c(x)]^{1/2})(b(x)+[a(x)c(x)]^{1/2})=0.
     \end{equation*}
Hence for two of these roots
$b(x)=-2[a(x)c(x)]^{1/2}$ and then $Y(x)\geq 0$
(see \eqref{def_YY}), and for two others $b(x)=2[a(x)c(x)]^{1/2}$
 and then $Y(x)\leq 0$.
 But $Y(x_2)$ and $Y(x_3)$ are on the segment
 $[y_2,y_3]\subset ]0,\infty[$.
 Hence $Y(x_1)\leq 0$ and $Y(x_4)\leq 0$.
 If in addition $x_4=\infty$,
  then $Y(x_4)$ equals $0$ or $\infty$;
      note also that if $Y(x_4)=\infty$
   then necessarily $x_4=\infty$.
    But the case when $Y(x_4)=\infty$ and $x_4=\infty$
    is excluded from our consideration at this moment.
    It follows that $\infty \in [y_4,Y(x_4)[$,
     and in fact $\infty \in ]y_4,Y(x_4)[$,
      since the case $y_4=\infty$ is excluded from Case II.

\begin{figure}[t]
\begin{center}
\begin{picture}(00.00,655.00)
\hspace{-112mm}
\includegraphics{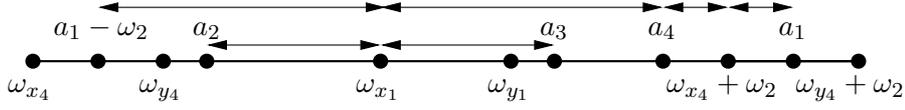}
\end{picture}
\end{center}
\vspace{-213mm} \caption{Location of $a_1,a_2,a_3,a_4$ if $y_4>0$,
$(x_4, Y(x_4))\ne (\infty, \infty)$, Subcases II.A, II.B and
II.C}\label{LoclocIIA}
\end{figure}

It follows from the above considerations that $a_1 \in
]\o_{x_4}+\o_2, \o_{y_4}+\o_2[$, see Figure \ref{LoclocIIA}. In
particular, we have $a_1-\o_2 \in ]\o_{x_4}, \o_{y_4}[$ and
$a_2=\widehat \eta a_1-\o_1=-(a_1-\o_2) + 2 \o_{y_4}$.
 This means that $a_1-\o_2$ and $a_2$ are symmetric w.r.t.\ $\o_{y_4}$.
  Furthermore $a_2 \in  ]\o_{y_4}, \o_{y_1}[$,
   but since $y_4>0$ and $Y(x_1)\leq 0$, we have
    $a_2 \in ]\o_{y_4},\o_{x_1}[$.
    We must put $a_3=\widehat \xi a_2-\o_1=-a_2+2\o_{x_1}$,
    in such a way that the points $a_2$
    and $a_3$ are symmetric w.r.t.\ $\o_{x_1}$.
     Finally, $a_4=\widehat \xi (a_1-\o_2)-\o_1
     =-(a_1-\o_2)+ 2\o_{x_1}=a_3+a_2-(a_1-\o_2)$.
     Note that $\o_{x_4}+\o_2-a_4=a_1- \o_2-\o_{x_4}>0$.
     Furthermore, $a_1-a_3=\o_3$ and $a_2+\o_2-a_4=\o_3$,
     so that $a_3 \ll a_1$ and $a_4 \ll a_2$.

\medskip

\noindent\underline{\it Case {\rm II}: $y_4>0$, Subcase {\rm II.A}:
$x_4<0$.} In this case we have $y^\d\neq y^\dd$; $y^\d,y^\dd\neq
\infty$; $x^\s, x^\ss\neq \infty$; $y^\s, y^\ss \neq \infty$. The
points $b_1,b_2$ are located as in Subcase I.A, see
Figure~\ref{Locloc}.
  Further, we can show as in Subcase I.A
that $r_y(\o)\neq\infty$ on the segments  $\{\o:\Im\o=b_1,b_2,\,
\o_{y_1}\leq \Re \o \leq \o_{y_4}+\o_2 \}$, and we deduce that
$r_x(\o)\neq\infty$ on $\{\o:\Im\o=b_1,b_2,\,\o_{x_1}\leq \Re \o
\leq \o_{x_4}+\o_2 \}$ except for their ends $b_1,b_2$.
Consequently, the set of poles of all branches of $x\mapsto Q(x,0)$
(resp.\ $y\mapsto Q(0,y)$) is dense on the curve
$\mathcal{I}_x(b_1)=\mathcal{I}_x(b_2)$ (resp.\
$\mathcal{I}_y(b_1)=\mathcal{I}_y(b_2)$), as claimed in Proposition
\ref{ph2}.

Consider now $r_x(\o)$ and $r_y(\o)$ for $\o$
with $\Im \o=0$. See Figure \ref{LoclocIIA}
for the location of the points $a_1,a_2,a_3,a_4$.

We first prove that for
\begin{equation}
\label{oaa}r_y(\o_0)\ne \infty,\quad
\forall \o_0 \in \{\o : \Im \o=0,\, \o_{y_1}\leq \o_0 \leq
\o_{y_4}+\o_2\}\setminus \{a_1\}.
\end{equation}
The proof consists in three steps.

\smallskip

{\noindent}\textit{Step 1.} We prove that $r_y(a_3)\neq\infty$ and
$r_y(a_4)\neq\infty$.

If $a_3 \in \Delta$, then necessarily $r_y(a_3)\neq\infty$, since
$x^\ss,y^\ss\neq \infty$. Otherwise $|x^\ss|, |y^\ss|\geq 1$, and
then $a_2=\widehat \xi a_3-\o_1 \notin \Delta$. Since
$a_2<\o_{x_1}<a_3$, $\o_{x_1} \in \Delta$ and $a_2,a_3 \notin
\Delta$, it follows that any $\o \in \mathscr{O}^\Delta(a_3)$ must
be in $]a_2,a_3[$, hence $f_y(\o)\neq \infty$, $x(\o)\neq \infty$
and $y(\o)\neq\infty$. Thus by
\eqref{samwino}, $r_y(a_3)\neq\infty$.

We now show that $r_y(a_4)\neq\infty$. Suppose first that $a_1-\o_2
\in \Delta$. Then, since $a_1-\o_2\notin \Pi_y$,
 we have  $a_1-\o_2 \in \Delta_x$, so that $r_x(a_1-\o_2)\neq \infty$
by Theorem~\ref{thmde}.  Then by Equations \eqref{xieta} and \eqref{buzz},
$r_x(a_4)=r_x(\widehat\xi(a_1-\o_2)-\o_1)=r_x(a_1)\neq\infty$. Since
$x^\s,y^\s\neq \infty$, we also have $r_y(a_4)\neq\infty$ by
\eqref{sqs2}. Assume now that $a_1-\o_2 \notin \Delta$. Since
$a_1-\o_2<\o_{x_1}<a_4$, then
$a_1-\o_2\notin\mathscr{O}^\Delta(a_4)$. Furthermore $a_2 \notin
\mathscr{O}^\Delta(a_4)$ as $a_4+\o_3=a_2+\o_2$ and $\o_3/\o_2$ is
irrational. Finally $a_4-a_3<a_1-a_3=\o_3$, so that $a_3 \notin
\mathscr{O}^\Delta(a_4)$. It follows that for any $\o \in
\mathscr{O}^\Delta(a_4)$, we have
$f_y(\o)\neq \infty$ and $y(\o)\neq\infty$. Hence
$r_y(a_4)\neq\infty$.

\smallskip

{\noindent}\textit{Step 2.} We prove that
$r_y(a_2)+f_y(a_2)\neq\infty$ and that
$r_y(a_1-\o_2)+f_y(a_1-\o_2)\neq\infty$. By Equation \eqref{sqs2}
\begin{align}
\nonumber
 r_y(a_2)+f_y(a_2)&=-r_x(a_2)+K(0,0)Q(0,0)+x(a_2) y(a_2)+
x(a_2)[y(\widehat \xi a_2)-y(a_2)]\\
&= -r_x(a_2)+K(0,0)Q(0,0)+
x(a_2)y(\widehat \xi a_2).\label{mphh}
\end{align}
  Since $r_y(a_3)\neq\infty$ and since $x^\ss,y^\ss \neq \infty$, it
follows from Equation \eqref{sqs2} that $r_x(a_3)\neq\infty$. Then
by \eqref{xieta} and \eqref{buzz}, $r_x(a_2)=r_x(\widehat \xi
a_3-\o_1)=r_x(a_3)\neq\infty$.
  Since $x(a_2)=x^\ss\ne \infty$ and $y(\widehat \xi a_2)=y^\ss \ne
 \infty$, \eqref{mphh} is finite.
   By completely analogous arguments we obtain that
   $r_y(a_1-\o_2)+f_y(a_1-\o_2)\neq\infty$.

\smallskip

{\noindent}\textit{Step 3.} Let us now show \eqref{oaa}.
 If $\o_0
\in \Delta$, then  by the same arguments as in Subcase I.A
 $\o_0 \in \Delta_y$ and then $r_y(\o_0)\ne \infty$ by Theorem
 \ref{thmde}.
Otherwise, consider $\mathscr{O}^\Delta(\o_0)$ defined in
\eqref{ooo}. Note that
\begin{equation}
\label{zozo}
\o \notin \mathscr{O}^\Delta(\o_0),
\quad \forall \o \in ]\o_{x_4}+\o_2-\o_3/2, \o_{y_4}+\o_2],
\end{equation}
   since in this case
$\o_0-\o_3<\o$. In particular, $a_4 \notin
\mathscr{O}^\Delta(\o_0)$. Furthermore, $a_3 \in
\mathscr{O}^\Delta(\o_0)$ implies that $\o_0=a_3+\ell\o_3$ for some
$\ell\geq 1$. But $a_3+\o_3=a_1\neq \o_0$ and
$a_3+\ell\o_3>\o_{y_4}+\o_2$ for any $\ell\geq 2$. Hence $a_3
\notin \mathscr{O}^\Delta(\o_0)$. Since $a_2-(a_1-\o_2)<\o_3$, it
is impossible that both $a_2$ and $a_1-\o_2$ belong to
$\mathscr{O}^\Delta(\o_0)$.
    If none of them belongs to $\mathscr{O}^\Delta(\o_0)$, then
  $y(\o)\neq \infty$ and $f_y(\o) \ne \infty$ for any $\mathscr{O}^\Delta(\o_0)$, and
  then $r_y(\o_0) \ne \infty$.
 Suppose, e.g., that $a_2\in
\mathscr{O}^\Delta(\o_0)$. Then for some $\ell\geq 1$,
$\o_0=a_2+\ell\o_3$, and by \eqref{samwino},
   $r_y(\o_0)=r_y(a_2)+
f_y(a_2)+\sum_{k=\ell-1}^{1}
  f_y(\o_0-k\omega_3)$.
     But $r_y(a_2)+f_y(a_2)\neq\infty$ by the second step, and
obviously $\sum_{k=\ell-1}^{1} f(\o_0-k\o_3)\neq\infty$ by all facts
said above, so that $r_y(\o_0)\neq\infty$. The reasoning is the same
if $a_1-\o_2 \in \mathscr{O}^\Delta(\o_0)$.
  This concludes the proof of \eqref{oaa}.

  \medskip

Applying \eqref{xieta} and \eqref{buzz} exactly as in Subcase I.B,
we now reach the conclusion that $r_y(\o_0)\neq \infty$ for all
$\o_0 \neq a_2=\widehat \eta a_1$ with $\Im \o_0=0$ and
$\o_{y_4}\leq \Re\o_0 \leq \o_{y_1}$ as well. Next, exactly as in
Subcase I.B, thanks to \eqref{sqs2}, we derive that
$r_x(\o)\neq\infty$ for all $\o_0$ such that $\Im \o=0$ and
$\o_{x_1}\leq \Re\o\leq \o_{x_4}+\o_2$ except possibly for points
where $x(\o)=\infty$. But these points are absent on this segment in
this case. This concludes the proof of Proposition \ref{ph1}.

For the same reason as in Subcase I.A, the fact that $x^\s \ne
x^\ss$ gives
 $\lim_{\o \to a_4}\{f_y(\o)+f_y(\o+\o_3)\}=\infty$ and
 $\lim_{\o \to a_3}\{f_y(\o)+f_y(\o+\o_3)\}=\infty$.
Further, $a_3 \ll_1 a_1$ and $a_4 \ll_1 a_2$. If in addition $a_3
\ll a_4$, then due to the fact that $a_3+\o_3=a_1$ we have $a_3
\ll_1 a_1 \ll a_4 \ll_1 a_2$. There is no point $\o \in \Pi_y
\setminus \{a_2\}$  such that $a_4 \ll \o$ and $f_y(\o)=\infty$.
 Lemma \ref{dop} applies to $\o_0=a_4$. If $a_4 \ll a_3$, then also $a_4
\ll_1 a_2 \ll a_3 \ll_1 a_1$ and this lemma applies to $\o_0=a_3$.
If $a_3 \not\sim a_4$, Lemma \ref{dop} can be applied to both $a_3$
and $a_4$. Since $\mathcal{I}_x(a_3)=\mathcal{I}_x(a_4)=
\mathcal{I}_x(a_1)=\mathcal{I}_x(a_2)= [x_1,x_4]$ and
$\mathcal{I}_y(a_3)=\mathcal{I}_y(a_4)=
\mathcal{I}_y(a_1)=\mathcal{I}_y(a_2)= {\bf R}\setminus ]y_4,y_1[$,
the set of poles of $x\mapsto Q(x,0)$ (resp.\ $y\mapsto Q(0,y)$) is
dense on the announced intervals and Proposition \ref{ph2} is
proved.

\medskip

\noindent\underline{\it Case {\rm II}: $y_4>0$, Subcase {\rm II.B}:
$x_4>0$ and exactly
 one of $y^\d,y^\dd$ is $\infty$.}
Assume, e.g., that $y^\d=\infty$ and $y^\dd\neq \infty$. Then
$x^\s=\infty$; $y^\s=y^\dd\neq \infty$; $x^\ss \neq x^\s=\infty$;
$y^\ss \neq \infty$. It follows that $b_1=a_1$ and $b_2=\widehat \xi
b_1-\o_1-\o_2=a_4$,
 while $a_1,a_2,a_3,a_4$ are pictured as previously, in Subcase
 II.A, see Figure~\ref{LoclocIIA}.

  We first derive \eqref{oaa}.
By the same reasoning as in Subcase II.A,
 we reach the conclusion that $r_y(a_3)\neq\infty$.
 Let us note that in this case $a_1-\o_2 \notin \Delta$,
 as $x(a_1-\o_2),y(a_1-\o_2)=\infty$. Next, we derive as in Subcase
 II.A that $r_y(a_4)\neq\infty$ and that $r_y(a_2)+f_y(a_2)\ne\infty$,
 since $x^\ss y^\ss\neq \infty$.
 Finally, again by the same arguments as in Subcase II.A,
 we conclude that for any $\o_0\neq a_1$ with $\Im \o_0=0$
 and $\o_{y_1}\leq \Re\o_0 \leq \o_{y_4}+\o_2$,
 the orbit $\mathscr{O}^{\Delta}(\o_0)$ does not contain $a_3$ and $a_4$.
 The orbit can neither contain $a_1-\o_2$, since $a_1-\o_2 \notin \Delta$
 and $a_1-\o_2<\o_{y_1}<\o_0$, where $\o_{y_1} \in \Delta$.
 Since $r_y(a_2)+f_y(a_2)\ne\infty$, then as in Subcase II.A
 we have $r_y(\o_0)\neq\infty$.

     It follows from \eqref{xieta} and \eqref{buzz} that
     $r_y(\o_0)\neq \infty$ for any $\o_0$
      such that $\Im \o_0=0$
     and $\o_{y_4}\leq \o_0 \leq \o_{y_4}+\o_2$, except for
     $a_1$ and $\widehat \eta a_1-\o_1=a_2$. By \eqref{sqs2},
     $r_x(\o_0)\neq \infty$ for any $\o_0$ such that
     $\Im \o_0=0$ and $\o_{x_1}\leq \o_0 \leq \o_{x_4}+\o_2$,
     except for points $\o_0$ where $x(\o_0)=\infty$
      (this is $\o_0=a_4$ in this case). This finishes the proof of
      Proposition \ref{ph1} in this case.

     Since $x^\s \ne x^\ss$, the same reasoning as in Subcase I.A
  gives $\lim_{\o \to a_4}\{f_y(\o)+f_y(\o+\o_3)\}=\infty$ and
 $\lim_{\o \to a_3}\{f_y(\o)+f_y(\o+\o_3)\}=\infty$.
 The rest of the proof of Proposition \ref{ph2}
    via the use of Lemma \ref{dop} with $\o_0=a_3$
 if $a_4 \ll a_3$ or with $\o_0=a_4$ if $a_3 \ll a_ 4$, or with
 indifferent choice of $a_3$ or $a_4$ if $a_3\not\sim a_4$,
  is the same as
in Subcase II.A.

\medskip

\noindent\underline{\it Case {\rm II}: $y_4>0$, Subcase {\rm II.C}:
$x_4>0$ and $y^\d,y^\dd\neq \infty$, or $x_4=\infty$ and $Y(x_4)\neq
\infty$.} In this case, we have $y^\d,y^\dd\neq \infty$; $x^\s,
x^\ss \neq \infty$; $x^\s \neq x^\ss$, $y^\s,y^\ss\neq \infty$.

The points $a_1,a_2,a_3,a_4$ are pictured as in Subcases II.A and
II.B, see Figure \ref{LoclocIIA},  while $b_1,b_2$ are pictured as
in Subcase I.B (where $b_1=b_2=\omega_{x_4+\omega_2}$) or I.C, see
Figure \ref{LoclocICIIBC}. They are such that $b_1\ne a_1$ and $b_2
\ne a_4$.
 In particular,  $b_1,b_2 \in ]\o_{x_4}+\o_2-\o_3,\o_{x_4}+\o_2[$
 and are symmetric w.r.t.\ $\o_{x_4}+\o_2$; $b_1=b_2$ is in the
 middle of this interval if and only if $x_4=\infty$.
          Hence for any $\o_0$ with $\Im \o_0=0$ and
 $\o_{y_1}\leq \Re \o_0 \leq \o_{y_4}+\o_2$,
 we have $\o_0-\o_3<b_2<b_1$, so
    that $b_1,b_2 \notin \mathscr{O}^{\Delta}(\o_0)$.
Furthermore, by the same arguments as in Subcase I.C, $b_1-\o_2
\notin \mathscr{O}^{\Delta}(\o_0)$.  Hence \eqref{oaa} proved in
Subcase II.A stays valid in this case and by \eqref{xieta} and
\eqref{buzz}, $r_y(\o)\neq \infty$ for all $\o$ with $\Im \o=0$ and
$\o_{y_4}\leq \Re \o \leq \o_{y_4}+\o_2$, except for $\o=a_1,a_2$.
By the identity \eqref{sqs2}, $r_x(\o) \ne \infty$
   for all $\o$ with $\Im \o=0$
and $\o_{x_1}\leq \Re \o \leq \o_{x_4}+\o_2$, except for
 points $\o$ where $x(\o)=\infty$, namely $\o=b_2$.
   This concludes the proof of Proposition \ref{ph1} and
   proves in particular that $r_y(\o_0)\ne \infty$ for any
   $\o_0 \in \{a_3,a_4,b_1,b_2\}$.

  Using $x^\s \ne x^\ss$, we verify as in  Subcase I.A
that $\lim_{\o \to a_4}\{f_y(\o)+f_y(\o+\o_3)\}=\infty$ and
 $\lim_{\o \to a_3}\{f_y(\o)+f_y(\o+\o_3)\}=\infty$.
  If $x_4>0$, since $y^\d \ne y^\dd$, we verify as in Subcase I.A
that $\lim_{\o \to b_1}f_y(\o)=\infty$ and
 $\lim_{\o \to b_2}f_y(\o)=\infty$.
 If $x_4=\infty$, then $b_1=b_2$ and we verify as in Subcase I.B
  that $\lim_{\o \to b_1}f_y(\o)=\infty$.
    If $a_3, a_4, b_1, b_2$ are ordered (e.g., $a_3 \ll b_1 \ll a_4
\ll b_2$, then immediately $a_3 \ll_1 a_1 \ll b_1 \ll a_4 \ll_1 a_2
\ll b_2$), there is a {\it maximal} point {\it in the sense of
this order}.
   If the maximal element is $b_\ell$ for some $\ell\in\{1,2\}$, then
  there is no $\o \in \Pi_y$ with $f_y(\o)=\infty$ such that
  $b_\ell \ll \o$.
  If the maximal element is $a_3$ (resp.\ $a_4$), then
there is no $\o \in \Pi_y$ except for $a_1$ (resp.\ $a_2$) with
$f_y(\o)=\infty$ such that
  $a_3 \ll \o$ (resp.\ $a_4 \ll \o$).
  Lemma \ref{dop} applies with $\o_0$ equal this maximal element since
  all assumptions \ref{(i)}, \ref{(ii)} and \ref{(iii)} are satisfied.
    If $a_3,a_4,b_1,b_2$ are not all ordered, then it is enough to apply Lemma
\ref{dop} to the maximal element of any ordered subset.
   Finally $\mathcal{I}_x(\o_0)={\bf R}\setminus ]x_1,x_4[$
and $\mathcal{I}_y(\o_0)={\bf R}\setminus ]y_4,y_1[$ for {any}
$\o_0\in \{a_3,a_4, b_1,b_2\}$, hence the set of poles of $x\mapsto
Q(x,0)$ (resp.\ $y\mapsto Q(0,y)$) is dense on the announced intervals.
  Proposition \ref{ph2} is proved.

\medskip

\noindent\underline{\it Case {\rm II}: $y_4>0$, Subcase {\rm II.D}:
$x_4=\infty$ and $Y(x_4)=\infty$.} In this case $x^\s=\infty$;
$y^\s=\infty$; $x^\ss, y^\ss \neq \infty$.

We have $a_1=\o_{x_4}+\o_2$, and
 $a_2=\widehat \eta a_1-\o_1=\o_{x_4}+\o_3$ is
  symmetric to $a_1-\o_2$ w.r.t.\ $\o_{y_4}$.
  Further, since $y_4>0$ and
  $Y(x_1)\leq 0$,  $a_2\in ]\o_{y_4}, \o_{x_1}[$.
  Then $a_3=\widehat \xi a_2-\o_1=\o_{x_4}+\o_2-\o_3$
  is symmetric to $a_2$ w.r.t.\ $\o_{x_1}$.
   Finally, $a_4=\o_{x_4}+\o_2=a_1$,
    $b_1=b_2=a_1$, $a_3+\o_3=a_1$ and $a_1+\o_3=a_2+\o_2$,
    so that $a_3 \ll_1 a_1 \ll_1 a_2$, see Figure \ref{LoclocIID}.

We prove \eqref{oaa}. For this purpose, we first show that
$r_y(a_3)\neq \infty$. If $a_3 \in \Delta$, $x^\ss, y^\ss \ne
\infty$, then by Theorem \ref{thmde}, $r_y(a_3)\ne \infty$
  If $a_3 \notin \Delta$, consider $\mathscr{O}^{\Delta}(\o_0)$.
   Since $a_3+2\o_3=a_2+\o_2$,  $a_2\notin \mathscr{O}^{\Delta}(\o_0)$
     by the irrationality of $\o_2/\o_3$. Obviously $a_1-\o_2
  =\o_{x_4}\notin \Delta$ and then it is not in $\mathscr{O}^{\Delta}(\o_0)$.
  Hence $r_y(a_3)\neq \infty$.
  Next we prove  that $r_y(a_2)+f_y(a_2)$ exactly as in Subcase II.A
    by using that $x^\ss y^\ss \ne \infty$.

   For any $\o_0\ne a_1$ with $\Im \o_0=0$ and $\o_{y_1}\leq \o_0 \leq
   \o_{y_4}+\o_2$, the orbit $\mathscr{O}^{\Delta}(\o_0)$ cannot
   contain $a_3$, since $a_3+\o_3=a_1$ and $a_3+2\o_3>\o_0$. It can
   neither contain $a_1$, since $\o_0-\o_3 <a_1$, nor obviously
   $a_1-\o_2=\o_{x_4}$. If it does not contain $a_2$, then by
   \eqref{samwino}, $r_y(\o_0)\ne \infty$. If it does, then
    exactly as in Subcase II.A, using $r_y(a_2)+f_y(a_2)\ne
    \infty$, we prove that $r_y(\o_0)\ne \infty$ as well.
    This finishes the proof of \eqref{oaa}.

\begin{figure}[t]
\begin{center}
\begin{picture}(00.00,645.00)
\hspace{-102mm}
\includegraphics{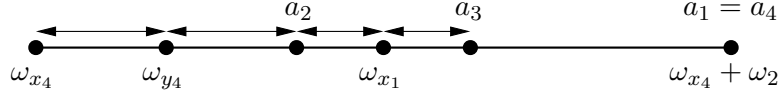}
\end{picture}
\end{center}
\vspace{-213mm} \caption{Location of $a_1,a_2,a_3,a_4$ if $y_4>0$,
$x_4=\infty$, $Y(x_4)=\infty$, i.e., Subcase II.D} \label{LoclocIID}
\end{figure}

   By Equations \eqref{xieta}, \eqref{buzz} and \eqref{sqs2},
    we derive as in Subcase II.A
   that $r_x(\o_0)\neq \infty$
   for all $\o_0 \in \{\o: \Im \o=0,\, \o_{x_1} \leq \o_0 \leq \o_{x_4}+\o_2\}$
    except for $\o_0$ where $x(\o_0)=\infty$, that is for $a_1$.
    This finishes the proof of
    Proposition~\ref{ph1}
    in this
     case.

       To prove Proposition~\ref{ph2}, we would like to apply Lemma~\ref{dop}
       with $\o_0=a_3$. We have shown that $r_y(a_3) \ne
       \infty$, $a_3 \ll_1 a_1=a_4=b_1=b_2 \ll a_2$, so
        that there is no $\o \in \Pi_y \setminus \{a_1,a_2\}$  such that
        $a_3 \ll \o$ and $f_y(\o)=\infty$.
         It remains to verify  assumption \ref{(ii)} of Lemma \ref{dop} for
          $\o_0=a_3$, that is that
     $f_y(\o)+f_y(\o+\o_3)+f_y(\o+2\o_3)$ converges to infinity if
     $\o \to a_3$. The last quantity is the sum of
     \begin{equation*}
     x(\o)[y(\widehat \xi \o)-y(\o)]+ x(\widehat \eta \widehat \xi \o)[y (\widehat
     \xi \widehat \eta \widehat \xi \o) -y(\widehat \eta \widehat \xi \o)]
      + x( \widehat \eta \widehat \xi \widehat \eta \widehat \xi \o)[y (\widehat
     \xi \widehat \eta \widehat \xi \widehat \eta \widehat \xi \o) -y(\widehat \eta \widehat \xi
     \widehat \eta \widehat \xi \o)],
     \end{equation*}
     which equals
     \begin{equation}
     \label{domin}
     x( \widehat \eta \widehat \xi \widehat \eta \widehat \xi \o)y (\widehat
     \xi \widehat \eta \widehat \xi \widehat \eta \widehat \xi \o)-x(\o)y(\o)
       +x(\widehat \eta \widehat \xi \o)[y (\widehat
     \xi \widehat \eta \widehat \xi \o) -y(\widehat \eta \widehat \xi
     \o)]+ x(\o) y(\widehat \xi \o) - x(\widehat \eta \widehat \xi
          \widehat \eta \widehat \xi \o) y(\widehat \xi \widehat \eta \widehat \xi
          \o)
       \end{equation}
        where we used \eqref{hat}.
   If $\o \to a_3$, then the first term in this sum converges to
   $x^\s  y^\s-x^\s y^\s=0$.
    Next, $\widehat \eta \widehat \xi \o \to a_1$, so that
    $x(\widehat \eta \widehat \xi \o) \to \infty$.
     We can also compute
    the values of $y(\widehat \xi \o)$ and
    $y(\widehat \xi \widehat \eta \widehat \xi \o)$
      as $( -b(x )\pm [b(x)^2 -  4 a(x  ) c( x)]^{1/2})/ (2 a( x))$  with $x=x(\widehat \eta \widehat \xi \o)$.
         Since for all of the $9$ models composing Subcase II.D,
        $ \deg a=\deg b=1$ and
        $\deg c=2$, then
        $y(\widehat \xi \o)$ and
        $y(\widehat \xi \widehat \eta \widehat \xi \o)$ are of
        order $O(|x(\widehat \eta \widehat \xi \o)|^{1/2})$, and their
        difference
         $|y(\widehat \xi \o)-y(\widehat \xi \widehat \eta \widehat \xi
        \o)|$ is not smaller than
        $O(|x(\widehat \eta \widehat \xi \o)|^{1/2})$
        as $\o \to a_3$.
         Finally, $x(\o), x(\widehat \eta \widehat \xi
          \widehat \eta \widehat \xi \o) \to x^\s\ne \infty$ as $\o \to
          a_3$.
          Then as $\o \to a_3$ in the sum \eqref{domin} the second term
           is of the order not smaller than
              $O(|x(\widehat \eta \widehat \xi \o)|^{3/2})$
           while the first vanishes and the third has the order
             $O(|x(\widehat \eta \widehat \xi \o)|^{1/2})$.
           This proves the assumption
          \ref{(ii)} of Lemma \ref{dop} for $\o_0=a_3$.
           By this lemma the poles of $x \mapsto  Q(x,0)$ and $y\mapsto
            Q(0,y)$ are  dense on the intervals of the real line, as
            announced in the proposition.

\medskip

\noindent\underline{\it Case {\rm III}: $y_4=\infty$.} It remains
here exactly one case to study, see Figure \ref{Allcases}. It is
such that $y_4=\infty$, $x_4=\infty$ and $X(y_4)\neq \infty$. Then
$x^\s=x^\ss \neq \infty$; $y^\d=y^\dd\neq \infty$;
$y^\s=y^\ss\ne\infty$.

The points $b_1=b_2=\o_{x_4}+\o_2$ are located as in Subcase I.B,
$a_1=a_2=\o_{y_4}+\o_2$ and $a_3=a_4=\o_{y_4}+\o_2-\o_3$. In
particular, $a_3+\o_3/2=b_1$, $b_1+\o_3/2=a_1$, see
Figure~\ref{LoclocIII}.

\begin{figure}[t]
\begin{center}
\begin{picture}(00.00,645.00)
\hspace{-110mm}
\includegraphics{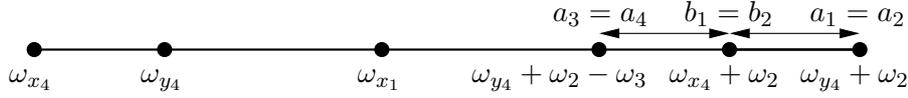}
\end{picture}
\end{center}
\vspace{-213mm} \caption{Location of $a_1,a_2,a_3,a_4,b_1,b_2$ if
$y_4=\infty$, Case III}
 \label{LoclocIII}
\end{figure}

  We start by showing that $r_y(a_3)\neq\infty$. If $a_3 \in \Delta$, this
is true thanks to \eqref{sqs2} and since $x^\s,y^\s\neq \infty$. If
$a_3 \notin \Delta$, consider the orbit
$\mathscr{O}^{\Delta}(a_3)$.
     It cannot contain $a_1-\o_2$ since $a_3+\o_3=a_1$, neither
 $b_1$, nor $b_1-\o_2=\o_{x_4}$.
 It follows that $r_y(a_3)\neq\infty$.

Since $x^\s,y^\s\neq \infty$, it follows from Equations
\eqref{sqs2},
 \eqref{xieta} and \eqref{buzz} that
 $r_x(a_1-\o_2)=r_x(\widehat \xi (a_1-\o_2))=
 r_x(\widehat \xi (a_1-\o_2)-\o_1)=r_x(a_3) \neq \infty$.
 Then, by \eqref{sqs2},
     $r_y(a_1-\o_2)+f_y(a_1-\o_2)=-r_x(a_1-\o_2)+K(0,0)Q(0,0)
 +x(a_1-\o_2)y(\widehat \xi (a_1-\o_2))\neq\infty$.

 Take any $\o_0$ with $\Im \o_0=0$ and $\o_{y_1} \leq \o_0 < \o_{y_4}+\o_2$.
  If $\o_0 \in \Delta$, then by the same arguments as in Subcase I.A,
  $\o_0 \in \Delta_y$, so that $r_y(\o_0)\neq\infty$.
  Otherwise, we notice that
   $\o_0 -\o_3 <a_3$, so that no point---and in particular $b_1$---of $[\o_{y_4}+\o_2-\o_3, \o_{y_4}+\o_2[$
    belongs to $\mathscr{O}^\Delta(\o_0)$.
    Clearly $b_1-\o_2=\o_{x_4}\notin \mathscr{O}^\Delta(\o_0)$.
     Since either $a_1-\o_2 \notin \mathscr{O}^\Delta(\o_0)$ or
  $a_1-\o_2  \in \mathscr{O}^\Delta(\o_0)$
     but $r_y(a_1-\o_2)+f_y(a_1-\o_2)\neq \infty$,
     and by the same reasoning as in Subcase I.A,
      we derive that $r_y(\o_0)\neq \infty$. The rest of the proof
        of Proposition \ref{ph1} in this case goes along the
        same lines as in Case II.

   Now note that $b_1$ is not ordered with $a_1$ and $a_3$. Indeed,
   since $b_1+\o_3/2=a_1$ and $b_1-\o_3/2=a_3$, this would
   contradict the irrationality of $\o_2/\o_3$.
    Then there is no $\o \in \Pi_y$ such that $b_1 \ll \o$ and
     $f_y(\o)=\infty$. We also have $f_y(\o)=x(\o)[y(\widehat \xi
      \o)-y(\o)]=x(\o) [b^2(x(\o))-4 a(x(\o))c(x(\o))]^{1/2}/
      a(x(\o))$.  If $\o \to b_1$, then $x(\o) \to \infty$ and
      since $\deg a=2$ and $\deg b=\deg c=1$, we have
      $f_y(\o)\to \infty$. Lemma \ref{dop} applies with $\o_0=b_1$
      and proves Proposition~\ref{ph2}.

\subsection{\bf Asymptotic of $\o_2(\o)/\o_3(\o)$ as $z \to 0$}

It remains to prove the following result announced at the beginning
of Section~7.
\begin{prop}
\label{nncc} For any of $51$ non-singular walks having an infinite
group, there exist a rational constant $L>0$ and a constant
$\widetilde L\ne 0$ such that
\begin{equation}
\label{btcn} \o_2/\o_3=L+\widetilde L /\ln (z)+O(1/\ln(z))^2.
\end{equation}
\end{prop}

\begin{proof}
In order to prove \eqref{btcn}, we shall use expressions of the
periods $\o_2$ and $\o_3$ different from that given in
\eqref{expression_omega_1_2} and \eqref{expression_omega_3}. To that
purpose,
 define the complete and incomplete elliptic integrals of the first kind by, respectively,
\begin{align}
     K(k)=\int_0^1&\frac{\text{d}t}{[1-t^2]^{1/2}[1-k^2t^2]^{1/2}},\label{cei}\\
     F(w,k)=\int_0^w&\frac{\text{d}t}{[1-t^2]^{1/2}[1-k^2t^2]^{1/2}}\label{iei}.
\end{align}
Then the new expressions of $\o_2$ and $\o_3$ are
\begin{equation}
\label{after_change_1}
\o_2 = M\Omega_2,\qquad
\o_3 = M\Omega_3,
\end{equation}
where
\begin{align}
     \Omega_2&= K\left(\sqrt{\frac{(x_4-x_1)(x_3-x_2)}{(x_4-x_2)(x_3-x_1)}}\right),\label{after_change_2}\\
     \Omega_3&= F\left(\sqrt{\frac{(x_4-x_2)(x_1-X(y_1))}{(x_4-x_1)(x_2-X(y_1))}},\sqrt{\frac{(x_4-x_1)(x_3-x_2)}{(x_4-x_2)(x_3-x_1)}}\right)\label{after_change_3},
\end{align}
and where (below, ${\bf 1}_{(i,j)}=1$ if $(i,j)\in\mathcal{S}$, otherwise $0$)
\begin{equation}
\label{after_change_4}
M=\left\{\begin{array}{cc}\displaystyle
\frac{2}{z}\frac{1}{\sqrt{({\bf 1}_{(1,0)}-4{\bf 1}_{(1,1)}{\bf 1}_{(1,-1)})(x_3x_4-x_2x_3-x_1x_4+x_1x_2)}} & \text{if}\ x_4\neq\infty,\\ \displaystyle
\frac{2}{\sqrt{(2z{\bf 1}_{(1,0)}+4z^2[{\bf 1}_{(1,1)}{\bf 1}_{(0,-1)}+{\bf 1}_{(1,-1)}{\bf 1}_{(0,1)}])(x_3-x_1)}} & \text{if}\ x_4=\infty.
\end{array}\right.
\end{equation}
The expressions of $\o_2$ and $\o_3$ written in \eqref{after_change_1},
\eqref{after_change_2}, \eqref{after_change_3} and \eqref{after_change_4} are obtained
from \eqref{expression_omega_1_2} and \eqref{expression_omega_3} by making simple changes of variables.

We are now in position to analyze the behavior of $\o_2/\o_3$
(or equivalently, thanks to \eqref{after_change_1}, that of $\Omega_2/\Omega_3$)
in the neighborhood of $z=0$.
First, with \eqref{d_d_tilde}
and \cite[Proposition 6.1.8]{STAN},
 we obtain that as $z\to 0$, $x_1,x_2\to 0$ and $x_3,x_4\to \infty$. For this reason,
\begin{equation}
\label{t11}
k=\sqrt{\frac{(x_4-x_1)(x_3-x_2)}{(x_4-x_2)(x_3-x_1)}}\to 1.
\end{equation}
The behavior of $X(y_1)$ as $z\to 0$
is not so simple as that of the branch points
$x_\ell$ (indeed, as $z\to 0$, $X(y_1)$
 can converge to $0$, to $\infty$ or to some non-zero constant),
  but we can show that for all $51$ models,
\begin{equation}
\label{t12}
w=\sqrt{\frac{(x_4-x_2)(x_1-X(y_1))}{(x_4-x_1)(x_2-X(y_1))}}\to 1.
\end{equation}
Due to \eqref{t11} and \eqref{t12},
in order to determine the behavior of $\Omega_2/\Omega_3$ near $z=0$
it suffices to know
\begin{enumerate}[label={\rm (\roman{*})},ref={\rm (\roman{*})}]
\item \label{exp1} the expansion of $K(k)$ as $k\to 1$;
\item \label{exp2} the expansion of $F(w,k)$ as $k\to 1$ and $w\to 1$.
\end{enumerate}

Point \ref{exp1} is classical, and is known as Abel's
identity (it can be found, e.g., in \cite{SG2}):
 there exist two functions $A$ and $B$, holomorphic at $z=0$,
 such that $K(k)=A(k)+\ln(1-k)B(k)$. Both $A$
 and $B$ can be computed in an explicit way, see \cite{SG2},
  and from all this we can deduce an expansion of $K(k)$ as $k\to 1$
  up to any level of precision. For our purpose, it will be enough to use the following:
\begin{align*}
A(k)&=(3/2)\ln(2)+((k-1)/4)(1-3\ln(2))+O(k-1)^2,\\
B(k)&=-1/2+(k-1)/4+O(k-1)^2.
\end{align*}

As for Point \ref{exp2}, we proceed as follows. We have $F(w,k)=K(k)-\widetilde F(w,k)$, with
\begin{equation*}
\widetilde F(w,k)=\int_w^1\frac{\text{d}t}{[1-t^2]^{1/2}[1-k^2t^2]^{1/2}}.
\end{equation*}
Then, introduce the expansion
$
{1}/({[1+t]^{1/2}[1+kt]^{1/2}})= \sum_{\ell=0}^\infty \mu_\ell(k)(1-t)^\ell
$, so that
\begin{equation}
\label{intsom}
\widetilde F(w,k)=\sum_{\ell=0}^\infty \mu_\ell(k)\int_w^1\frac{\text{d}t}{[1-t]^{1/2-\ell}[1-kt]^{1/2}}.
\end{equation}
In Equation \eqref{intsom}, all $\mu_\ell(k)$ as well as all integrals can be computed.
 As an example (that we shall use), we have
\begin{multline*}
\mu_0(k)\int_w^1\frac{\text{d}t}{[1-t]^{1/2}[1-kt]^{1/2}}=\frac{1}{[2k(1+k)]^{1/2}}\times \\
\times [\ln\{(1-k)/k^{1/2}\}-\ln\{-[1-(1+k)w+kw^2]^{1/2}+[(k+1)/2-kw]/k^{1/2}\}].
\end{multline*}
Moreover, it should be noticed that as $k\to 1$ and $w\to 1$, the
speed of convergence to zero of the integrals in \eqref{intsom}
increases with $\ell$. This way, we can write an expansion of
$\widetilde F(w,k)$---and thus of $F(w,k)$---up to any level of
precision.

Unfortunately, the end of the proof cannot be done simultaneously for all $51$ models,
 but should be done model by model.
 For the sake of shortness, we choose to present
 the details only for one model, namely for the model
 with $\mathcal{S} = \{(-1,0),(-1,1),(0,1),(1,-1)\}$
  (which belongs to Subcase II.D of Figure \ref{Allcases}).
  For this model, we easily obtain from \eqref{d_d_tilde} that
\begin{align*}
x_1 &= z-2z^2+3z^3+O(z^4),\\
x_2 &= z+2z^2+5z^3+O(z^4), \\
x_3 & = 1/(4z^2)-1-2z-8z^3 +O(z^4), \\
x_4 & = \infty,\\
X(y_1) &=0.
\end{align*}
Then, with \eqref{t11} and \eqref{t12}, we reach the conclusion that
\begin{align}
k&=1-8z^4-4z^5+O(z^6),\label{t11t}\\
w&=1-2z+z^2-(7/4)z^3-(65/8)z^4+(613/64)z^5+O(z^6)\label{t12t}.
\end{align}
Then, using Points \ref{exp1} and \ref{exp2} above, we obtain
 \begin{align*}
\Omega_2&=-2\ln(z)-(1/4)z+(1/16)z^2+O(z^3\ln(z)),\\
\Omega_3&=-(1/2)\ln(z)-(1/2)\ln(2)+(1/4)z+(57/16)z^2+O(z^3\ln(z)),
\end{align*}
so that
\begin{equation}
\label{LastExp}
\Omega_2/\Omega_3 = 4-4\ln(2)/\ln(z)+O(1/(\ln(z))^2).
\end{equation}
The latter proves \eqref{btcn}, and thus Proposition \ref{nncc},
with $L=4$ and $\widetilde L=-4\ln(2)$.

Making expansions of higher order of $k$ and $w$ in \eqref{t11t} and \eqref{t12t}, we could obtain more terms in the expansion \eqref{LastExp} of $\Omega_2/\Omega_3$. A contrario, we could also be interested in obtaining the first term only (the constant term $L$) in \eqref{LastExp}. (Indeed, we saw in the proof of Proposition \ref{proir} that it was sufficient for our purpose, i.e., for proving that in the infinite group case, the ratio $\omega_2/\omega_3$ is not constant in $z$.) To that aim, instead of \eqref{t11t} and \eqref{t12t}, we just need  two-terms expansions of $k$ and $w$,  say $k=1+\alpha z^p+o(z^p)$ and $w=1+\beta z^q+o(z^q)$, with $\alpha,\beta\neq 0$. Then with \ref{exp1} and \ref{exp2} we deduce that $\Omega_2 = -(p/2) \ln(z) +o(\ln(z))$ and $\Omega_3 = -(q/2)\ln(z)+o(\ln(z))$, in such a way that $L = p/q$, which obviously is (non-zero and) rational. 

\end{proof}

\begin{figure}[t]
\begin{center}
(Case I: $y_4<0$, Subcase I.A: $x_4<0$)
\begin{picture}(420.00,40.00)
\includegraphics{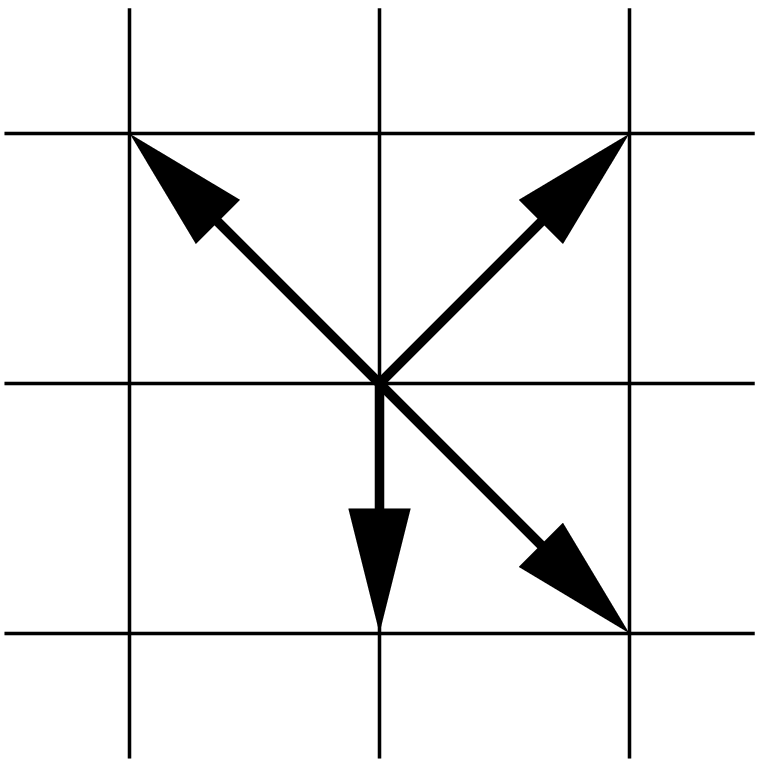}
\hspace{12.5mm}
\includegraphics{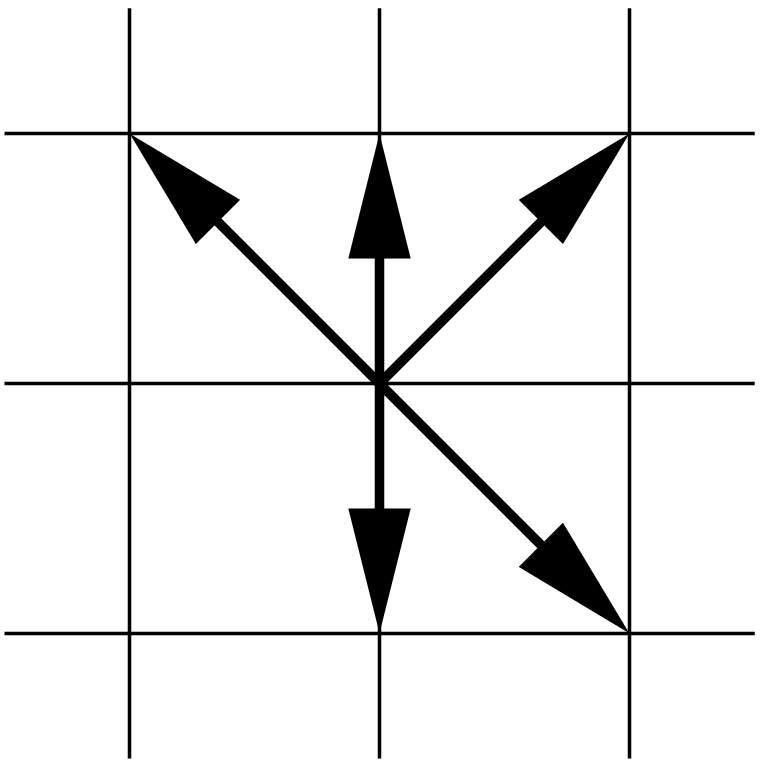}
\hspace{12.5mm}
\includegraphics{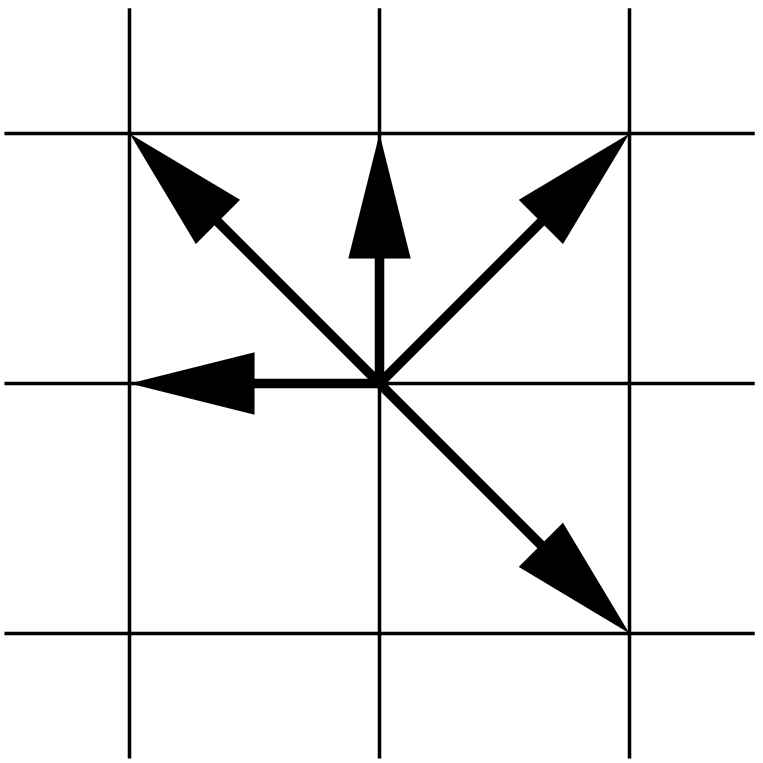}
\hspace{12.5mm}
\includegraphics{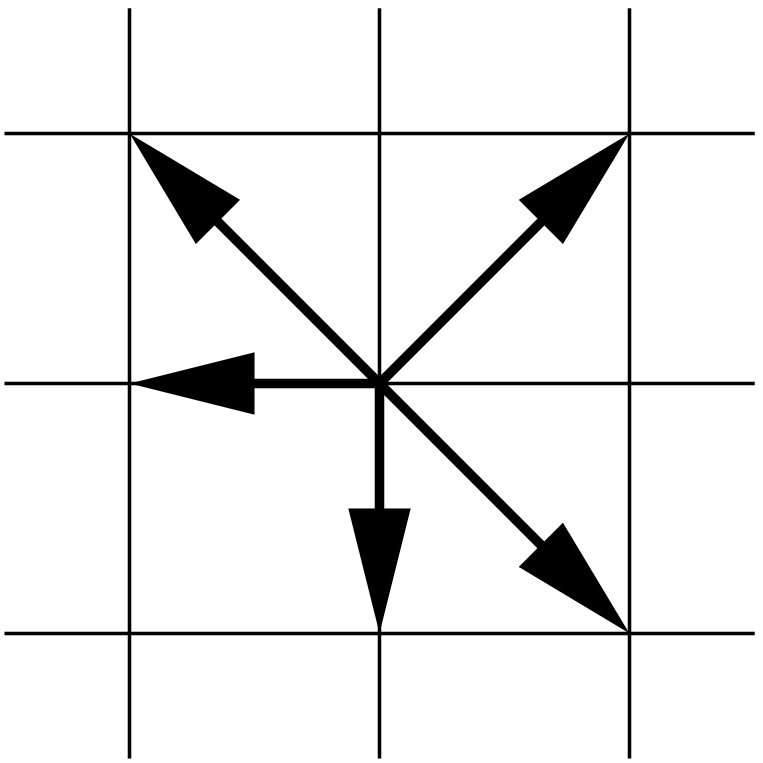}
\hspace{12.5mm}
\includegraphics{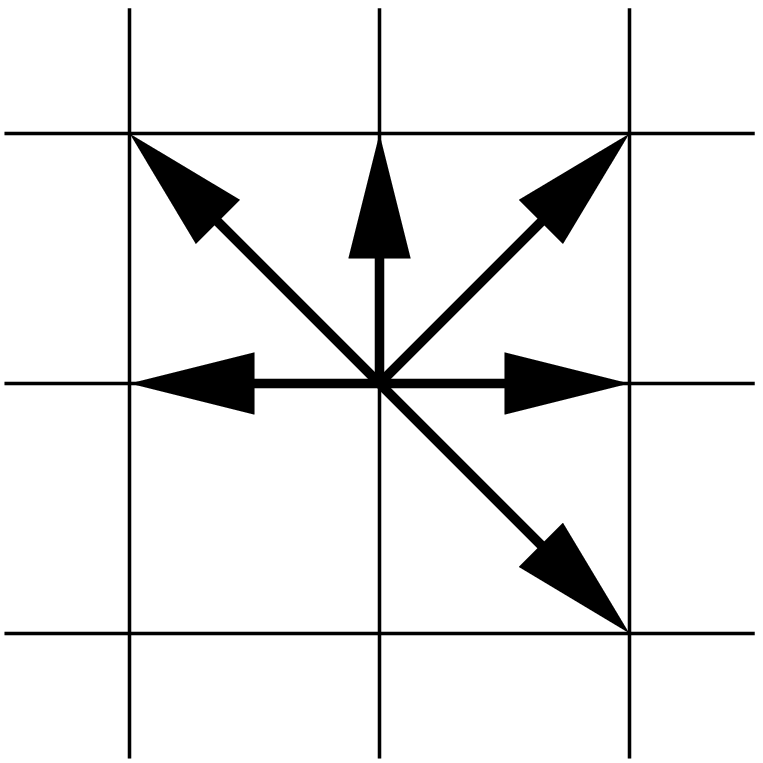}
\hspace{12.5mm}
\includegraphics{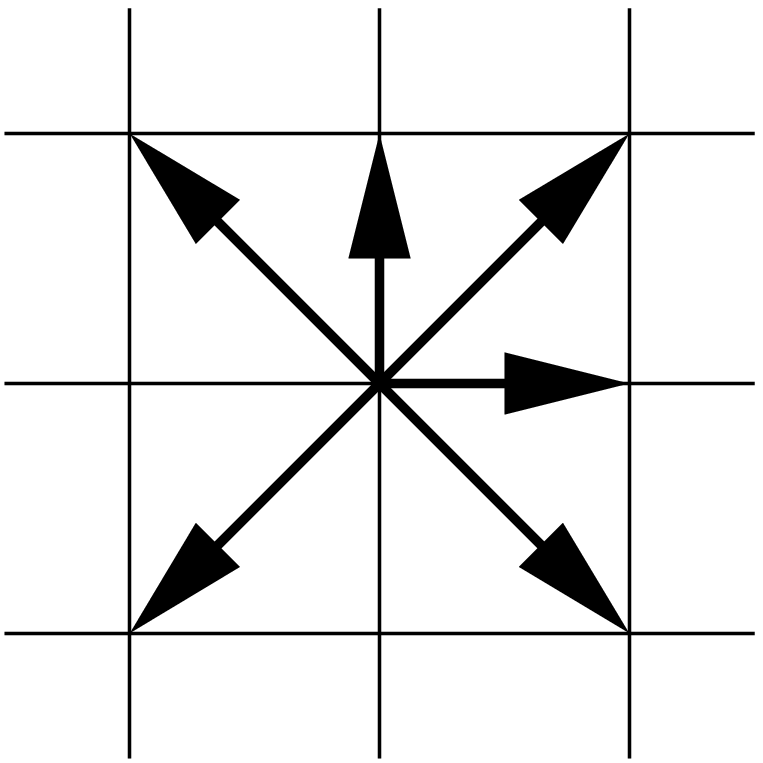}
\hspace{12.5mm}
\includegraphics{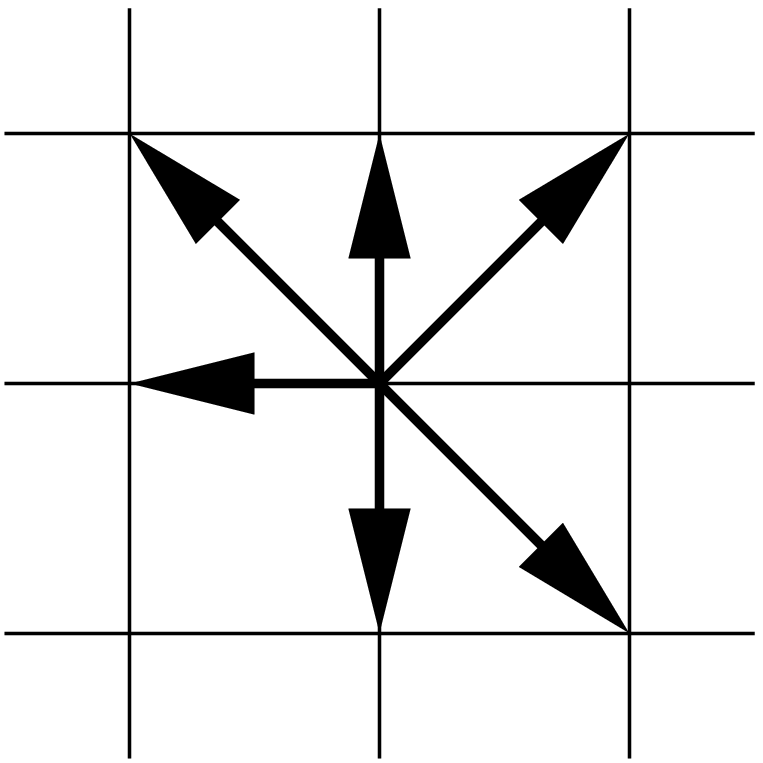}
\hspace{12.5mm}
\includegraphics{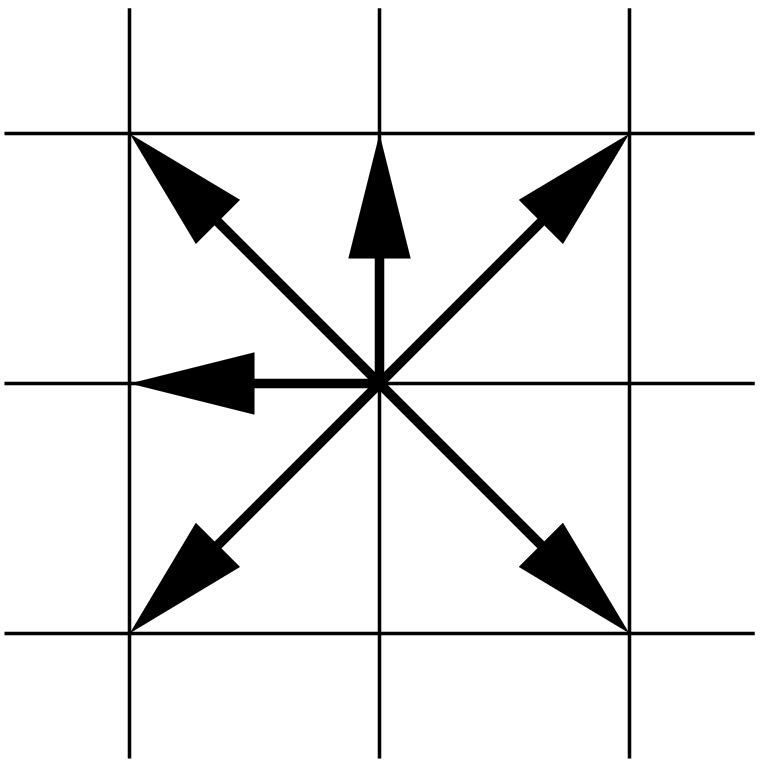}
\hspace{12.5mm}
\includegraphics{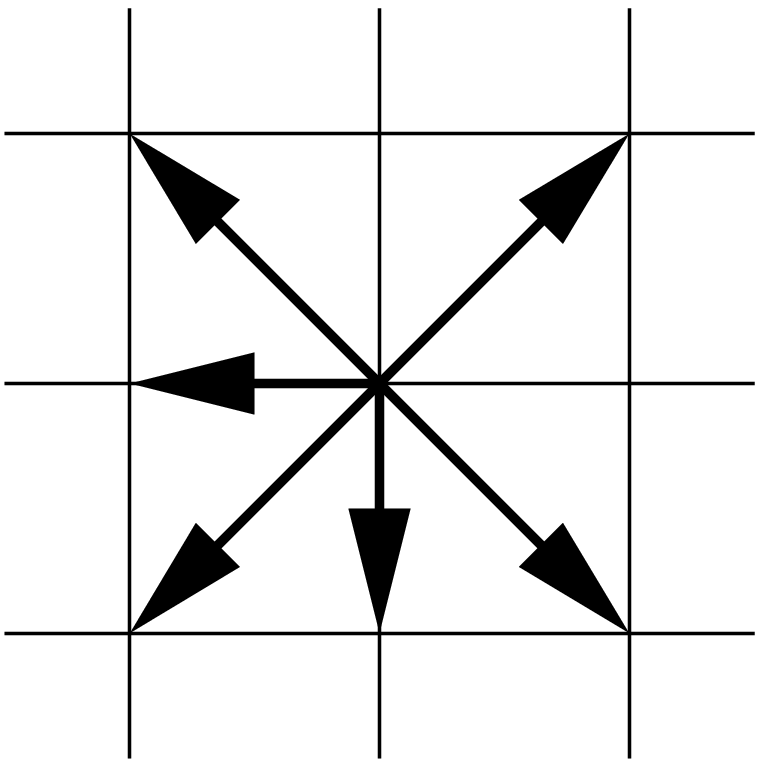}
\hspace{12.5mm}
\includegraphics{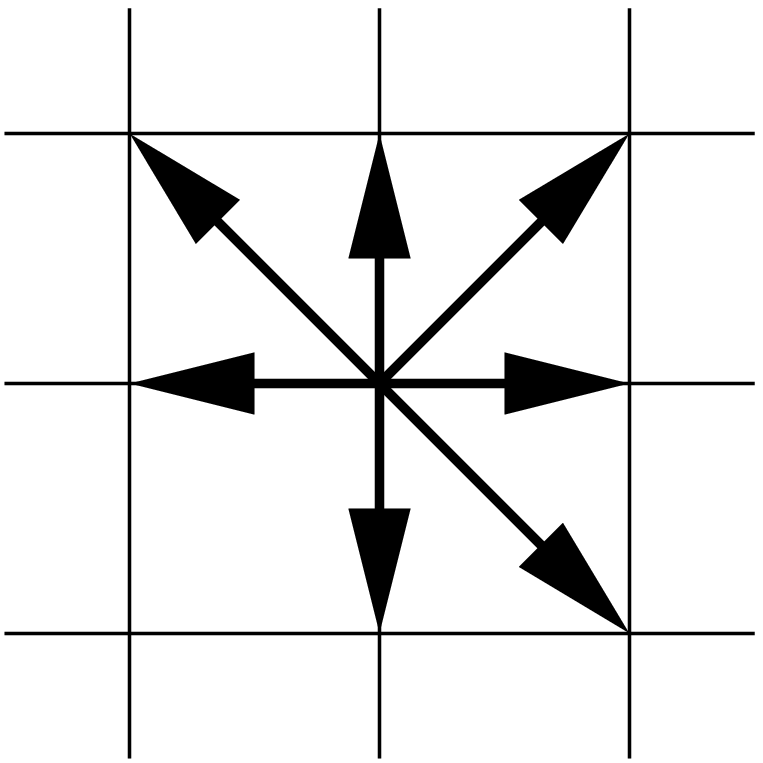}
\end{picture}

 (Case I: $y_4<0$, Subcase I.B: $x_4=\infty$)
\begin{picture}(420.00,40.00)
\hspace{10mm}
\includegraphics{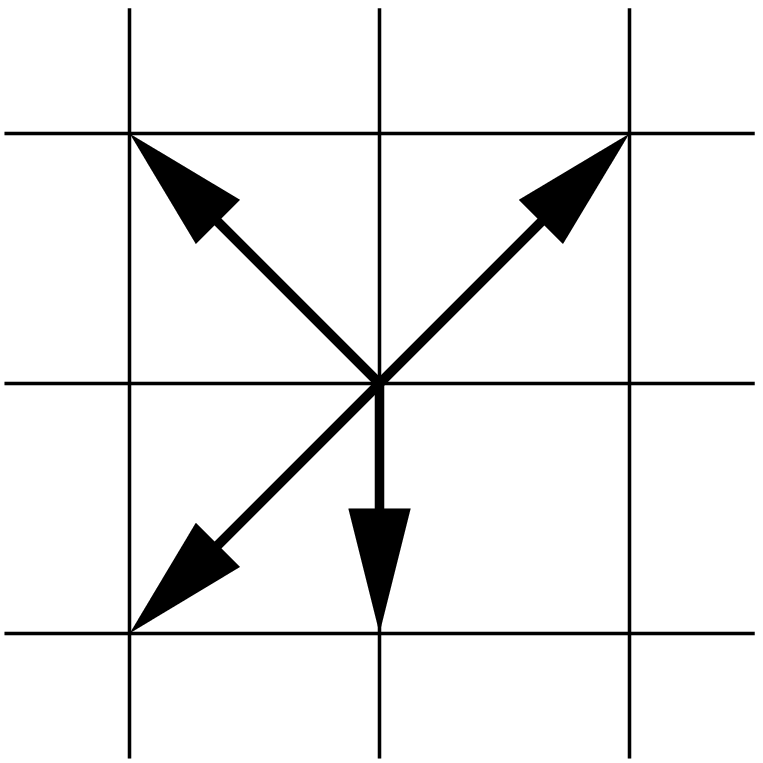}
\hspace{20mm}
\includegraphics{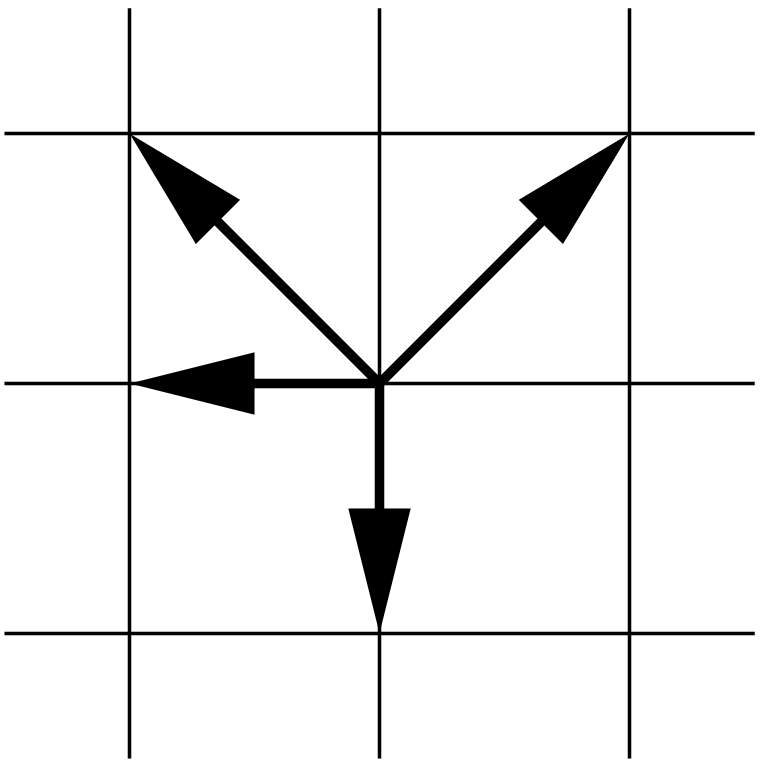}
\hspace{20mm}
\includegraphics{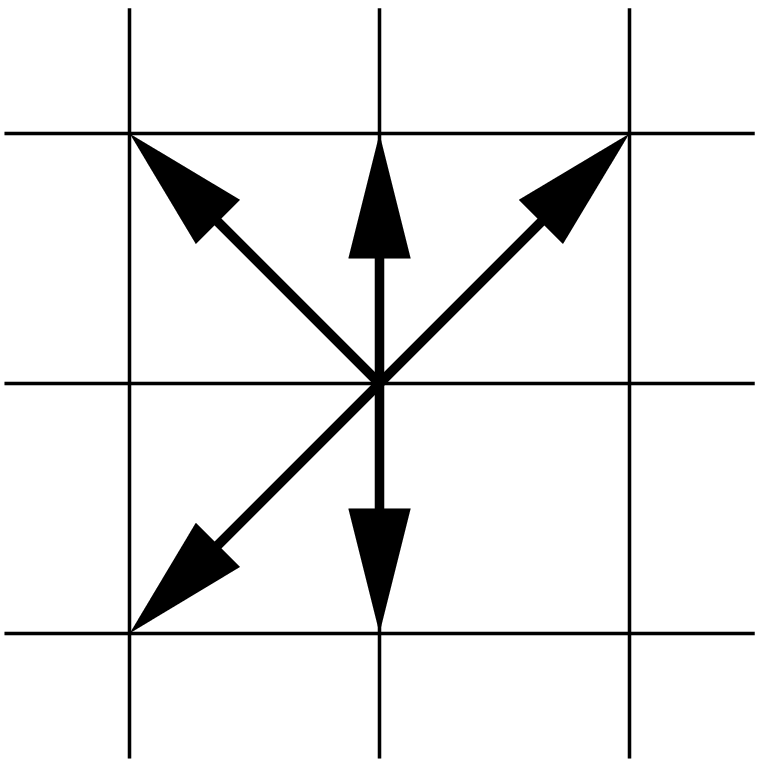}
\hspace{20mm}
\includegraphics{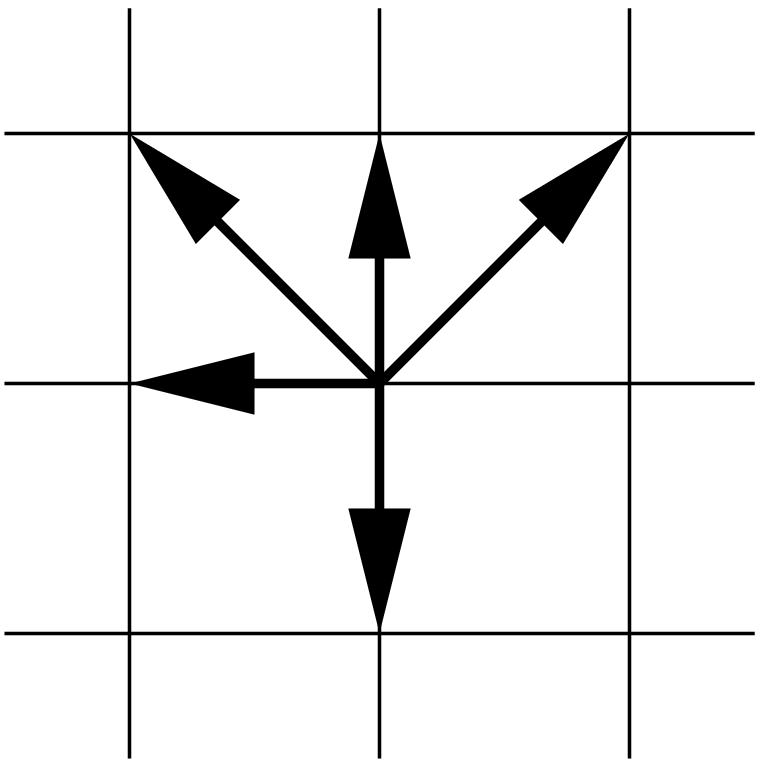}
\hspace{20mm}
\includegraphics{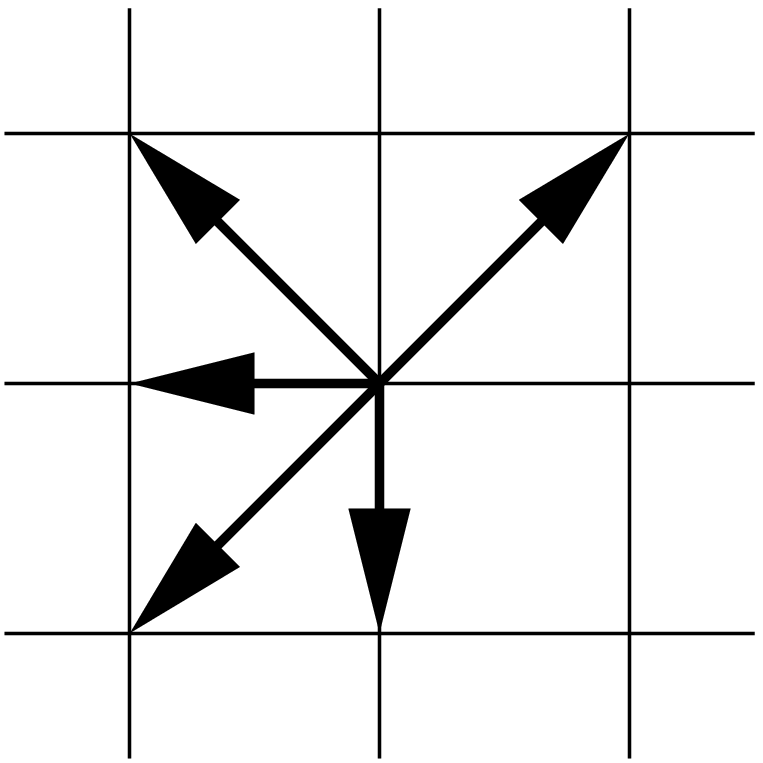}
\hspace{20mm}
\includegraphics{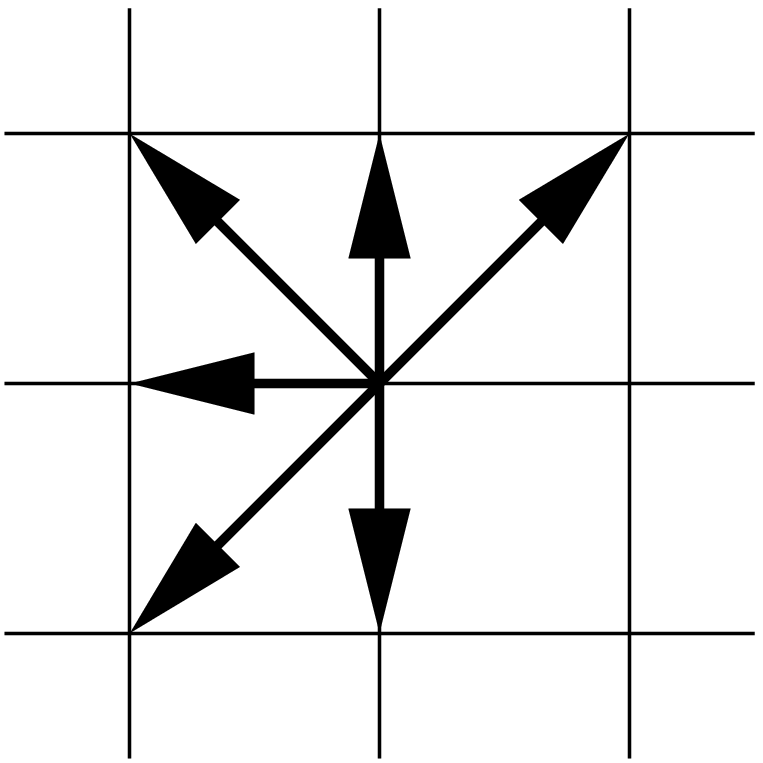}
\end{picture}

(Case I: $y_4<0$, Subcase I.C: $x_4>0$)
\begin{picture}(420.00,40.00)
\hspace{43mm}
\includegraphics{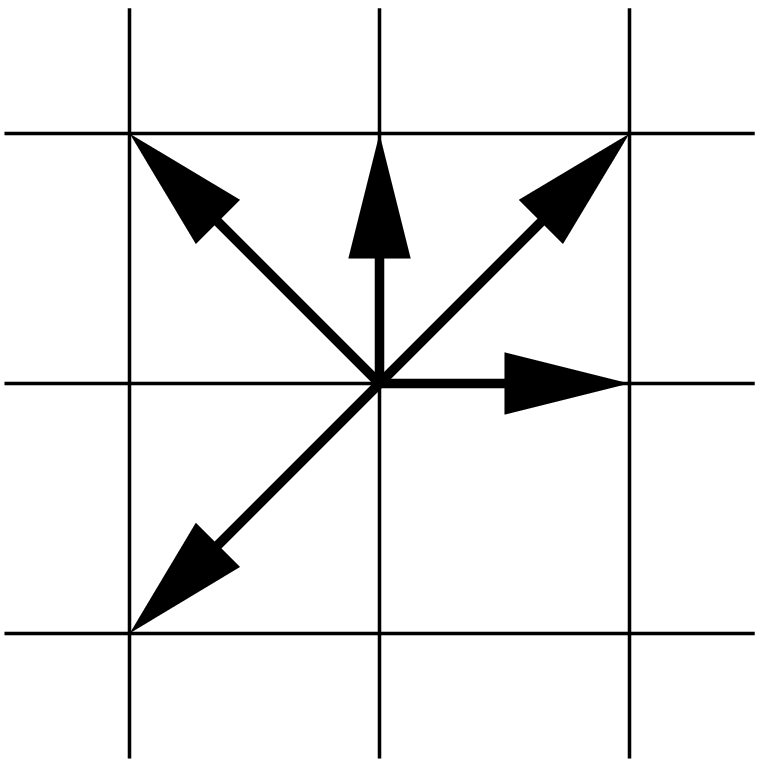}
\hspace{20mm}
\includegraphics{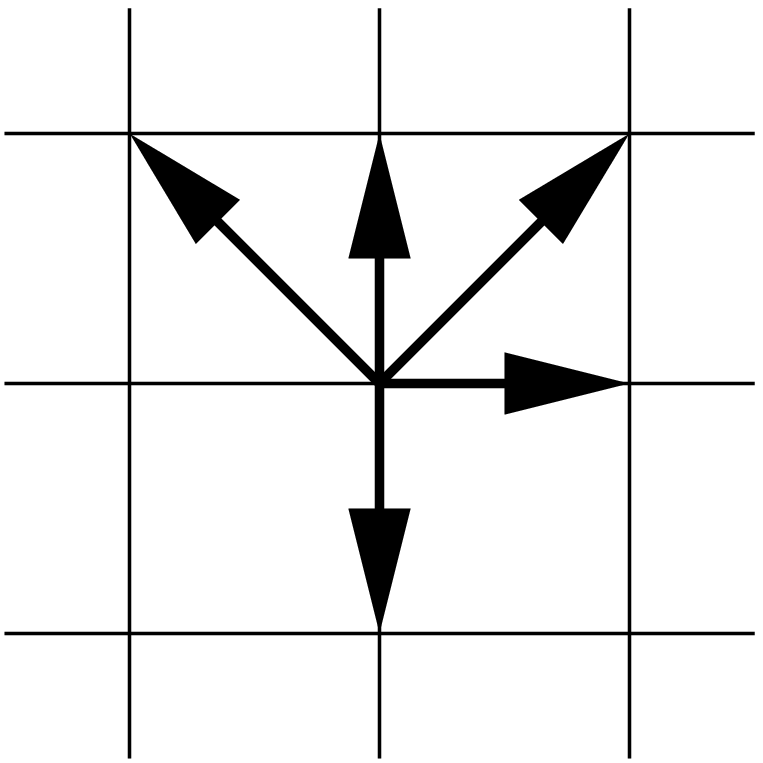}
\hspace{20mm}
\includegraphics{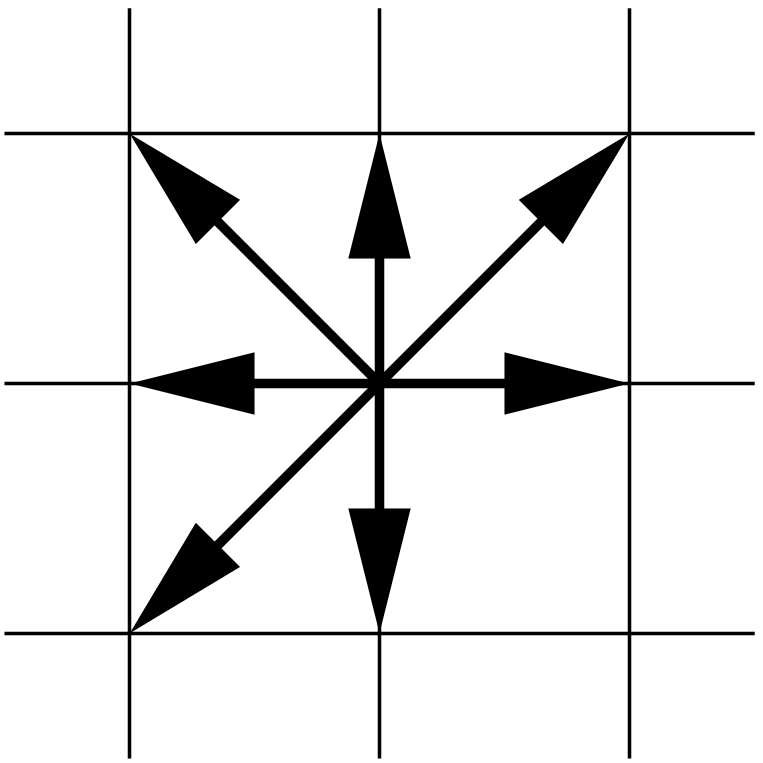}
\end{picture}

(Case II: $y_4>0$, Subcase II.A: $x_4<0$)
\begin{picture}(420.00,40.00)
\hspace{0mm}
\includegraphics{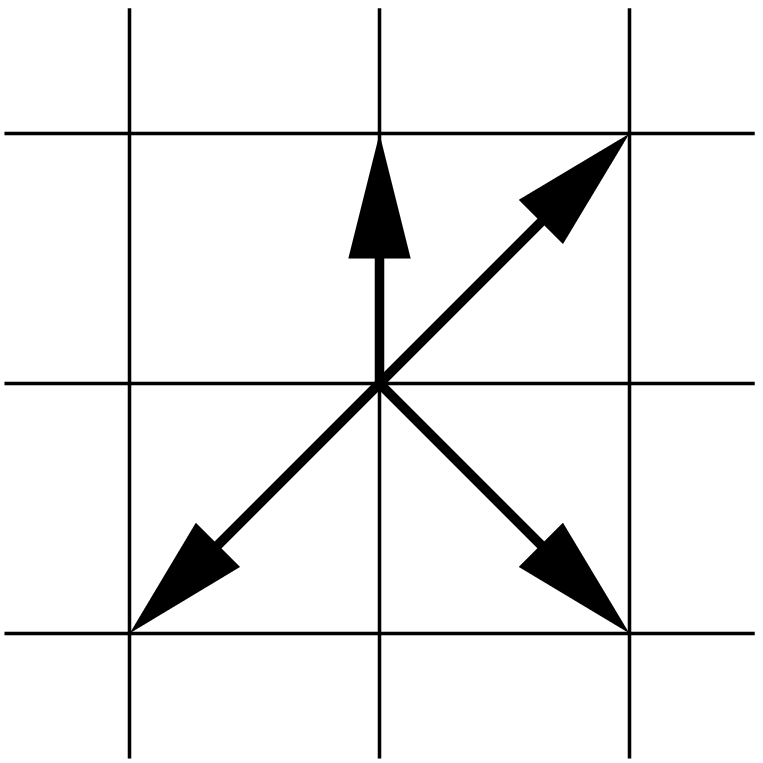}
\hspace{19.5mm}
\includegraphics{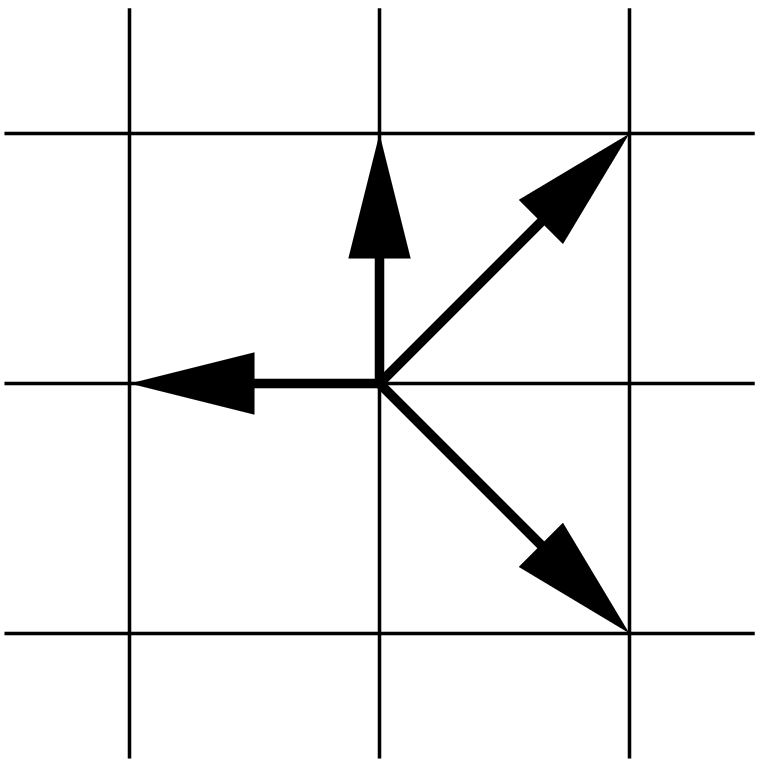}
\hspace{19.5mm}
\includegraphics{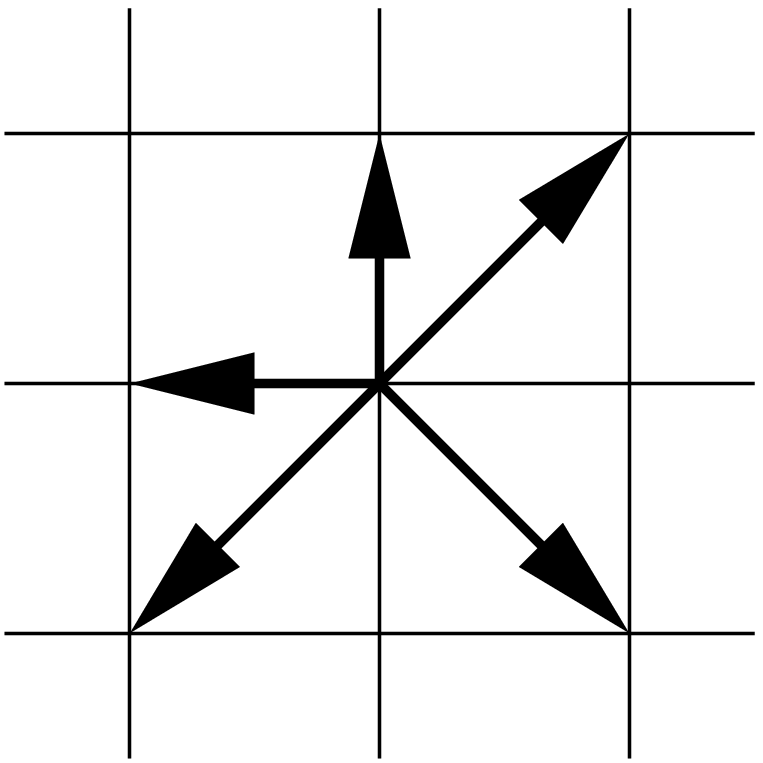}
\hspace{19.5mm}
\includegraphics{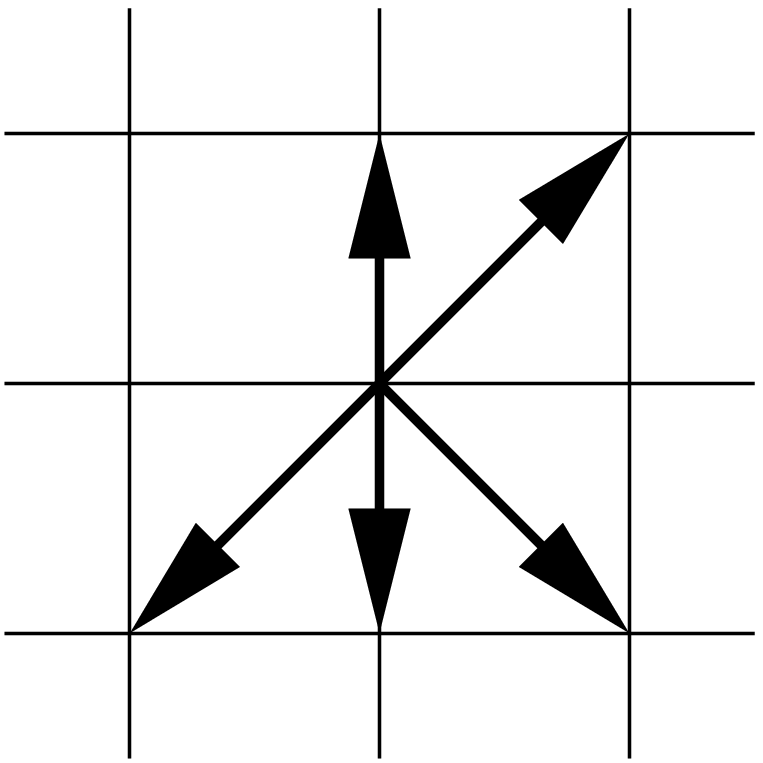}
\hspace{19.5mm}
\includegraphics{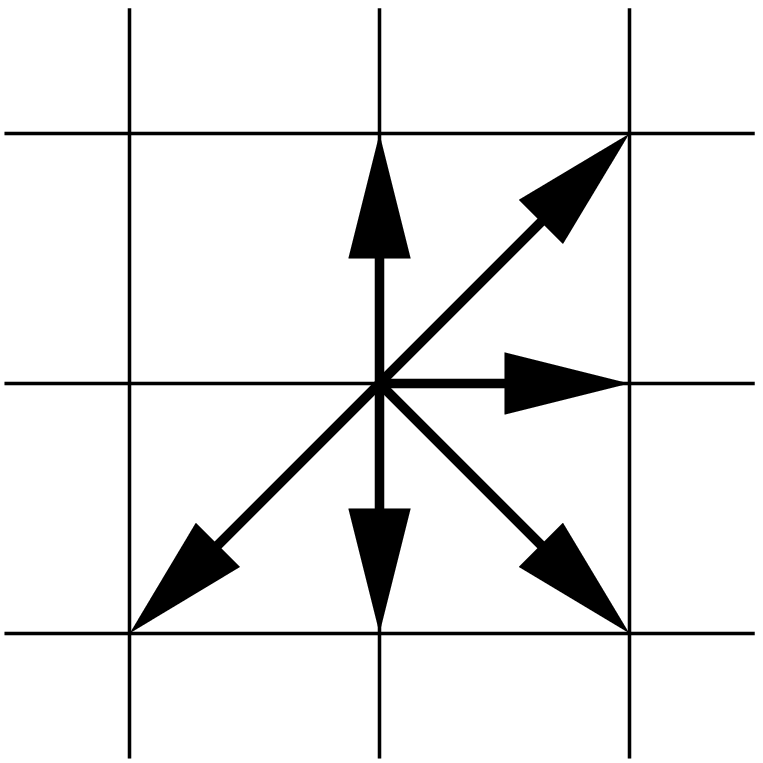}
\hspace{19.5mm}
\includegraphics{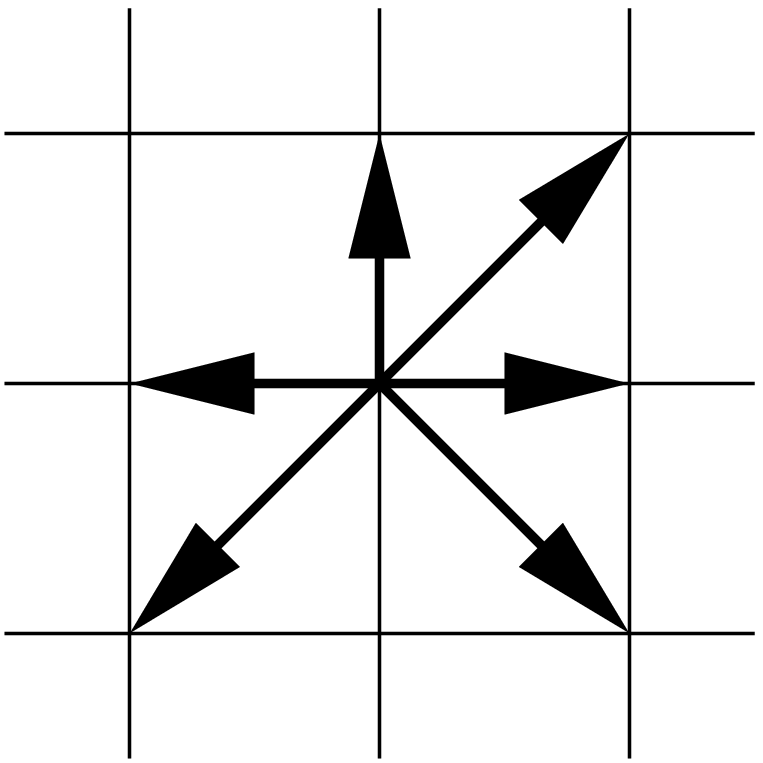}
\hspace{19.5mm}
\includegraphics{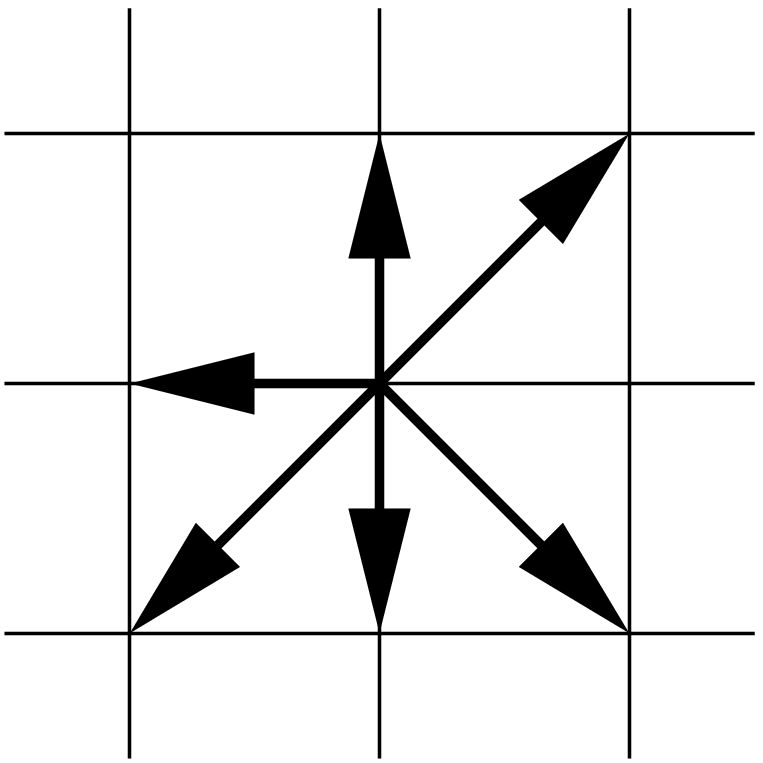}
\end{picture}

(Case II: $y_4>0$, Subcase II.B: $x_4>0$, one of $y^\d,y^\dd$ is
$\infty$)
\begin{picture}(420.00,40.00)
\includegraphics{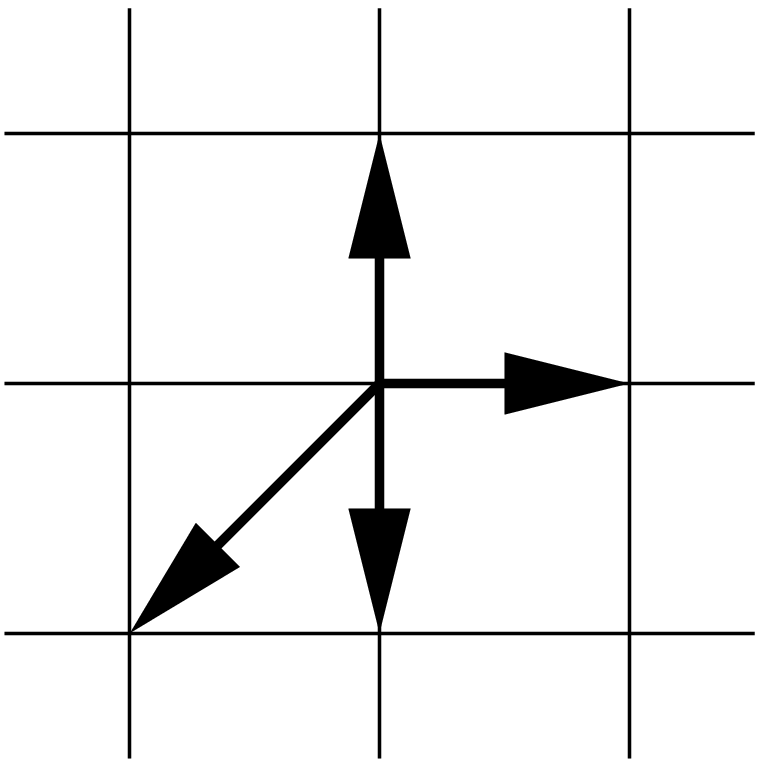}
\hspace{12.5mm}
\includegraphics{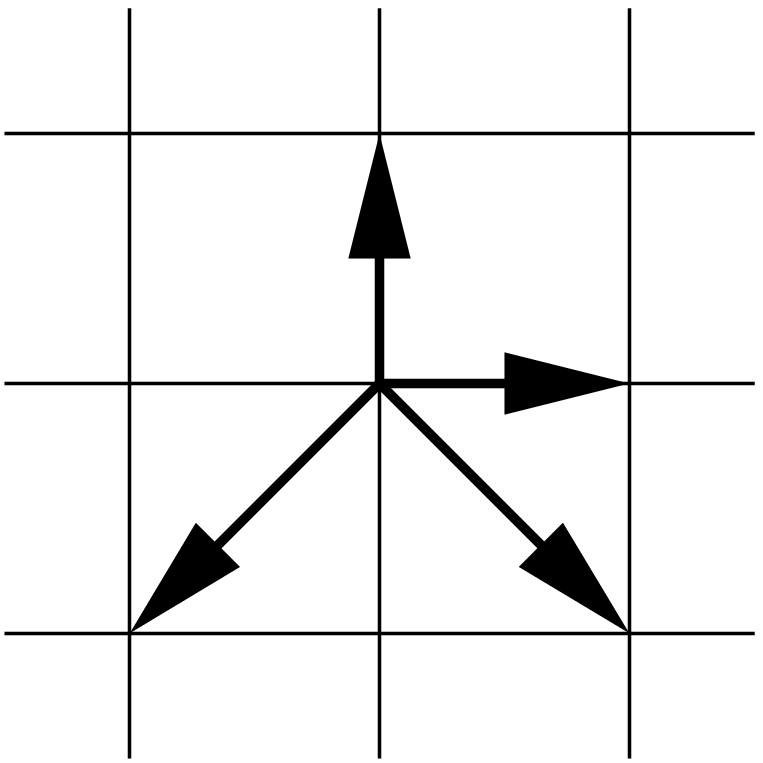}
\hspace{12.5mm}
\includegraphics{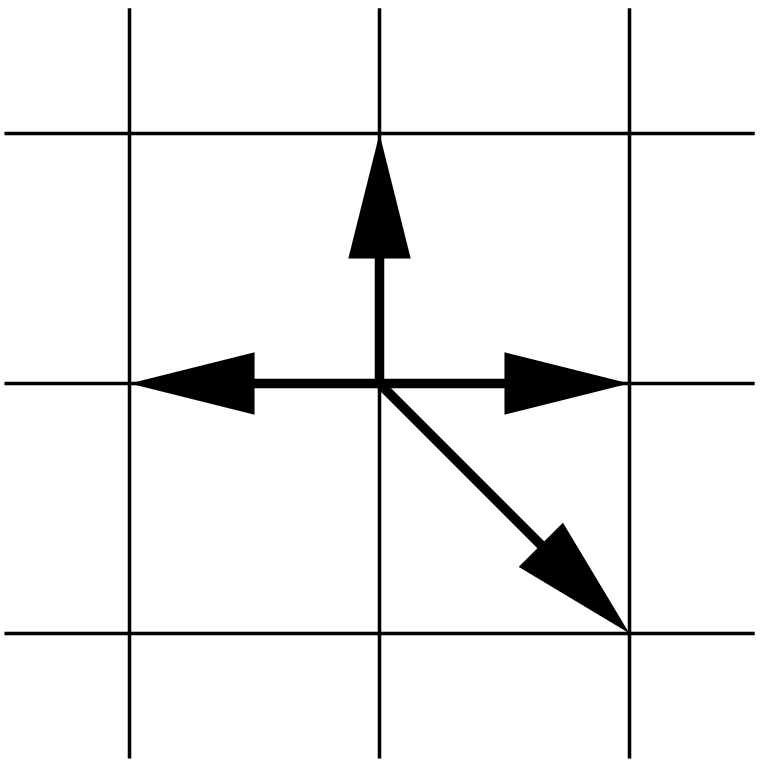}
\hspace{12.5mm}
\includegraphics{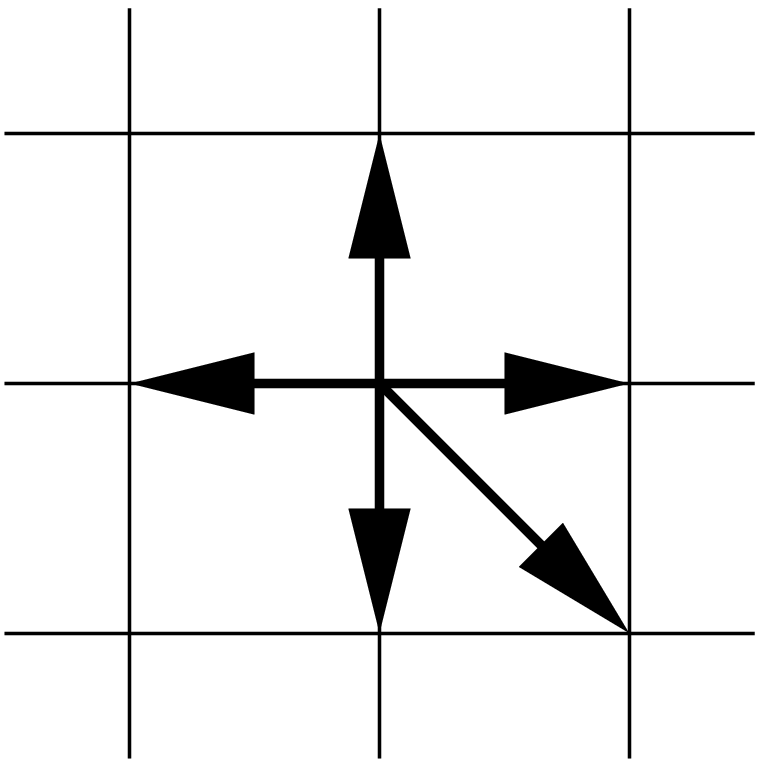}
\hspace{12.5mm}
\includegraphics{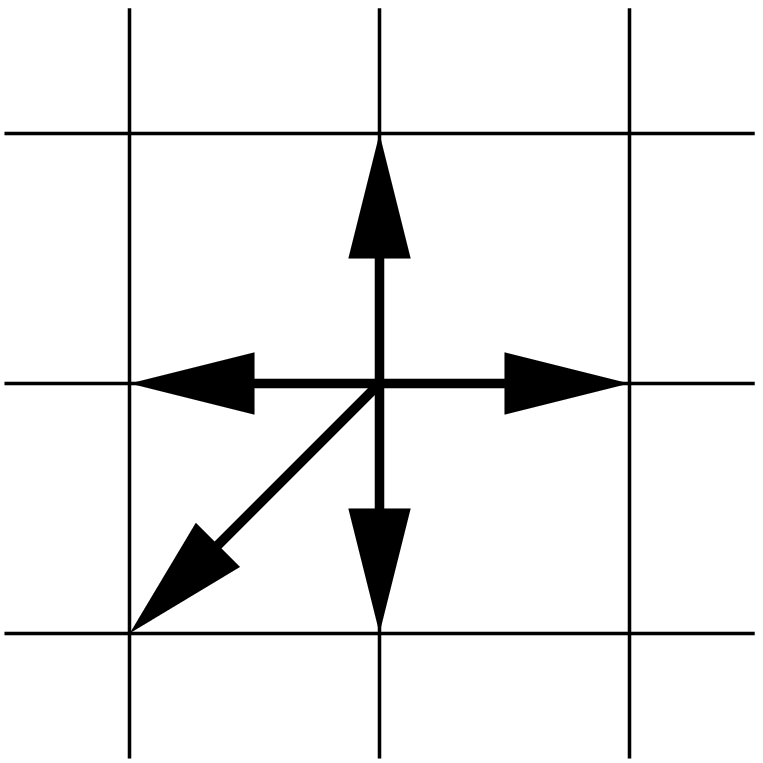}
\hspace{12.5mm}
\includegraphics{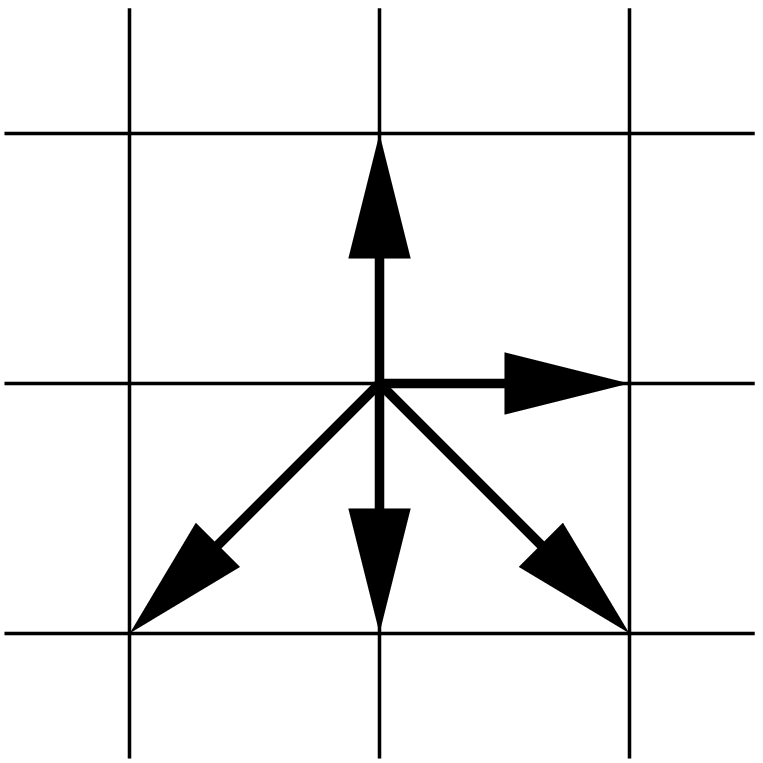}
\hspace{12.5mm}
\includegraphics{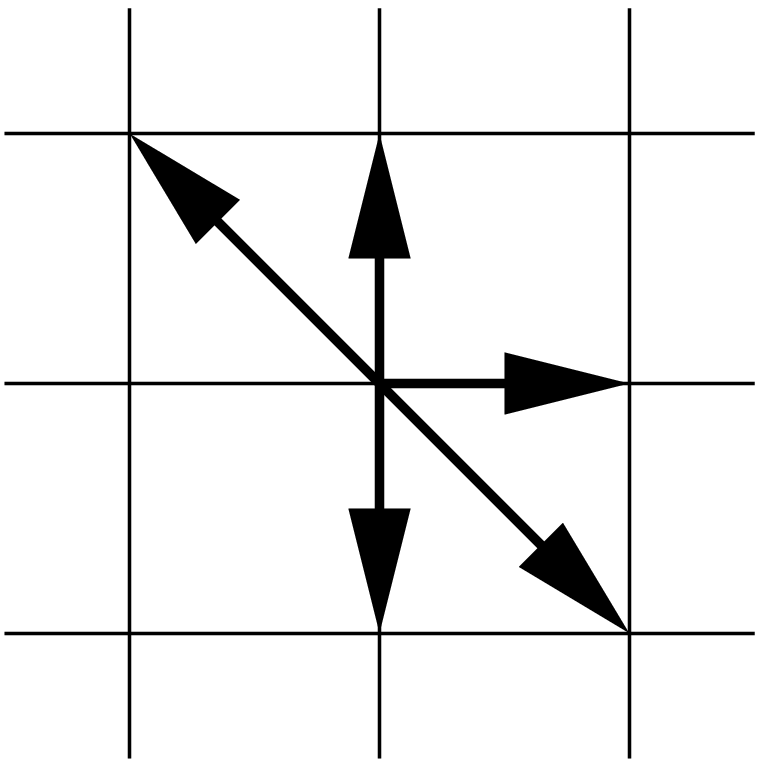}
\hspace{12.5mm}
\includegraphics{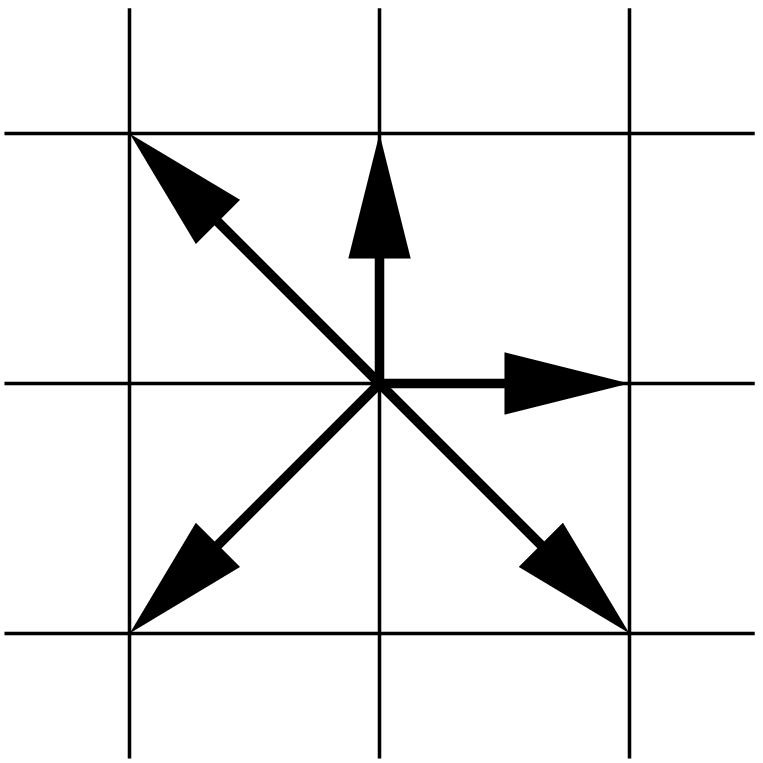}
\hspace{12.5mm}
\includegraphics{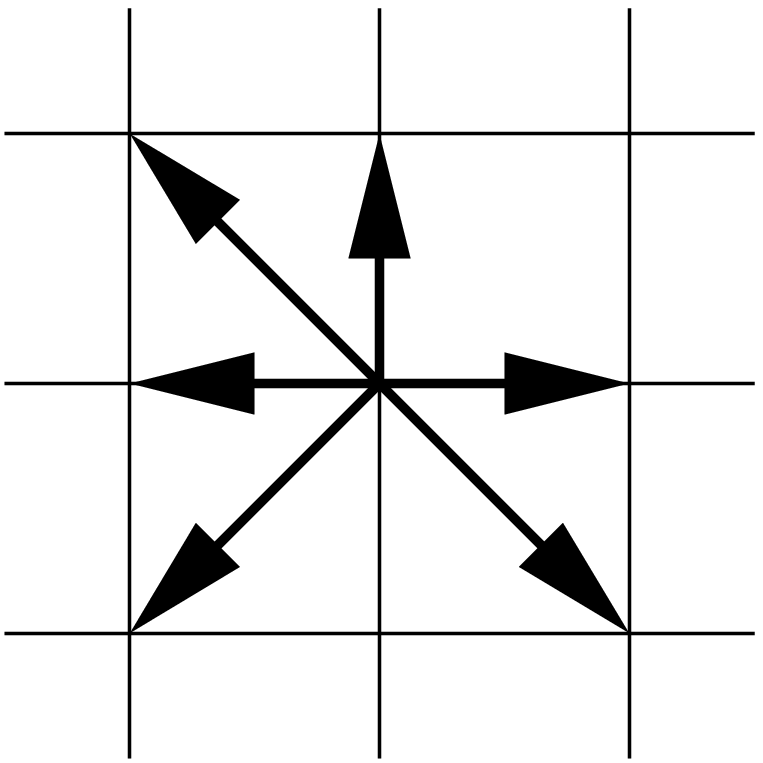}
\hspace{12.5mm}
\includegraphics{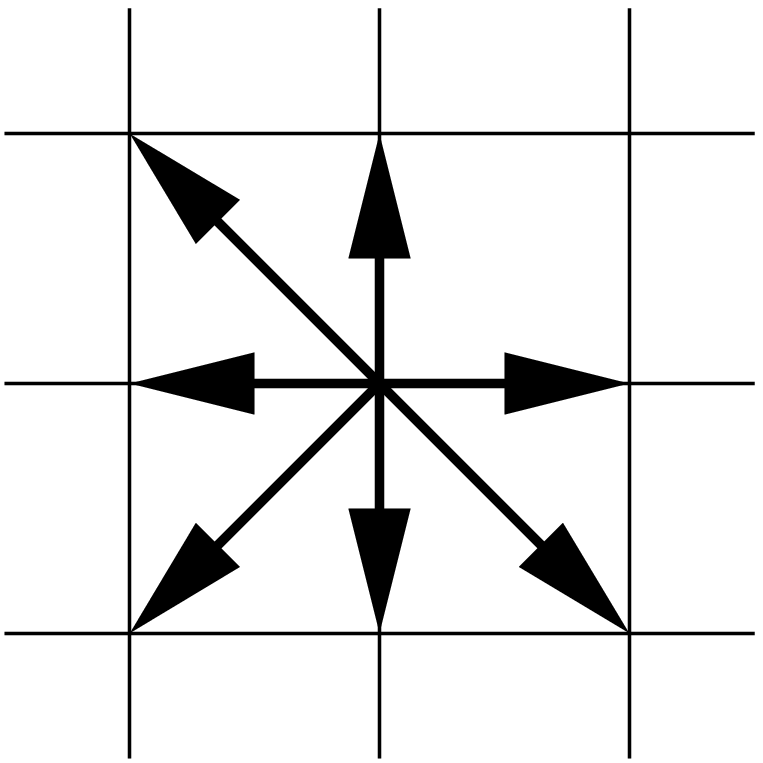}
\end{picture}

(Case II: $y_4>0$, Subcase II.C: $x_4=\infty$ and $Y(x_4)\neq
\infty$ as well as $x_4>0$, $y^\d,y^\dd\neq \infty$)
\begin{picture}(420.00,40.00)
\hspace{20.5mm}
\includegraphics{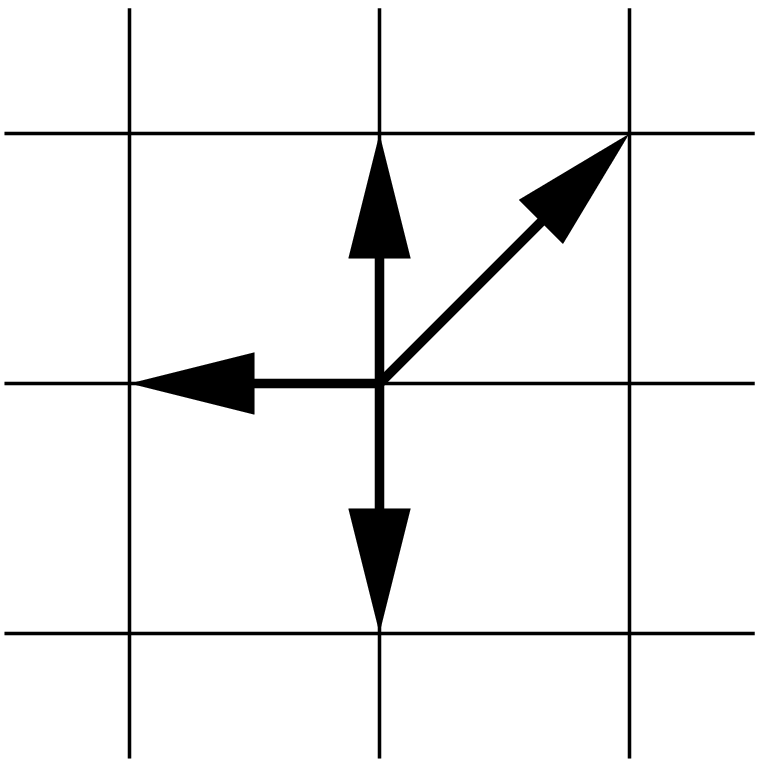}
\hspace{20mm}
\includegraphics{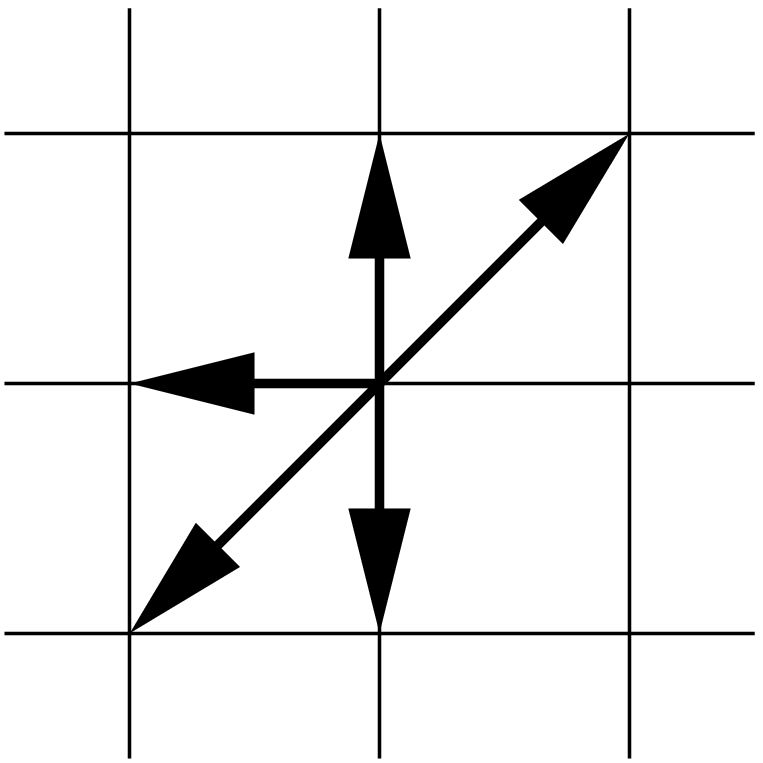}
\hspace{20mm}
\includegraphics{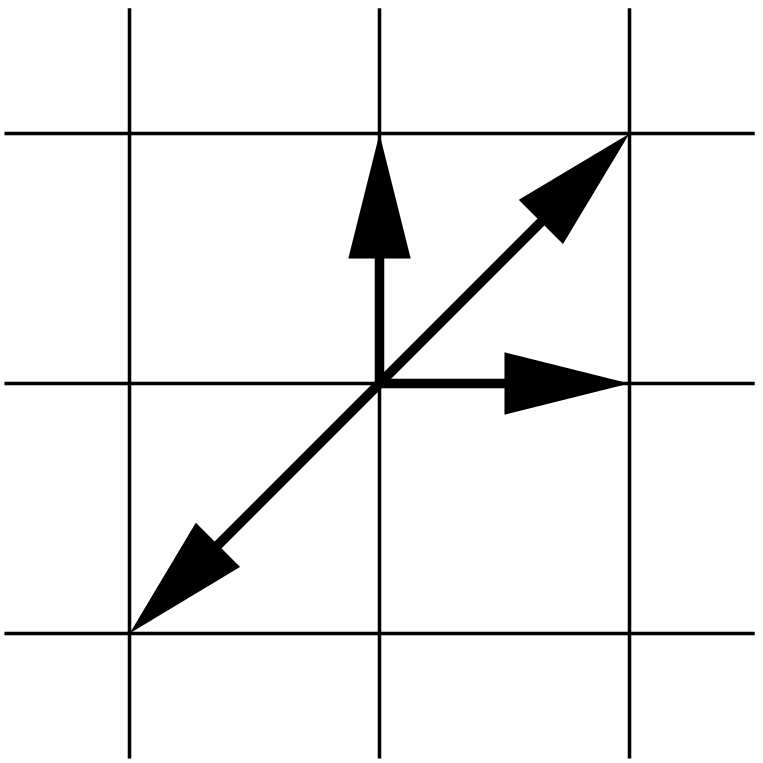}
\hspace{20mm}
\includegraphics{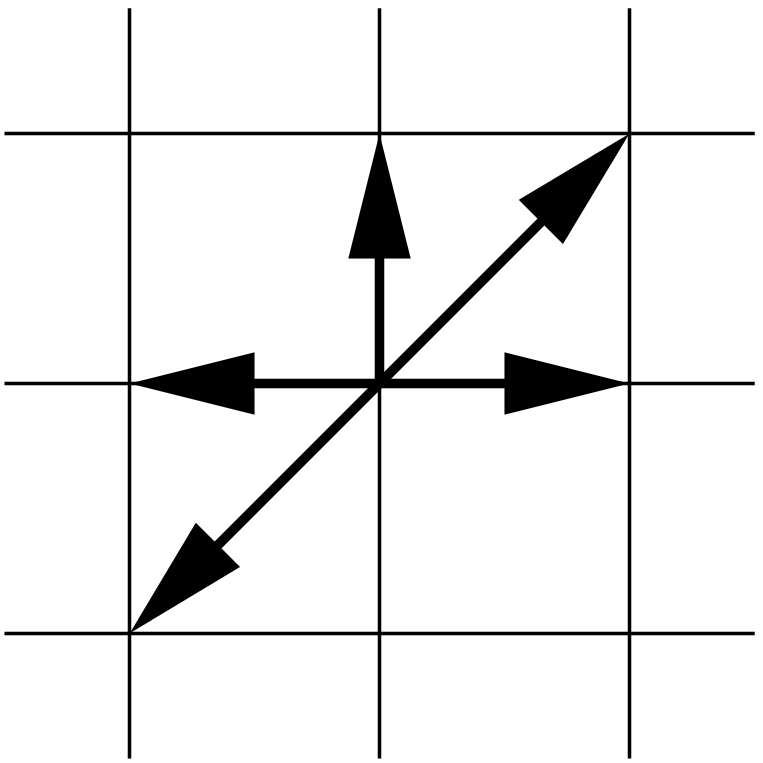}
\hspace{20mm}
\includegraphics{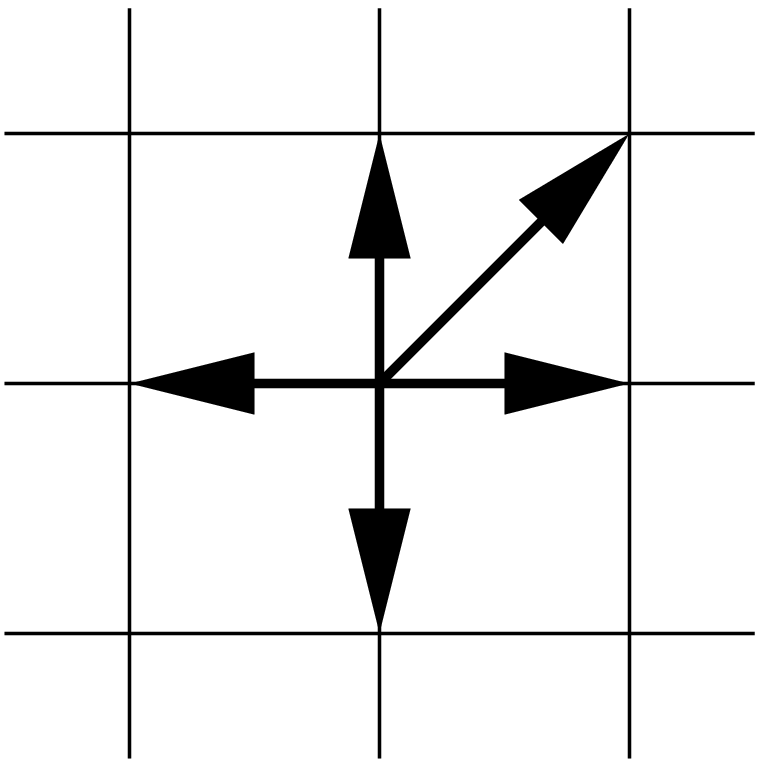}
\end{picture}

(Case II: $y_4>0$, Subcase II.D: $x_4=\infty$ and $Y(x_4)=\infty$)
\begin{picture}(420.00,40.00)
\includegraphics{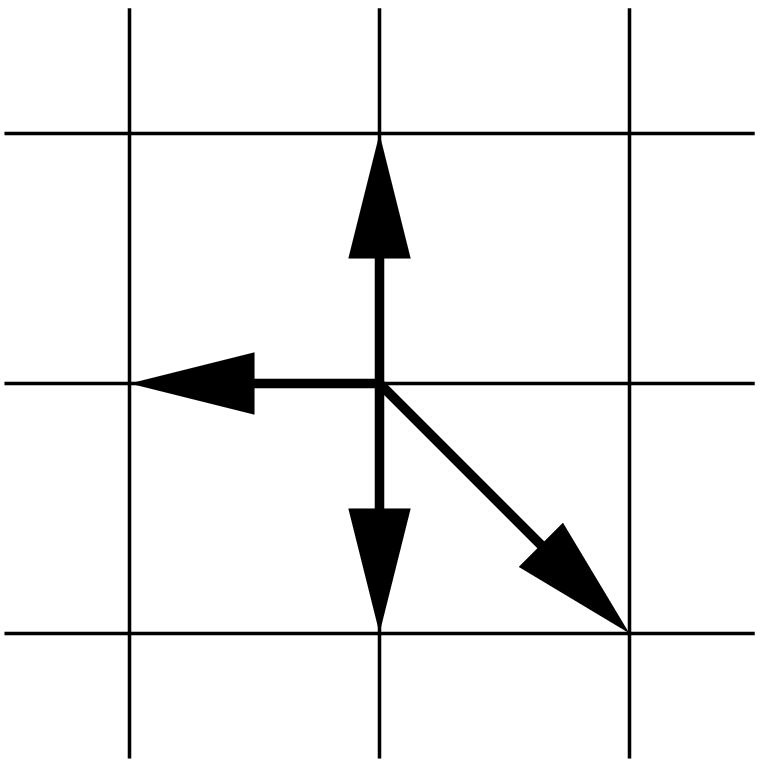}
\hspace{14.3mm}
\includegraphics{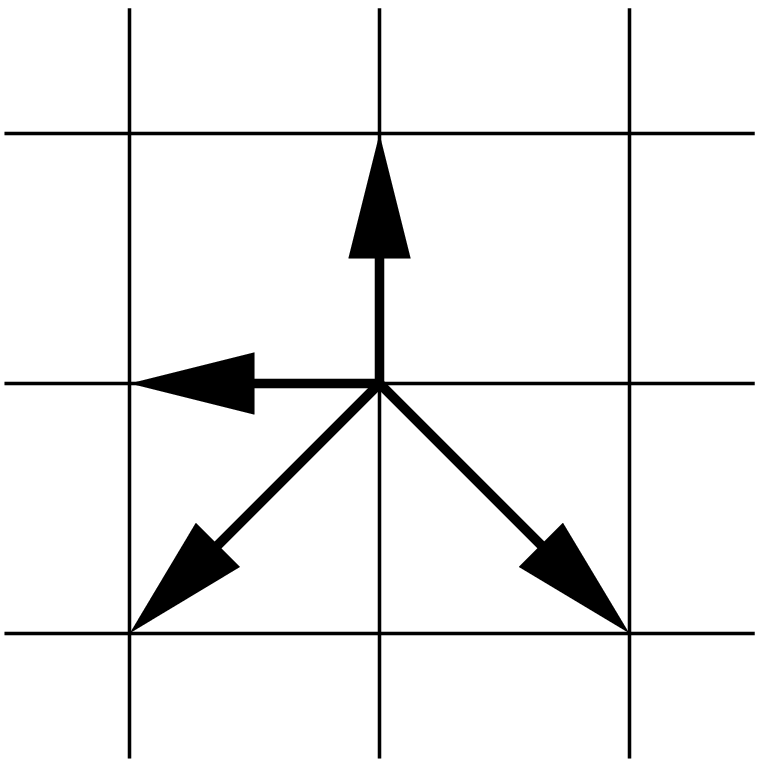}
\hspace{14.3mm}
\includegraphics{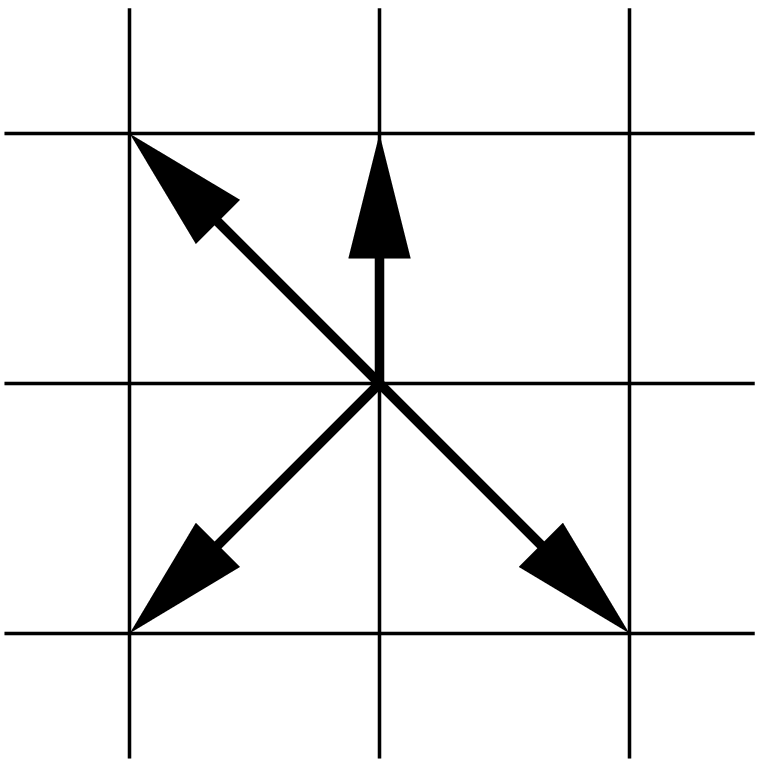}
\hspace{14.3mm}
\includegraphics{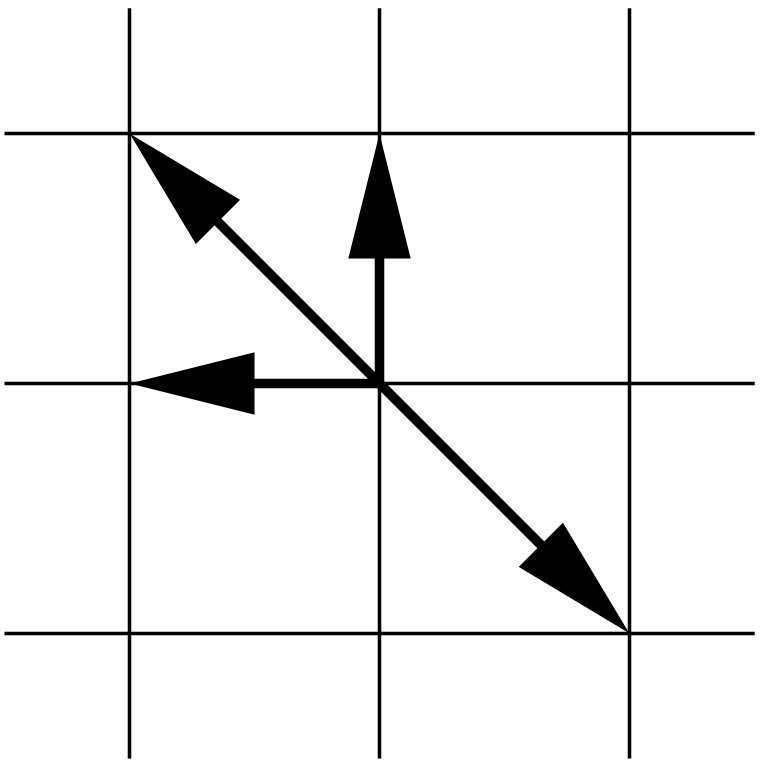}
\hspace{14.3mm}
\includegraphics{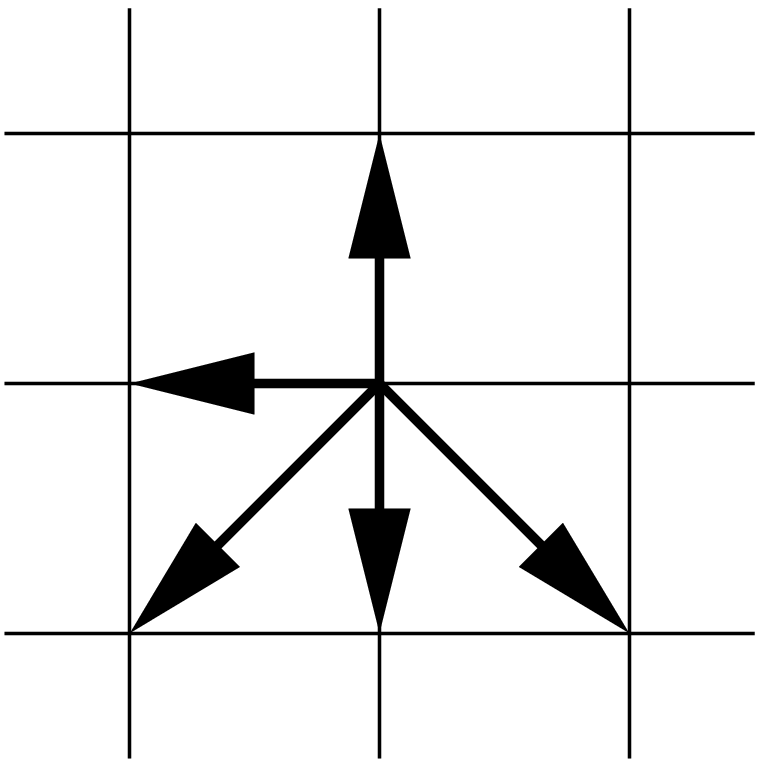}
\hspace{14.3mm}
\includegraphics{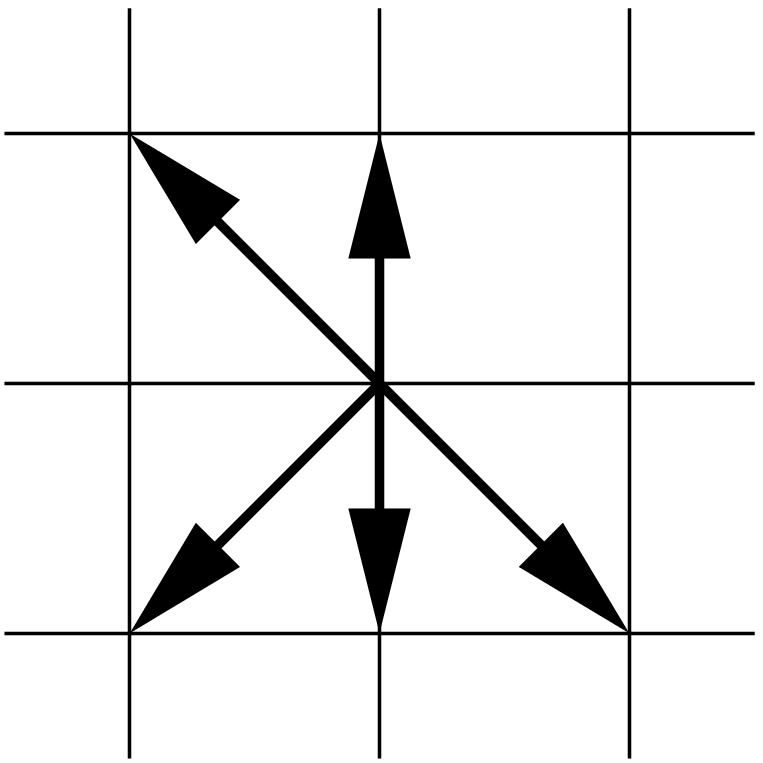}
\hspace{14.3mm}
\includegraphics{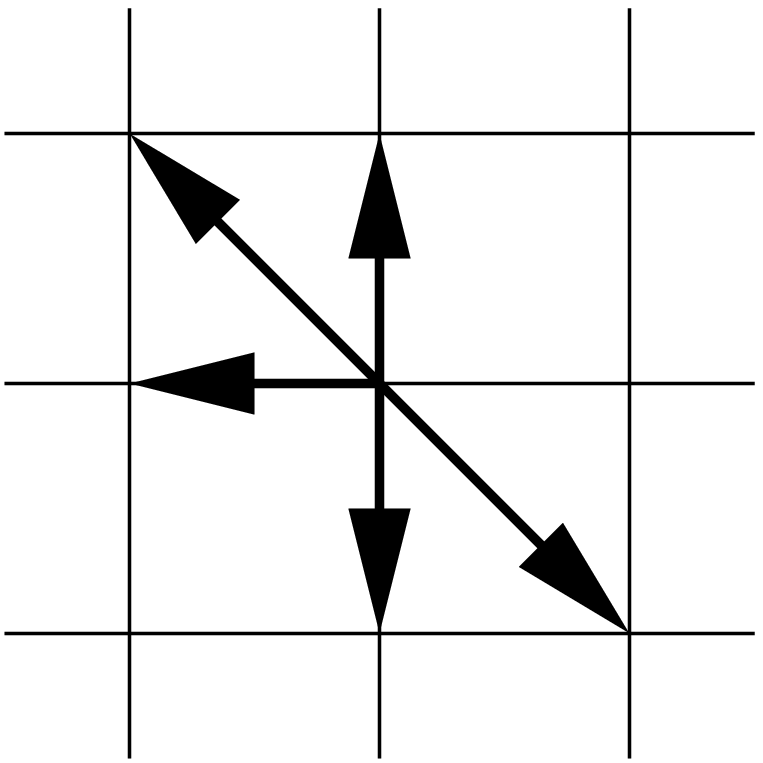}
\hspace{14mm}
\includegraphics{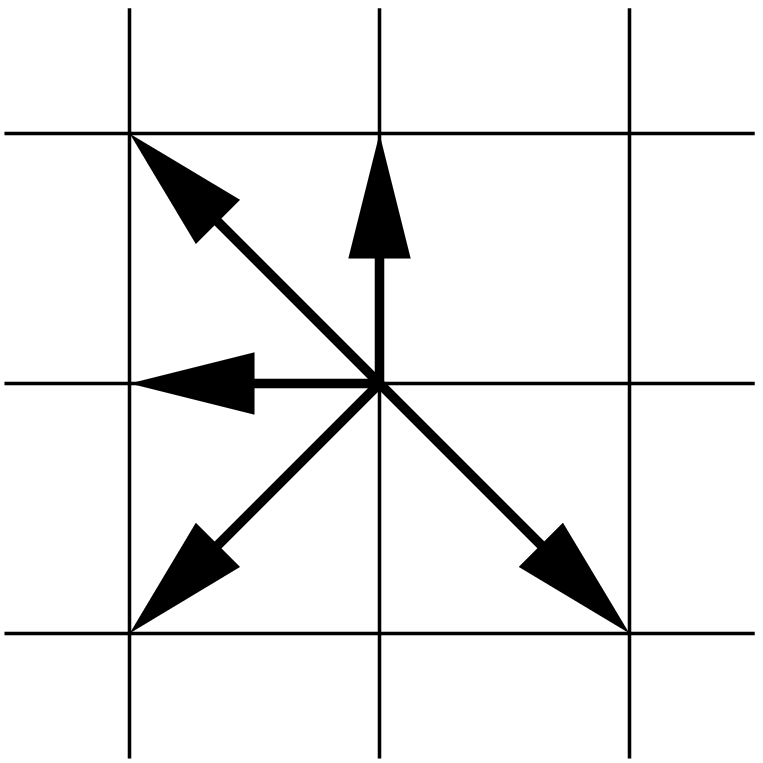}
\hspace{14.3mm}
\includegraphics{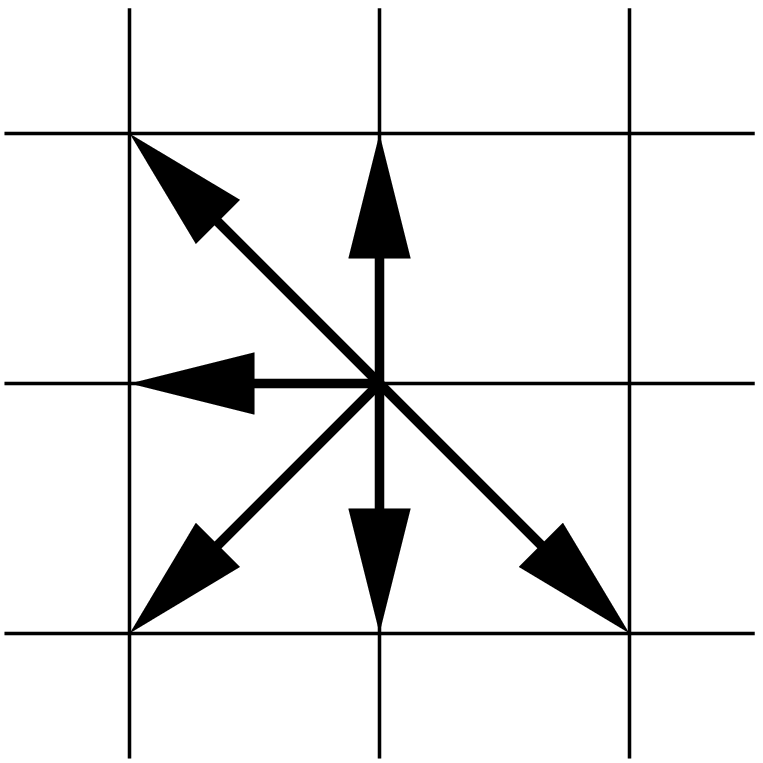}
\end{picture}

(Case III: $y_4=\infty$)
\begin{picture}(420.00,40.00)
\hspace{66mm}
\includegraphics{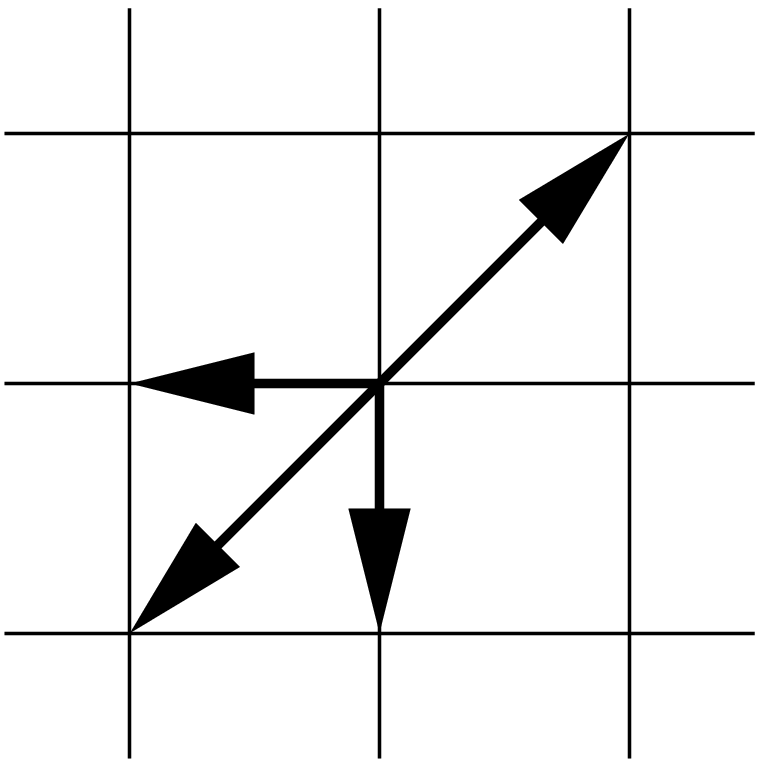}
\end{picture}

\end{center}
\caption{Different cases considered in
the proof of Theorem \ref{main_tt}---they correspond to the $51$ non-singular walks with infinite group,
see \cite{BMM}}
\label{Allcases}
\end{figure}

\section*{Acknowledgments}
K.\ Raschel's work was partially supported by CRC $701$, Spectral
Structures and Topological Methods in Mathematics at the University
of Bielefeld. We are grateful to E.\ Lesigne: his mathematical
knowledge and ideas---that he generously shared with us---have been very helpful in the
 elaboration  of this paper.
    We also warmly thank R.~Krikorian and J.-P.~Thouvenot
 for encouraging discussions.
 Finally, we thank two anonymous referees for their careful reading and their remarks.

\nocite{BCK}

\end{document}